\documentclass[1 [leqno,11pt]{amsart}
\usepackage{amssymb, amsmath,amsmath,latexsym,amssymb,amsfonts,amsbsy, amsthm}
\usepackage{graphicx}

\usepackage{tikz}
\usepackage{tikz-cd}

\numberwithin{equation}{section}

\allowdisplaybreaks

\let\D=\Delta

\def\R{\mathbf R}
\def\Z{\mathbf Z}
\def\N{\mathbf N}

\newcommand{\beq}{\begin{equation}}
\newcommand{\eeq}{\end{equation}}
\newcommand{\ben}{\begin{eqnarray}}
\newcommand{\een}{\end{eqnarray}}
\newcommand{\beno}{\begin{eqnarray*}}
\newcommand{\eeno}{\end{eqnarray*}}

\newtheorem{theorem}{Theorem}[section]

\newtheorem{lemma}[theorem]{Lemma}
\newtheorem{proposition}[theorem]{Proposition}

\newtheorem{remark}[theorem]{Remark}
\newtheorem{Theorem}{Theorem}[section]

\newtheorem{Corollary}[Theorem]{Corollary}

\newtheorem{example}[Theorem]{Example}

\begin{document}

\title[Global Three dimensional  steady  Prandtl Equation]
{Global stability  of  three dimensional  steady \\ Prandtl Equation}

\author{Weiming Shen}
\address{School of Mathematical Sciences, Capital Normal University,  Beijing, China
}
\email{wmshen@aliyun.com}

\author{Yue Wang}
\address{School of Mathematical Sciences, Capital Normal University,  Beijing, China}
\email{yuewang37@aliyun.com}

\author{Tong Yang}
\address{Department of Applied Mathematics, The Hong Kong Polytechnic University,  Hong Kong, China}
\email{t.yang@polyu.edu.hk}

\begin{abstract}
The well-posedness of  three dimensional Prandtl equation is an outstanding open problem
despite of the study in analytic and Gevrey function spaces. This problem is raised as the third open problem in the classical monograph by  Oleinik and Samokhin \cite{Olei}.
The paper aims to address
this  open problem  in the steady case by introducing  a novel approach to establish the  global  stability of  background profile that includes the celebrated Blasius solutions,  which is of particular interest in light of the	 recent affirmation of Prandtl's ansatz  in two dimensional steady setting by Iyer and Masmoudi \cite{IN3}.
In three dimensions, the well-established analytic approaches for two dimensional setting can not be applied because of  the appearance of the secondary flow.
     Rather than employing cancellation mechanisms or  coordinate transforms,
  we introduce new intrinsic vector fields featuring   curvature-type commutators   and establish vector field-based  maximum principles through pointwise and integral estimates  to address the loss of tangential derivatives.

\end{abstract}

\maketitle
\tableofcontents

\date{\today}
\section{Introduction}\label{20250703-intro}
To describe the fluid behavior governed by  the Navier-Stokes equation with no-slip boundary condition in high Reynolds number regime,
Prandtl in \cite{Pr} developed the celebrated boundary layer theory by observing the balance between convection and diffusion effects. In two dimensions (2D), Iyer and Masmoudi \cite{IN3} justified  Prandtl's boundary layer theory globally in the $x$-variable for a large class of
boundary layers including the classical Blasius profiles.
Denoting by $(u,v)$ the tangential component and by $w$ the vertical component of the velocity field,
the  steady boundary layer equation  in three  spatial dimensions (3D) takes the following form
\begin{equation}\label{eq:3DPrandtl}
	\left\{
	\begin{aligned}
		&u \partial_x u +v\partial_{y}u+w\partial_{z}u-\partial_{z}^2u=-\partial_{x}p,\quad in \quad  \Omega,\\&u \partial_x v +v\partial_{y}v+w\partial_{z}v-\partial_{z}^2v=-\partial_{y}p,\quad in \quad  \Omega,\\
		&\partial_xu+\partial_y v+\partial_z w=0,\quad in \quad  \Omega,\\
		&(u,v,w)|_{z=0}=0\quad\mbox{and}\quad \displaystyle\lim_{z\to+\infty} (u,v)=(U,V),\\
		&u|_{x=0,\,y=0}= u_{bd}, \quad v|_{x=0,\,y=0}= v_{bd},
	\end{aligned}
	\right.
\end{equation}where  $$\Omega=(0,X]\times(0,Y]\times(0,+\infty),$$ and Bernoulli's law holds for the outer flow,
\begin{align}\label{1121-3}\begin{split}
		&U\partial_x U+V\partial_y U+\partial_x p=0,\\
		&U\partial_x V+V\partial_y V+\partial_y p=0.
\end{split}\end{align}

Take $(u_B, v_B,w_B)$ to be a steady solution to the 3D Prandtl equation \eqref{eq:3DPrandtl} as follows, \begin{align}\label{FS} (u_B(x,y,z),v_B(x,y,z),w_B(x,y,z))=(u_{s}(\frac{x+y}{2},z),u_{s}(\frac{x+y}{2},z),w_{s}(\frac{x+y}{2},z)),
\end{align} which is a symmetric solution obtained by rotation from a smooth steady solution $(u_s, w_s)$ to the 2D Prandtl equation  \eqref{eq:Prandtl}  with  constant outer flow $\displaystyle\lim_{z\to+\infty} u_s=U$. Here and in the following, symmetry is respect to $x$ and $y$ coordinates. However, it is noted that such symmetry is just for convenience of presentation. The direction of the tangential velocity field $(u_B,v_B)$ in fact can be  arbitrary due to rescaling and  rotation so that the result holds for a 2D profile under the 3D perturbation.
We also assume the background  profile
 satisfies   the following monotonicity and decay conditions:
\begin{align}\label{mnt-0228}
	\partial_z u_B\gtrsim e^{-m_0z^2}\quad\text{for}\quad z>0,
\end{align}
\begin{align}\label{tail-1-0316-0228}
	|\partial_x^{n_1} \partial_y^{n_2} \partial_z^{n_3}u_B|\leq C_{n_1,n_2,n_3} e^{-\frac{2}{3}m_0z^2}\quad\text{as}\quad z\rightarrow+\infty,
\end{align} where $m_0$ is any positive constant, $n_1,n_2,n_3 $ are non-negative integers with $n_1+n_2+n_3>0$
and $C_{n_1,n_2,n_3}$ is a positive constant depending only on $n_1,n_2,n_3.$

We note that the above condition on the background profile is  general in the sense that  a  typical tail of boundary layer profiles mentioned in the classical book of   Oleinik-Samokhin \cite{Olei} (Beginning of Chapter 4) satisfies
\begin{align}\label{tail}
	U-u_s\sim e^{-mz^2} \quad \text{as}\quad z\rightarrow+\infty,
\end{align} where $m$ is a positive constant.

It is important to note that the above assumptions \eqref{mnt-0228}-\eqref{tail-1-0316-0228}  are satisfied by the well-known Blasius solutions which are  self-similar solutions to  the 2D  steady Prandtl equations with $-\partial_xp=0, U=const>0$ and have been experimentally confirmed as a basic flow in the Prandtl theory \cite{Sch00}. That is, consider
\begin{equation}\label{eq:Prandtl}
	\left\{
	\begin{aligned}
		&u \partial_x u +w\partial_{z}u-\partial_{z}^2u=0,\quad in \quad  D,\\
		&\partial_xu+\partial_z w=0,\quad in \quad  D,\\
		&u|_{z=0}=w|_{z=0}=0\quad\mbox{and}\quad \displaystyle\lim_{z\to+\infty} u(x,z)=U,\\
		& u|_{x=0}= u_0,
	\end{aligned}
	\right.
\end{equation}where  $D=(0,X]\times(0,+\infty)$ and Bernoulli's law holds,\begin{align}\label{Blaw}
	U\partial_x U+\partial_x p=0.
\end{align}

The  Blasius solution takes  the form
\begin{align}\label{Blasiusu}
	u_{s}(x,z)=Uf'(\zeta), \quad \zeta:=\frac{ z}{\sqrt{x+x_0}},
\end{align} where the  parameter $x_0$ is a positive constant
and $f$ satisfies $f''>0$ for $0\leq \zeta<+\infty,$ and
\begin{align}\label{df1}
	f'''+ff''=0,\quad f(0)=f'(0)=0.
\end{align}
It is known that
\begin{align}\label{df2}
	1-f'(\zeta)\sim \zeta^{-1}e^{-\frac{\zeta^2}{2}-C\zeta},\quad f''(\zeta)\sim \zeta(1-f') \quad as\quad\zeta\rightarrow +\infty,
\end{align} so that for $x\in[0,X],$
\begin{align}
	\partial_zu_s=\frac{U}{\sqrt{x+x_0}} f''\sim e^{-\frac{\zeta^2}{2}-C\zeta} \quad as\quad\zeta\rightarrow +\infty.
\end{align}

\subsection{Main Result}

We now present the main result in the paper. Without loss of generality, we assume
\begin{align}\label{1121-1-1}
	U=V=const
\end{align}
so that the Bernoulli's law \eqref{1121-3} implies
\begin{align}\label{1121-2}
	(\partial_xp,\partial_yp)=(0,0).
\end{align}
The stability  of the  profile  \eqref{FS} to \eqref{eq:3DPrandtl} is  stated as follows.

\begin{theorem}\label{1121-250628} Let $X,Y$ be any positive constants and $(u_B, v_B,w_B)$ defined in \eqref{FS} satisfies  \eqref{mnt-0228} and \eqref{tail-1-0316-0228}.
	Assume the boundary data $(u_{bd},v_{bd})$ in \eqref{eq:3DPrandtl} with \eqref{1121-1-1}-\eqref{1121-2}  satisfy the smooth compatibility conditions and the following growth rate assumptions. For any positive constant $\tilde{\varepsilon}$ with $\tilde{\varepsilon}\leq \tilde{\varepsilon}_0$ and any fixed  constant $l>1 $, we assume
	\begin{align}\label{bddata-250628-1}\begin{split}
			&\|(\partial_zu_{bd},\partial_zv_{bd}) -(\partial_zu_B,\partial_zv_B)\|_{W_{x,y}^{2,\infty}(\{x=0\}\cup\{y=0\}\cup\{z=0\}\cap\partial \Omega)}\\ \leq& \tilde{\varepsilon}^l [z\chi +e^{-m_0 z^2}(1-\chi)],	 \end{split}
	\end{align}  where $\tilde{\varepsilon}_0$ is a constant depending only on $X,$ $Y$ and $\chi(z)$ is the standard cutoff function with $\chi=1$ on $[0,1]$ and $\chi=0$ on $[2,+\infty]$.
	Then  there exists a unique solution to \eqref{eq:3DPrandtl} with \eqref{1121-1}-\eqref{1121-2}
 satisfying the following properties,
	\begin{align}\label{zt-250628}\begin{split}
			\|(u,v) -(u_B,v_B)\|_{W_{x,y,z}^{2,\infty}(\Omega)}\leq \tilde{ \varepsilon}.
		\end{split}
	\end{align}
	The  detailed growth rates
	in $z$ near $z=0$ and decay rates in $z$ for  large $z$  can be found in the proof.
\end{theorem}

 Theorem \ref{1121-250628}
firstly established the global  stability of steady solutions in function spaces with finite-order derivatives.
In light of the potential instability resulting from the three-dimensional phenomenon of secondary flow, the powers of $z$ near $z=0$ for the derivatives of fluid  variables are needed to  address  some critical indices in the proof with more details given in Theorem \ref{1121}.

 Although \eqref{eq:3DPrandtl}  is no longer a parabolic equation, the boundary data on $\{x=0\}\cup\{y=0\}\cup\{z=0\}\cap\partial \Omega$ in \eqref{bddata-250628} can still be obtained   from the prescribed data in \eqref{eq:3DPrandtl} by compatibility condition as follows.
By $u=0$ at $z=0,$ we have
$|\frac{\int_0^z\partial_x udz'}{u}|\leq Cz$ so that
\begin{align}\label{intuu}
	\frac{\int_0^z\partial_x udz'}{u}=0\quad \text{at}\quad z=0.
\end{align}
By \eqref{eq:3DPrandtl},  at $x=0,$ we have
\begin{align*}
	u^2\partial_{z}\frac{\int_0^z\partial_x udz'}{u} =&u \partial_x u -\int_0^z\partial_x udz'\partial_z u\\
	=&-v\partial_{y}u+\int_0^z\partial_y vdz'\partial_{z}u+\partial_{z}^2u
	:=RHS.
\end{align*}
Then by \eqref{intuu}, at $x=0$, we have
\begin{align}\label{code}
	\partial_x u =\partial_{z}(u\int_0^z \frac{ RHS}{u^2}dz').
\end{align}
Similarly, we can obtain the boundary data of $\partial_y v$ on $y=0.$

The smoothness requirement is used to ensure the existence of a smooth solution of the approximate quasi-linear equations. This is not the main concern of the paper. 
In fact, the uniform estimates guarantee the existence of strong solutions. 

\subsection{Methodologies of the Proof }
It is well known that the essential feature of the three dimensional problem is the appearance of the secondary flow. Hence in analysis,
in addition to the difficulty induced by the degeneracy and non-local terms as in two dimensional case,
the challenge  in three dimensional case is how to estimate the flow component perpendicular to the primary flow. In the setting of this paper, this  leads to estimate  the difference between $v$ and $u$ in the two tangential directions. Precisely, due to the loss of symmetry,  we need to introduce some new analytic techniques which are different from the existing ones for two space dimensions to establish the uniform estimates on  $$q=v-u, \quad \int_0^z q dz'$$
 and their derivatives.

First of all, we note that the well established analytic techniques such as coordinate transformation (von Mises transformation for steady flow) and the cancellation mechanisms through convection terms and vorticity can not be applied in the three space dimensions. In addition, it is well known that there is no maximum principle for system of equations unlike the scalar case which correspond to the two space dimensional problem studied in Oleinik's classical work. To overcome these difficulties, we apply the following new approach that can be also applied to other boundary layer problems in three space dimensions.

To prove the stability estimate, we will apply an induction argument. Based on the induction assumption and the boundary condition, the proof on the $n$-th approximation in \eqref{appeqeuclidean} relies on a bootstrap argument on the estimates given in \eqref{assumpn-2}.

In order to close the bootstrap argument, we introduce
a set of vector fields to replace the coordinate transformation and cancellation mechanism techniques used in two space dimensions. For this,  the analysis of the commutators of the vector fields plays an essential role.  Let us illustrate the vector fields as follows. Since  $w=-\int_0^z\,\partial_xu+\partial_y v\,dz'$ by \eqref{eq:3DPrandtl}, we rewrite \eqref{eq:3DPrandtl} as
\begin{align}\label{2.1}\begin{split}
		&u \partial_x u +(u+q)\partial_{y}u-\int_0^z\partial_xu+\partial_y vdz'\partial_{z}u-\partial_{z}^2u=0,\\   &u \partial_x v +(u+q)\partial_{y}v-\int_0^z\partial_xu+\partial_y vdz'\partial_{z}v-\partial_{z}^2v=0.\end{split}
\end{align} 
To absorb the non-local terms, we introduce the following operator  with given $(u,v)$
\begin{align}\label{p1uu}
	P^{u,v}_1w=&\nabla_\xi w +(1+\tilde{q})\nabla_{\eta}w -\nabla_{\psi}(u\nabla_{\psi}w),\end{align}
where $\tilde{q}=\frac{q}{u}$ and
\begin{align}\label{vf}\nabla _{\xi}=\partial_x-\frac{\int_0^z \partial_{x}u dz'}{u }\partial_z,\quad
	\nabla _{\eta}=\partial_y-\frac{\int_0^z \partial_yv dz'}{v }\partial_z,
	\quad\nabla _{\psi}=\frac{1}{u }\partial_z.\end{align}
Note that the newly introduced differential operators are not commutative, that is,
$$[\nabla _{\eta},\nabla _{\psi}]\neq 0,\quad [\nabla _{\eta},\nabla _{\xi}]\neq 0.$$
Therefore, we can not employ a coordinate transformation to  view the vector fields in \eqref{vf} as partial differential derivatives with respect to the new variables. This  is essentially  different
from the von Mises transformation in two dimensional case.
Then  the  question is how to obtain uniform estimates on the solution by  using these vector fields and how to reduce them  to uniform estimates in  the Euclidean coordinates.

When we estimate the derivatives of $u$, we can not
commute the vector fields as directly as  commuting ordinary partial differential derivatives. In fact, the  commutator of the ``tangential" vector fields generates the
curvature-type quantity defined by
\begin{align*}
	K\partial_z= [\nabla_\xi,\nabla_\eta].
\end{align*}
Note that $K$ also  measures the symmetry breaking related to the secondary flow which is
the essential feature in  the three dimensional flow compared to the  two  dimensional case.
If $(u,v,w)$ is symmetric as defined in \eqref{FS}, then $K=0$.  One main ingredient of the proof is to estimate $K$ and its derivatives through a surprising representation of $K$ given in \eqref{yq} with detailed calculation presented in the Appendix.

Another trade-off for using these vector fields is to deal with the blow-up rates.
To be specific, there is a multiplier $\frac{1}{u}$ in the expression of $\nabla_\xi,$ $\nabla_\eta$ while we note $u=0$ on $z=0.$ For example, the tangential vector fields related to $(u_B,v_B)$ is $$\nabla _{\tau}=\partial_x-\frac{\int_0^z \partial_{x}u_B dz'}{u_B }\partial_z.$$ Then for any  integer $i\geq 2,$
\begin{align}\label{26-04-26-2}
   \partial_{z}^i(\nabla _{\tau}z)=c_i\frac{(\partial_z u_B )^i}{u_B^{i+1} }\int_0^z \partial_{x}u_B dz'+\cdots\sim\frac{1}{z^{i-1} }+\cdots\quad \text{near}\quad z=0,
\end{align} where $c_i$ is a constant. Hence, we will 
construct some suitable barrier functions
to cope with the vector fields and the background boundary layer behavior.



Unlike the steady Prandtl equations in two space dimensions, one can apply von Mises transformation to reduce the 2D system to a degenerate parabolic equation so that maximum principle can be applied. For three space dimensions, there are two coupled equations for two tangential velocity components. As well known, there is no maximum principle for system of equations in general. However, with the vector fields, we can consider the  operator $P^{u,v}_1$ defined in \eqref{p1uu} 
and
\begin{align}\label{1121-6}
	P^{u,v}_2 w=&\nabla_\xi w +(1+\tilde{q})\nabla_{\eta}w -u\nabla_{\psi}^2 w.\end{align}
For these operators, we will establish   maximum principles  both in bounded and unbounded domains. Moreover, we  generalize these maximum principles for barrier functions with ridges in order to take care of the inner and outer layers inside the Prandtl boundary layer which suitably cope with the blow-up rates as mentioned in \eqref{26-04-26-2}.

Different from  the approximate solutions and their first order tangential vector field derivatives which can be estimated via  maximum principles established through pointwise estimates, the estimation of the second order tangential vector field derivatives is very subtle as shown in the proof of Theorem \ref{5.1}. A different maximum principle is established through integral estimates which crucially depends on determining the sign of the term associated with extra loss of derivatives induced by tangential symmetry breaking. Here, we use a toy model to present the essential ideas.

Denote by $f$ the difference between the second order tangential vector field derivatives of the approximate solution and the corresponding second order derivatives of the background profile plus suitable barrier function. Note that  by the representation of $K$, both $K$ and its normal derivative $\partial_z K$ have precise bounds based on the a priori assumptions.

Then $f$ satisfies an equation like
\begin{equation}
	\partial_x f+\partial_y f=-\partial_y K + R,
\end{equation}
which is a simplified version of \eqref{1104-1}. Here, $R$ represents the error term that can be controlled while $\partial_y K$ contains \textit{loss of tangential derivative which is one order higher than a priori assumptions.} Hence, instead of bounding $\partial_y K$, we construct new auxiliary function $F$ in the proof  and take advantage of the sign of $\partial_y F_+$
to overcome this difficulty as follows.

In order to prove $f\le 0$ up to $y=Y$, we can apply proof by contradiction. If it holds up to $y=y_*<Y$, then for some constant $y_2\in (y_*,Y)$, we introduce an auxiliary function $F$ defined by
$$
F=f+\gamma (y-y_*)\phi_{1,\frac{\alpha}{2}}-B+B\zeta(y) \quad\text{in}\quad [0,X]\times[y_*,y_2]\times[0,+\infty),
$$
where $\gamma>0$ is an arbitrarily small constant, $B$ is a controllable positive function used to bound $K$, $\phi_{1,\frac{\alpha}{2}}$ is defined in \eqref{phi1}  and $\zeta(y)$ is defined in \eqref{zeta}. The goal is to show that
\begin{equation}\label{toy_1}
\partial_x (\int_{y_*}^{y_2} \int_0^\infty F_+^2dzdy)\lesssim \int_{y_*}^{y_2} \int_0^\infty F_+^2dzdy \quad\text{in}\quad [0,X],
\end{equation}
so that Gronwall inequality implies $F_+=\max\{F,0\} \equiv 0$ by boundary conditions. Hence, this  implies $f\le 0$ holds beyond $y=y_*$. To achieve this, we will show that there exists $y_1\in (y_*,y_2)$ such that
$$
\partial_y F_+\ge 0, \, y\in (y_*,y_1);\quad \partial_y F_+\le 0, \, y\in (y_1,y_2),
$$ through a finite covering  argument.
By suitably choosing $B,$ $y_1$ and $y_2,$ we can prove \eqref{toy_1} as shown in the proof of Theorem \ref{5.1}.

By the methodologies above, the solution will be constructed through the following iteration scheme. Firstly, for an arbitrarily small positive  constant $\epsilon_0,$ set
$$u_0=v_0= \bar{u},\quad q_0=0,$$
where $ \bar{u}(x,y,z)=u_B(x,y,z+{\epsilon_0})$ and consider the $n$-th approximation
\begin{align}\label{appeqeuclidean}\begin{split}
		0= &\partial_x u_n +(1+\frac{q_{n-1}}{u_{n-1}})\partial_{y}u_n -(\frac{\int_0^z\partial_xu_{n-1}dz'}{u_{n-1}}
		+(1+\frac{q_{n-1}}{u_{n-1}})\frac{\int_0^z\partial_y v_{n-1}dz'}{v_{n-1}})\partial_{z}u_n \\&-\frac{1}{u_{n-1}}\partial_{z}(\frac{u_{n}}{u_{n-1}}\partial_{z}u_n)\quad \text{in} \quad \Omega,\\  0= &\partial_x v_n +(1+\frac{q_{n-1}}{u_{n-1}})\partial_{y}v_n -(\frac{\int_0^z\partial_xu_{n-1}dz'}{u_{n-1}}
		+(1+\frac{q_{n-1}}{u_{n-1}})\frac{\int_0^z\partial_y v_{n-1}dz'}{v_{n-1}})\partial_{z}v_n \\&-\frac{1}{u_{n-1}}\partial_{z}(\frac{v_{n}}{v_{n-1}}\partial_{z}v_n) \quad \text{in} \quad \Omega,
	\end{split}
\end{align} where $v_{n}=u_{n}+q_{n}$ and
$$\Omega=(0,X]\times(0,Y]\times(0,+\infty).$$
For the boundary data, we assume
$\lim_{z\to+\infty} (u_n,v_n)=(1,1)$ and the approximate non-degenerate boundary data on $\{z=0\}\cup\{x=0\}\cup\{y=0\}$ relying on the parameter $\epsilon_0$.
The detailed definition of the approximate boundary data will be  given in Subsection \ref{appbddata} in the Appendix with suitable growth rates.

For the above system, the main goal is to obtain  uniform estimates independent of ${\epsilon_0}, n$ on the following functions
\begin{align*}
	\partial_x^2 u_n,\, \partial_y \partial_x u_n, \, \partial_z \partial_x u_n, \,\partial_z \partial_yu_n, \, \partial_y^2 u_n,\, \partial_z^2 u_n,\\
	\partial_x^2 v_n,\, \partial_y \partial_x v_n, \, \partial_z \partial_x v_n, \,\partial_z \partial_yv_n, \, \partial_y^2 v_n,\, \partial_z^2 v_n.
\end{align*}
Then by compactness, for any ${\epsilon_0}$, there exists a solution $(u_{\epsilon_0},v_{\epsilon_0})$ of \eqref{2.1} with approximate boundary data. Finally, by letting ${\epsilon_0}$ go to $0,$ we obtain the solution $(u,v)$ of \eqref{eq:3DPrandtl}.

\subsection{Existing literature}

For two-dimensional steady flow, Oleinik-Samokhin in the  classical book \cite{Olei}  proved the  existence and uniqueness  of strong solutions by using von Mises transformation and the maximum principle. In the case of favorable pressure gradient, global-in-$x$ existence of  solutions was also obtained in \cite{Olei}. For the regularity, \cite{GI1} established higher order regularity and then \cite{YWZ} established the global $C^\infty$ regularity. For stability, we mention the pioneer work
\cite{Serrin} and the refined asymptotic analysis in \cite{I1}, \cite{asyBgeneral}, \cite{Iy2024} and \cite{JLY} for structural stability. In addition, Guo, the second author and Zhang \cite{GWZ24} established dynamic stability of steady solutions.
In
the case of adverse pressure gradient, the physical phenomenon of boundary layer separation is justified  in \cite{DM} and \cite{SWZ}. For validity, \cite{GM18} and \cite{GI1, GI2} firstly justified the  Prandtl layer expansion in a  small interval for shear flows and  Blasius-like profiles respectively,  cf.  also
\cite{GZ20}.
Recently, \cite{IN3} justified the  Prandtl  boundary layer expansion globally-in-$x$ for a large class of steady solutions that include the
Blasius profiles.
We also mention \cite{GN1, I2,I3,I4,I5,I6} for validity study with the moving boundary assumption.

For two-dimensional  unsteady flows, under the monotonicity condition,
the well-posedness was established by using Crocco transformation  in \cite{Olei,XZ}
and later cancellation mechanisms were observed by two research groups independently \cite{AWXY, MW}. Without monotonicity assumption,  the well-posedness  was  established in  analytic  and Gevrey function spaces,  cf. \cite{DG, GM, IV, LY, LCS, WWZ2024}, etc. In Sobolev spaces, Prandtl equations are in general ill-posed without monotonicity assumption, cf.  \cite{GD} and \cite{GerN}. For finite time blowup with singularity analysis and boundary layer separation in the unsteady setting, one can refer to  \cite{CGIM, CGM, EE, KVW, WZ}.  The validity of Prandtl expansion was established  in the
analytic  and  Gevrey function spaces \cite{GMM, SC1, SC2} with different assumptions on the background profile. One can also refer   to \cite{Ka84, Ma, WWZ17} and the references therein.  In the Sobolev setting, the invalidity  of Prandtl expansion  was studied in \cite{GGN1, GGN2, GreN, GN}.

In three space dimensions, the stability  of boundary layer equations is very challenging  because  the secondary
flows
appear in the boundary layer
as explained in \cite{Olei}.   In fact, the well-posedness of Prandtl equation was raised as the third open question in \cite{Olei}. Note that the
 tangential velocity field is curved in general. Therefore, the cancellation mechanisms observed in
 two dimensional case no longer hold and this is why there are no counterparts of the
Crocco or von Mises   transformations in the three dimensional setting.
For this,  the ill-posedness of the three dimensional system was studied  in \cite{LWY24} about
 perturbation of shear flows when the initial data satisfy $U(z)\not\equiv cV(z)$ with a constant $c$.  The well-posedness in Gevrey function spaces with index $2$ for the three-dimensional Prandtl system without any
structural assumption was established by Li, Masmoudi and the third author \cite{LY24}.
One can also refer to
\cite{LWY24-1} and \cite{FTZ2024} for the well-posedness and zero-viscosity limit results in three dimensions respectively.
Compared to the fruitful mathematical theories in the  two dimensional case,
much less is known about the well-posedness theories for both steady and unsteady flows in function spaces with finite order of derivatives. And the result of this paper aims to be the first step to fill in this gap in the steady setting.

The rest of the paper is organized as follows. The main result Theorem \ref{1121-250628} will be reduced into a form suitable for our methodology in Section \ref{2025-0630-section}. We will introduce the new vector fields with the representation of  commutators and the bootstrap argument in terms of these vector fields in Section \ref{s2}.  The vector fields-based ``Maximum Principles"  in both bounded and unbounded domains will be given  in Section \ref{s3}.  We will  establish in Section \ref{s4} and Section \ref{s5} the a priori estimates for the existence of solutions. In the Appendix, we will give some detailed calculations and  estimates used in the proof of the main result, also the corresponding estimates of derivatives in the Euclidean variables.

\section{Reduction of Theorem \ref{1121-250628}}\label{2025-0630-section}

 In this section, we show that Theorem \ref{1121} yields
   Theorem \ref{1121-250628}
through the  self-similar change of  variables as follows and Lemma \ref{2025-0630-lemma}. We  prove Theorem \ref{1121} in this paper.

\subsection{Self-similar change of  variables}\label{26-03-08}
 Through the self-similar change of  variables  \eqref{26-02-08-sscg} in subsection \ref{ssv}  in the Appendix with $\theta$ and $\mu$  being small positive constants independent of $\tilde{\varepsilon}$,  the assumptions in  Theorem \ref{1121-250628} are transformed into the following case:
 \begin{align}\label{26-02-08-xtheta}
 X\leq \theta,
 \end{align}
 and  \begin{align}\label{mnt}
	\partial_z u_B\gtrsim e^{-\frac{3}{2}\mu z^2} \quad\text{for}\quad z>0,
\end{align} \begin{align}\label{tail-1-0316}
	|\partial_x^{n_1} \partial_y^{n_2} \partial_z^{n_3}u_B|\leq C_{n_1,n_2,n_3} e^{-\mu \frac{z^{2}}{x+1} }\quad\text{as}\quad z\rightarrow+\infty,
\end{align}
\begin{align}\label{bddata-250628}\begin{split}
			&\|(\partial_zu_{bd},\partial_zv_{bd}) -(\partial_zu_B,\partial_zv_B)\|_{W_{x,y}^{2,\infty}(\{x=0\}\cup\{y=0\}\cup\{z=0\}\cap\partial \Omega)}\\ \leq& \tilde{\varepsilon}^l [z\chi +e^{-\frac{3}{2} \mu z^2}(1-\chi)]. \end{split}
	\end{align}

By the self-similar change of  variables above, we prove the following theorem and
 assume \eqref{26-02-08-xtheta}, \eqref{mnt}, \eqref{tail-1-0316} and  \begin{align}\label{1121-1}
 (U,V)=(1,1)
\end{align} throughout the paper
without loss of generality.

 \begin{theorem}\label{1121} Let $(u_B, v_B,w_B)$ be defined in \eqref{FS} satisfy \eqref{mnt}-\eqref{tail-1-0316}.
	Assume the boundary data $(u_{bd},v_{bd})$ in \eqref{eq:3DPrandtl} with \eqref{1121-1}-\eqref{1121-2}  satisfy the smooth compatibility conditions and the following growth rate assumptions. For a sufficiently small positive constant $\varepsilon$, assume
	\begin{align}\label{bddata}\begin{split}
			&|u_{bd} -u_B|+|v_{bd} -v_B|\leq\varepsilon^8 \Phi_2,\\ &
			|\partial_zu_{bd} -\partial_zu_B|+|\partial_zv_{bd} -\partial_zv_B|\leq \varepsilon^7 \Phi_2 ,\\&
			|\partial_{x,y}u_{bd} -\partial_{x}u_B|+|\partial_{x,y}v_{bd} -\partial_{x}v_B|\leq \varepsilon^6 \Phi_2,
			\\&
			|\partial_z\partial_{x,y}u_{bd} -\partial_z\partial_{x}u_B|+|\partial_z\partial_{x,y}v_{bd} -\partial_z\partial_{x}v_B|\leq \varepsilon^6 \Phi_1,\quad
			\\&|\partial_{x,y}\partial_{x,y}u_{bd} -\partial_{x}^2u_B|+
			|\partial_{x,y}\partial_{x,y}v_{bd} -\partial_{x}^2v_B|\leq\varepsilon^6 \Phi_1,\\
&\partial_zu_{bd},\,\,\partial_zv_{bd}\gtrsim e^{-\frac{3}{2}\mu z^2},
		\end{split}
	\end{align} where we use the notation $\partial_{x,y}$ to stand for $\partial_{x}$ or $\partial_{y}$ and for  $i=1,2,$ and
	\begin{equation}\Phi_i= e^{-\frac{A}{x+1}}\left\{
		\begin{aligned}
			( \frac{z }{\sqrt{x+1}})^i ,\quad & 0\leq \frac{z }{\sqrt{x+1}}\leq \delta,\\
			\delta^i ,\quad &  \delta\leq  \frac{z }{\sqrt{x+1}}\leq N,\\
			\delta^i  e^{\frac{3}{2}N^2\mu}e^{-\frac{3}{2}\mu \frac{z ^2}{x+1}},\quad &  \frac{z }{\sqrt{x+1}}\geq N.
		\end{aligned}\right.\\
	\end{equation}
	Here, $\mu,$  $\delta$, $\frac{1}{N}$  are small positive constants independent of $\varepsilon $ and $A$ is $\varepsilon$ to  some negative power.
	Then there exists a solution to \eqref{eq:3DPrandtl} with \eqref{26-02-08-xtheta} and \eqref{1121-1}-\eqref{1121-2}  satisfying the following properties,
	\begin{align}\label{zt}\begin{split}
			&|u -u_B|+|v -v_B|\leq \varepsilon e^{-\frac{A}{x+1}},\\ &
			|\partial_zu -\partial_zu_B|+|\partial_zv -\partial_zv_B|\leq \varepsilon  e^{-\frac{A}{x+1}},\\&
			|\partial_{x,y}u -\partial_{x}u_B|+|\partial_{x,y}v -\partial_{x}v_B|\leq \varepsilon  e^{-\frac{A}{x+1}},
			\\&
			|\partial_z\partial_{x,y}u -\partial_z\partial_{x}u_B|+|\partial_z\partial_{x,y}v -\partial_z\partial_{x}v_B|\leq \varepsilon  e^{-\frac{A}{x+1}},\quad
			\\&|\partial_{x,y}\partial_{x,y}u -\partial_{x}^2u_B|+
			|\partial_{x,y}\partial_{x,y}v -\partial_{x}^2v_B|\leq\varepsilon e^{-\frac{A}{x+1}},
\\&|\partial_{z}^2u -\partial_{z}^2u_B|+
			|\partial_{z}^2v -\partial_{z}^2v_B|\leq2\varepsilon  e^{-\frac{A}{x+1}}.
		\end{split}
	\end{align} And the  detailed growth rates
	in $z$ near $z=0$ and decay rates for $z$ large  can be found in the proof.
\end{theorem}

\subsection{Theorem \ref{1121} to
   Theorem \ref{1121-250628} }
   Take $\varepsilon$ to be the constant satisfying
   \begin{align}\label{26-02-08-relt-epsilon}
  \tilde{\varepsilon}=2\varepsilon e^{-\frac{A}{1+X}}.
\end{align}
   In fact, setting $f(\varepsilon)=2\varepsilon e^{-\frac{A}{1+X}}$ and noting $ f(0^+)=0<\tilde{\varepsilon}\ll2 e^{-1}<2 e^{-\frac{1}{1+X}}=f(1),$ the choice of $\varepsilon$ in \eqref{26-02-08-relt-epsilon} is guaranteed by intermediate value theorem.
Then Theorem \ref{1121} directly yields
   Theorem \ref{1121-250628}
through a  self-similar change of  variables in subsection \ref{26-03-08} and the following Lemma \ref{2025-0630-lemma}.

\begin{lemma}\label{2025-0630-lemma} Through the self-similar change of  variables  \eqref{26-02-08-sscg}, \eqref{bddata} holds under the assumptions in  Theorem \ref{1121-250628}.
\end{lemma}
\begin{proof} As mentioned before, through the self-similar change of  variables  \eqref{26-02-08-sscg}, \eqref{26-02-08-xtheta}-\eqref{tail-1-0316} hold under the assumptions in  Theorem \ref{1121-250628}.
Take $l=1+\theta+2m_1$ where $m_1$ is any small positive constant. By \eqref{26-02-08-relt-epsilon}, for small $\tilde{\varepsilon},$
\begin{align*}
    \tilde{\varepsilon}^l=2^l\varepsilon^l e^{-\frac{A(1+\theta)}{1+X}}e^{-A\frac{2m_1}{1+X}}\leq 2^l\varepsilon^l e^{-A}e^{-Am_1}\leq \varepsilon^{10}e^{-A},
\end{align*}
where we have used  $\frac{2}{1+X}>1$ by \eqref{26-02-08-xtheta} with $\theta$ small and for any positive constants  $k$ and $m_1$, $$\varepsilon^{-k}e^{-Am_1}\rightarrow 0\quad \text{ as}\quad \varepsilon\rightarrow0,$$ since $A$ is $\varepsilon$ to some negative power.
Then \eqref{bddata-250628} implies \begin{align}\begin{split}
			&\|(\partial_zu_{bd},\partial_zv_{bd}) -(\partial_zu_B,\partial_zv_B)\|_{W_{x,y}^{2,\infty}(\{x=0\}\cup\{y=0\}\cup\{z=0\}\cap\partial \Omega)}\\ \leq& \varepsilon^9 e^{-A}[z\chi +e^{-\frac{3}{2} \mu z^2}(1-\chi)], \end{split}
	\end{align}by \eqref{26-02-08-relt-epsilon} and noting $A$ is $\varepsilon$ to some negative power. Note by \eqref{26-02-08-xtheta},
\begin{align*}
    e^{-A}\leq e^{-\frac{A}{x+1}}.
\end{align*}
 Then
by \eqref{1121-1} and \eqref{bddata-250628}, for $z$ large,
\begin{align*}
   |u_{bd}-u_B|\leq &\int_z^{+\infty}
    |\partial_z u_{bd}-\partial_z u_B|dz'
    \leq \varepsilon^9e^{-\frac{A}{x+1}} C\int_z^{+\infty}e^{-\frac{3}{2}\mu z^2}dz'
    \leq 	\varepsilon^9e^{-\frac{A}{x+1}} Ce^{-\frac{3}{2}\mu z^2}.
\end{align*} Moreover, by $u=v=0$ at $z=0$ and compatibility, it holds
 \begin{align*}
   \partial_{x,y} u=\partial_{x,y} v=  \partial_{x,y}\partial_{x,y} u=\partial_{x,y}\partial_{x,y} v=0 \quad at\quad z=0.
 \end{align*}
 Hence, \eqref{1121-1} and \eqref{bddata-250628} imply
\begin{align}\label{20250703-rd-main}\|(u_{bd},v_{bd}) -(u_B,v_B)\|_{W_{x,y}^{2,\infty}}\leq \varepsilon^9 e^{-\frac{A}{x+1}}C[z^2\chi +e^{-\frac{3}{2}\mu z^2}(1-\chi)],	
	\end{align}
Combining  \eqref{bddata-250628}, \eqref{20250703-rd-main}, \eqref{eq:3DPrandtl} and compatibility, we have
\begin{align}
\label{20250703-rd-main-1}\|(u_{bd},v_{bd}) -(u_B,v_B)\|_{W_{z}^{3,\infty}}\leq \varepsilon^9 Ce^{-\frac{A}{x+1}}.\end{align}
On the other hand,  by $u=v=0$ at $z=0$ and $w=-\int_0^z\,\partial_xu+\partial_y v\,dz'$, the solution $u$ to \eqref{eq:3DPrandtl} with \eqref{1121-1}-\eqref{1121-2}  satisfies
$$\partial_z^2 u=\partial_z^3 u=0\quad \text{at}\quad z=0,$$
by compatibility. Combining with \eqref{bddata-250628} and \eqref{20250703-rd-main-1}, it holds
\begin{align}
    |\partial_zu_{bd} -\partial_zu_B|\leq \varepsilon^9e^{-\frac{A}{x+1}}C [z^2\chi +e^{-\frac{3}{2}\mu z^2}(1-\chi)].
\end{align} In addition, \eqref{mnt} and  \eqref{bddata-250628} imply
$$\partial_zu_{bd},\,\,\partial_zv_{bd}\gtrsim e^{-\frac{3}{2}\mu z^2},\quad z>0.$$
		In a similar way, we can obtain other inequalities in \eqref{bddata}  from \eqref{bddata-250628}.
\end{proof}

\section{Vector fields and bootstrap argument}\label{s2}

In this section, we first  introduce the following  four vector fields
\begin{align}\label{26-03-08-def-der}\begin{split}&\nabla^{n-1}_{\xi}=\partial_x-\frac{\int_0^z \partial_{x}u_{n-1}dz'}{u_{n-1}}\partial_z,\quad
\nabla^{n-1}_{\eta}=\partial_y-\frac{\int_0^z \partial_yv_{n-1}dz'}{v_{n-1}}\partial_z,\\
&\nabla^{n-1}_{\psi}=\frac{1}{u_{n-1}}\partial_z,\quad
\nabla^{n-1}_{\tilde{\psi}}=\frac{1}{v_{n-1}}\partial_z.\end{split}\end{align}
By straight calculation, we have
\begin{align}\label{kh1}\begin{split}&\nabla^{n-1}_{\psi}=(1+\tilde{q}_{n-1})\nabla^{n-1}_{\tilde{\psi}},\\
 &[\nabla^{n-1}_{\xi},\nabla^{n-1}_{\psi}]\\=&\nabla^{n-1}_{\xi}\nabla^{n-1}_{\psi}-\nabla^{n-1}_{\psi}\nabla^{n-1}_{\xi}
\\=&\{\frac{1}{u_{n-1}}\partial^2_{xz}-\frac{\partial_{x}u_{n-1}}{u_{n-1}^2}\partial_{z}
-\frac{\int_0^z \partial_{x}u_{n-1}dz'}{u_{n-1}^2}\partial_z^2
+\frac{\int_0^z \partial_{x}u_{n-1}dz'}{u_{n-1}^3}\partial_zu_{n-1}\partial_z\}\\
&-\{\frac{1}{u_{n-1}}\partial^2_{zx}
-\frac{\int_0^z \partial_{x}u_{n-1}dz'}{u_{n-1}^2}\partial_z^2
+(-\frac{\partial_{x}u_{n-1}}{u_{n-1}^2}+\frac{\int_0^z \partial_{x}u_{n-1}dz'}{u_{n-1}^3}\partial_zu_{n-1})\partial_z\}\\
=&0,\end{split}
\end{align}
where 
$$\tilde{q}_{n-1}=\frac{q_{n-1}}{u_{n-1}}.$$
In addition, we have
\begin{align}\label{kh2}\begin{split}
[\nabla^{n-1}_{\eta},\nabla^{n-1}_{\tilde{\psi}}]=&0,\\
\nabla^{n-1}_{\psi}\nabla^{n-1}_{\eta}=&(1+\tilde{q}_{n-1})\nabla^{n-1}_{\tilde{\psi}}\nabla^{n-1}_{\eta}\\
=&(1+\tilde{q}_{n-1})\nabla^{n-1}_{\eta}\nabla^{n-1}_{\tilde{\psi}}\\
=&(1+\tilde{q}_{n-1})\nabla^{n-1}_{\eta}(\frac{1}{1+\tilde{q}_{n-1}}\nabla^{n-1}_{\psi})
\\
=&\nabla^{n-1}_{\eta}\nabla^{n-1}_{\psi}-\frac{\nabla^{n-1}_{\eta}\tilde{q}_{n-1}}{1+\tilde{q}_{n-1}}\nabla^{n-1}_{\psi},
\end{split}\end{align}implying
\begin{align*}
    [\nabla^{n-1}_{\eta},\nabla^{n-1}_{\psi}]=\frac{\nabla^{n-1}_{\eta}\tilde{q}_{n-1}}{1+\tilde{q}_{n-1}}\nabla^{n-1}_{\psi}
.
\end{align*}
However,
 $$[\nabla^{n-1}_{\eta},\nabla^{n-1}_{\psi}],\quad[\nabla^{n-1}_{\xi},\nabla^{n-1}_{\eta}],\quad [\nabla^{n-1}_{\psi},\nabla^{n-1}_{\tilde{\psi}}]\neq 0$$
 in general unless  $u_{n-1}= v_{n-1} .$ In particular, $[\nabla^{n-1}_{\xi},\nabla^{n-1}_{\eta}]$ is a vector field in normal direction with respect to boundary which is a key quantity related to  symmetry. For brevity, by denoting $\nabla^{n-1}_{\xi}=\partial_x-G\partial_z$ and $\nabla^{n-1}_{\eta}=\partial_y-F\partial_z$, we have
\begin{align}\label{K}\begin{split}
[\nabla^{n-1}_{\xi},\nabla^{n-1}_{\eta}]=&(\partial_yG-\partial_xF+G\partial_zF-F\partial_zG)\partial_z
\\=&K_{n-1}\partial_z.
\end{split}
\end{align} For symmetric solutions with respect to $x$ and $y$, $K_{n-1}\equiv0.$
For the perturbation of the symmetric profile, the smallness of $K_{n-1}$ is important to obtain stability.
For a better description of the perturbation,
  we give the notations for the following vector fields. Recall
 $ \bar{u}(x,y,z)=u_B(x,y,z+{\epsilon_0})$ and set
\begin{align*}&\nabla_{\tau_1}=\partial_x-\frac{\int_0^z \partial_{x}\bar{u}dz'}{\bar{u}}\partial_z,\quad
\nabla_{\tau_2}=\partial_y-\frac{\int_0^z \partial_y\bar{u}dz'}{\bar{u}}\partial_z,\quad
\nabla_{n}=\frac{1}{\bar{u}}\partial_z.\end{align*}
Note that
 $$\nabla_{\tau_1}\bar{u}=\nabla_{\tau_2}\bar{u},\quad [\nabla_{\tau_1},\nabla_{n}]=0.$$
The growth rates of the smooth function
 $\bar{u}$ and its derivatives will be used later with details given in Subsection \ref{1027} in the Appendix. The difference between vector fields $\nabla_\psi^{n-1}, \nabla_n$ and $\nabla_\xi^{n-1}, \nabla_\eta^{n-1}, \nabla_{\tau_1},\nabla_{\tau_2}$  will be given in Subsection \ref{tl} which yields the growth estimates of the remainders in the equations satisfied by $u_n-\bar{u}$ and its derivatives.

For brevity, we will use $\tilde{q},K,$ $\nabla_\xi, \nabla_\eta,\nabla_\psi$ to denote $\tilde{q}_{n-1},K_{n-1},$ $\nabla^{n-1}_\xi, \nabla^{n-1}_\eta,\nabla^{n-1}_\psi$ respectively  in the following discussion.

The relation of the vector field derivatives and the derivatives in  original coordinates is given as follows.
\begin{align}\label{rlt}\begin{split}&\partial_xf=\nabla _{\xi} f+\frac{\int_0^z \partial_{x}u_{n-1}dz'}{u_{n-1}}\partial_zf,\\
&\partial_yf=\nabla _{\eta} f+ \frac{\int_0^z \partial_{y}v_{n-1}dz'}{v_{n-1}}\partial_zf,\\
&\partial_zf=u_{n-1}\nabla _{\psi}f,\\
&\partial_z\partial_xf=\partial_z\nabla _{\xi} f+\partial_z(\frac{\int_0^z \partial_{x}u_{n-1}dz'}{u_{n-1}})\partial_zf+\frac{\int_0^z \partial_{x}u_{n-1}dz'}{u_{n-1}}\partial_z^2f,\\
&\partial_z\partial_yf=\partial_z\nabla _{\eta} f+\partial_z( \frac{\int_0^z \partial_{y}v_{n-1}dz'}{v_{n-1}})\partial_zf+ \frac{\int_0^z \partial_{y}v_{n-1}dz'}{v_{n-1}}\partial_z^2f.\end{split}\end{align}

\subsection{Commutator K}
In this subsection, we derive  the expressions of commutator $K$ and its derivative $\partial_z K$.
Suppose that  $f$ is a smooth function. With details given in Subsection \ref{AppK}  in the Appendix,  $K$ satisfies 
\begin{align}\label{0429-25}\begin{split}
K\partial_z(-u_{n-1}\nabla_{\psi}^2f)=  &- u_{n-1}\nabla_{\psi}^2(K\partial_zf)\\&- (\nabla_{\xi} \nabla_{\eta}u_{n-1}-\nabla_{\eta}\nabla_{\xi} u_{n-1})\nabla_{\psi}^2f
    \\
    &-2u_{n-1}\nabla_{\xi}(\frac{\nabla_{\eta}\tilde{q}}{1+\tilde{q}})\,\,\nabla_{\psi}^2f
-u_{n-1}\big(\nabla_{\psi}\nabla_{\xi}(\frac{\nabla_{\eta}\tilde{q}}{1+\tilde{q}})\big)\nabla_{\psi}f
.\end{split}
\end{align}
\subsubsection{Estimates of $K$ and $\partial_z K$ }
We will use the following function
\begin{align}\label{psi0125}
	\psi_{n-1}(x,y,z)=\int_0^z u_{n-1}(x,y,z')dz'.
\end{align}
Note that  $\psi_{n-1}$ is an increasing function with respect to $z$ since $ u_{n-1}>0$ for $z>0$ by induction assumption \eqref{n-1assumption}.
In particular, $$\nabla_\psi \psi_{n-1}=\frac{1}{u_{n-1}}\partial_z\int_0^z u_{n-1}dz'=1,\quad \nabla_\psi^2 \psi_{n-1}=0,\quad \nabla_\xi \psi_{n-1}=0,$$ and
\begin{align*}
\nabla_{\xi}\nabla_{\eta}\psi_{n-1}   -\nabla_{\eta}  \nabla_{\xi}\psi_{n-1} =K\partial_z\psi_{n-1}= Ku_{n-1}.
\end{align*} Hence, by replacing $f$ by $\psi_{n-1}$ in \eqref{0429-25}, we have
\begin{align}\label{3.7}\begin{split}
  \partial_z\nabla_{\psi}(Ku_{n-1}) =u_{n-1}\nabla_{\psi}^2(Ku_{n-1})=
-u_{n-1}\big(\nabla_{\psi}\nabla_{\xi}(\frac{\nabla_{\eta}\tilde{q}}{1+\tilde{q}})\big)
=-\partial_z\nabla_{\xi}(\frac{\nabla_{\eta}\tilde{q}}{1+\tilde{q}})
.\end{split}
\end{align}
By \eqref{K}, there exists a  function $K_{\infty}(x,y)$ such that \begin{align}\begin{split}
K=&(\partial_yG-\partial_xF+G\partial_zF-F\partial_zG)
\\ \rightarrow& K_\infty(x,y),\quad z\rightarrow +\infty,
\end{split}
\end{align}with $G=\frac{\int_0^z \partial_xu_{n-1}dz'}{u_{n-1}}$ and
$F=\frac{\int_0^z \partial_yv_{n-1}dz'}{v_{n-1}}$, where we have used the fact that $G, F, \partial_yG, \partial_xF$ are convergent as $z$ goes to infinity by the decay assumption of $\partial_{xy}^2u_{n-1},\partial_{xy}^2v_{n-1}, \partial_{x}u_{n-1}, \partial_{y}v_{n-1}.$
Then $K u_{n-1}\rightarrow K_{\infty}(x,y)$ as $z\rightarrow+\infty $ implies $\partial_z(K u_{n-1})\rightarrow 0$ as $z\rightarrow+\infty$. Hence
\begin{align*}
 -\nabla_{\psi}(Ku_{n-1})=\int_z^{+\infty} \partial_z\nabla_{\psi}(Ku_{n-1})dz'=-\int_z^{+\infty} \partial_z\nabla_{\xi}(\frac{\nabla_{\eta}\tilde{q}}{1+\tilde{q}})dz'
    =\nabla_{\xi}(\frac{\nabla_{\eta}\tilde{q}}{1+\tilde{q}}).
\end{align*}
Thus,
\begin{align}
 -\frac{1}{u_{n-1}}\partial_z (Ku_{n-1})=\nabla_{\xi}(\frac{\nabla_{\eta}\tilde{q}}{1+\tilde{q}}),
\end{align}
implies that
\begin{align}\label{dzk}
 \partial_z K=-\nabla_{\xi}(\frac{\nabla_{\eta}\tilde{q}}{1+\tilde{q}})
 - K\frac{\partial_zu_{n-1}}{u_{n-1}},
\end{align}
and
\begin{align}\label{yq}
    Ku_{n-1}&=\int_0^{z} \partial_z (Ku_{n-1})dz'= \int_0^{z} - u_{n-1}\nabla_{\xi}(\frac{\nabla_{\eta}\tilde{q}}{1+\tilde{q}})dz',
\end{align}
because $K=0$ at $z=0$ by the definition of $K$ in \eqref{K}.
Moreover,  by the induction assumption \eqref{n-1assumption} in Subsection \ref{Bsassm} and estimates of derivatives of $\tilde{q}$ in \eqref{tildq0221}, we have
\begin{align}\label{Ksmallinfty}
    |K|\leq \varepsilon^2e^{-\frac{A}{x+1}} C_2,
\end{align}and
\begin{align}\label{dzKsmallinfty}
    |\partial_z K|\leq\frac{ \varepsilon^2 C}{u_{n-1}}\phi_{1,0},
\end{align}where \begin{equation}\label{phi3}\phi_{1,0}= \left\{
  \begin{aligned}
 e^{-\frac{A}{x+1}},\quad &  0\leq \frac{z+{\epsilon_0}}{\sqrt{x+1}}\leq N,\\ e^{-\frac{A}{x+1}}e^{N^2\mu}e^{-\frac{(z+{\epsilon_0})^2}{x+1}\mu},\quad\quad & \frac{z+{\epsilon_0}}{\sqrt{x+1}}\geq N .
\end{aligned}\right.\\
\end{equation}

\subsection{Induction and bootstrap argument}\label{Bsassm}
\subsubsection{Induction }
We will apply an induction on $n\in\N$. Recall $$u_0=v_0=\bar{u}=u_B(x,y,z+{\epsilon_0}).$$
 Assume for $n\geq 1,$ $u_{k}$ and $v_{k},$ $k=0,1,\cdots,n-1,$ satisfy
 \begin{align}\label{n-1assumption}\begin{split}
&|u_{k}-\bar{u}|,\,|v_{k}-\bar{u}|\leq\varepsilon^6 \phi_{1,1+2\alpha},\\ &
|\partial_zu_{k}-\partial_z\bar{u}|,\,|\partial_zv_{k}-\partial_z\bar{u}|\leq \varepsilon^5 \phi_{1,\alpha} ,\\&|\partial_{x,y}u_{k}-\partial_{x}\bar{u}|,\,|\partial_{x,y}v_{k}-\partial_{x}\bar{u}|\leq \varepsilon^2 \phi_{1,1},
\\&
|\partial_z\partial_{x,y}u_{k}-\partial_z\partial_{x}\bar{u}|,\,|\partial_z\partial_{x,y}v_{k}-\partial_z\partial_{x}\bar{u}|\leq \varepsilon^2 \phi_{1,\alpha},\quad
\\&|\partial_{x,y}\partial_{x,y}u_{k}-\partial_{x}^2\bar{u}|,\,
|\partial_{x,y}\partial_{x,y}v_{k}-\partial_{x}^2\bar{u}|\leq\varepsilon^2 \phi_{1,\frac{\alpha}{2}},\\
&|\partial_{z}^2u_{k}-\partial_{z}^2\bar{u}|\leq \frac{\varepsilon e^{-\frac{A}{x+1}}}{1+\frac{\alpha}{7}}\min\{1,(z+\min_{[0,X]\times[0,Y]}\frac{\bar{u}(x,y,0)}{\partial_z\bar{u}(x,y,0)})^\alpha\}
,\\
&\partial_{z}u_{k},\,\,\partial_{z}v_{k}\geq c_0e^{-\frac{3}{2}\mu (z+\epsilon_0)^2},
 \end{split}
\end{align}
 in $\Omega$
  for some small positive constants $c_0<\min\{\min_{[0,X]\times[0,Y]\times[0,2]}\partial_zu_B,\frac{1}{10}e^{\frac{3}{2}\mu (z+\epsilon_0)^2}\partial_zu_B\},$
 $\delta\leq \frac{1}{3},$
 $\mu$, $\frac{1}{N}$ and   $\alpha$ independent of $\varepsilon$ and a large constant $A>C\varepsilon^{-\frac{6}{\alpha}}$.
 Recall $\phi_{1,0}$ in  \eqref{phi3} and set for $\beta\in(0,1+2\alpha],$
\begin{equation}\label{phi1}\phi_{1,\beta}= \left\{
  \begin{aligned}
  e^{-\frac{A}{x+1}}(\frac{z+{\epsilon_0}}{\sqrt{x+1}})^\beta,\quad & 0\leq \frac{z+{\epsilon_0}}{\sqrt{x+1}}\leq \delta,\\
  e^{-\frac{A}{x+1}}\delta^{\beta},\quad &  \delta\leq \frac{z+{\epsilon_0}}{\sqrt{x+1}}\leq N,\\
  e^{-\frac{A}{x+1}}\delta^{\beta}e^{N^2\mu}e^{-\frac{(z+{\epsilon_0})^2}{x+1}\mu},\quad\quad & \frac{z+{\epsilon_0}}{\sqrt{x+1}}\geq N .
\end{aligned}\right.\\
\end{equation}
Through the equation \eqref{appeqeuclidean}, the estimates of
 \begin{align*}
   \partial_z^4 u_k,\quad \partial_z^4 v_k,\quad \partial_{x,y}\partial_z^2 u_k,\quad \partial_{x,y}\partial_z^2 v_k,
\end{align*}
are derived from  \eqref{n-1assumption}.

Here are some remarks on \eqref{n-1assumption}.
The second to last inequality in \eqref{n-1assumption} is natural, since by compatibility, the solution of \eqref{eq:3DPrandtl} satisfies
\begin{align*}
    \partial_z^2 u=0\quad \text{on}\quad z=0.
\end{align*}
Moreover, for the second to  last inequality in \eqref{n-1assumption}, we have
\begin{align}\label{1119-1}\begin{split}
    |\partial_{z}^2u_{n-1}-\partial_{z}^2\bar{u}|\leq & \frac{\varepsilon e^{-\frac{A}{x+1}}}{1+\frac{\alpha}{7}}\min\{1,(z+\min_{[0,X]\times[0,Y]}\frac{\bar{u}(x,y,0)}{\partial_z\bar{u}(x,y,0)})^\alpha\}
\\ \leq & \varepsilon e^{-\frac{A}{x+1}}\min\{1,(z+\epsilon_0)^\alpha\},\end{split}
\end{align}
and the detailed calculation is given  in the Appendix.
In addition, by \eqref{mnt}, \eqref{tail-1-0316} and \eqref{n-1assumption},  for some positive constant $C_0,$ it holds
\begin{align}\label{positivelbu}
  0<c_0\leq \partial_z u_{n-1}, \partial_z v_{n-1}\leq C_0 \quad \text{in}\quad\Omega\cap\{z\leq 1\},
\end{align}
and  for $z$ large,
\begin{align}
  c_0e^{-\frac{3}{2}\mu (z+\epsilon_0)^2}\leq\partial_z u_{n-1}, \partial_z v_{n-1} \leq C_0 e^{-\frac{(z+{\epsilon_0})^2}{x+1}\mu}\quad \text{in}\quad\Omega.
\end{align}
Next, we derive the estimates of vector field derivatives based on \eqref{n-1assumption}.
\begin{lemma}\label{04160416e-v} Under the assumption in \eqref{n-1assumption},
it holds that
\begin{align*}
    |\partial_z\nabla_{\eta,\xi}(u_{n-1}-\bar{u})|\leq\varepsilon^2 C\phi_{1,\alpha},\,\,
|\nabla_{\eta,\xi}\nabla_{\eta,\xi}(u_{n-1}-\bar{u})|\leq\varepsilon^2 C\phi_{1,\frac{\alpha}{2}}\quad \text{in}\quad \Omega,
\end{align*}where $\nabla_{\eta,\xi}$ stands for $\nabla_{\eta}$ or $\nabla_{\xi}$.
\end{lemma}
\begin{proof}
If $n=1$, then $u_{n-1}=u_0=\bar{u}$ and the result holds.
Now we consider $n\geq 2.$

By \eqref{n-1assumption}, the growth rates of $\bar{u}$ and its derivatives in Subsection \ref{1027},
 \eqref{xitau} and \eqref{rlt},
we have \begin{align}\label{250416commongr}\begin{split}
&| \frac{\int_0^z \partial_{xy}^2u_{n-1}dz'}{u_{n-1}}|\leq C\bar{u}^{\frac{\alpha}{2}}, \quad|\partial_z( \frac{\int_0^z \partial_{y}v_{n-1}dz'}{v_{n-1}})|\leq C, \\&
  |\frac{\int_0^z \partial_{y}v_{n-1}dz'}{v_{n-1}}|\leq C\bar{u}\quad\text{in}\quad\Omega.\end{split}
\end{align}
and
\begin{align}\label{250416-nabla1}\begin{split}
   | \nabla_{\xi}^{n-2}u_{n-1}-\nabla_{\tau_1}\bar{u}|=&|\nabla_{\xi}^{n-2}(u_{n-1}-\bar{u})
    +(\nabla_{\xi}^{n-2}-\nabla_{\tau_1})\bar{u}|\leq \varepsilon^2C\phi_{1,1},\\
    | \nabla_{\eta}^{n-2}u_{n-1}-\nabla_{\tau_1}\bar{u}|=&
    | \nabla_{\eta}^{n-2}u_{n-1}-\nabla_{\tau_2}\bar{u}|\leq \varepsilon^2C\phi_{1,1}.
    \end{split}
\end{align}
Here we recall the definition of $ \nabla_{\xi}^{n-2},$ $\nabla_{\eta}^{n-2}$ in \eqref{26-03-08-def-der} and  $$\nabla_{\xi}^{0}=\nabla_{\tau_1}, \nabla_{\eta}^{0}=\nabla_{\tau_2}.$$

\textit{Step 1} We will prove
\begin{align}\label{250416-zz}
|\partial_z^2(u_{n-1}-\bar{u})|\leq  \varepsilon^{2}C\frac{1}{\bar{u}} \phi_{1,\alpha}\quad\text{in}\quad \Omega.
\end{align}

By \eqref{appeqeuclidean}, we have
\begin{align}\begin{split}
\partial_{z}^2u_{n-1}= &[u_{n-2}(\nabla_{\xi}^{n-2} u_{n-1} +(1+\tilde{q}_{n-2})\nabla_{\eta}^{n-2} u_{n-1} )-\partial_{z}u_{n-1} \partial_{z}(\frac{u_{n-1}}{u_{n-2}} )]\frac{u_{n-2}}{u_{n-1}},\\
\partial_{z}^2\bar{u}= &2\bar{u}\nabla_{\tau_1} \bar{u} \quad \text{in}\quad \Omega.
    \end{split}
\end{align}
Note
by \eqref{n-1assumption}, it holds
\begin{align*}\begin{split}
|\partial_z (\frac{ u_{n-1}}{u_{n-2}})|=&|\partial_z (\frac{ u_{n-1}-u_{n-2}}{u_{n-2}})|\\ \leq &\frac{ |\partial_z u_{n-1}-\partial_zu_{n-2}|}{u_{n-2}}+\frac{\partial_zu_{n-2}|u_{n-2}-u_{n-1}|}{u_{n-2}^2}
\\ \leq& C\varepsilon^{5}\frac{1}{u_{n-2}} \phi_{1,\alpha}\quad\text{in}\quad \Omega.\end{split}
\end{align*}
Moreover, combining \eqref{250416-nabla1}, we complete the proof of \eqref{250416-zz}.

\textit{Step 2} We will prove
\begin{align}\label{250416-znabla1}
| \partial_z\nabla _{\eta,\xi} (u_{n-1}-\bar{u})|\leq
\varepsilon^2 C\phi_{1,\alpha}\quad \text{in}\quad \Omega.
\end{align}
By \eqref{rlt}, we have
\begin{align*}
  \partial_z\nabla _{\eta} (u_{n-1}-\bar{u})=  & \partial_z\partial_y(u_{n-1}-\bar{u})-\partial_z( \frac{\int_0^z \partial_{y}v_{n-1}dz'}{v_{n-1}})\partial_z(u_{n-1}-\bar{u})
  \\&- \frac{\int_0^z \partial_{y}v_{n-1}dz'}{v_{n-1}}\partial_z^2(u_{n-1}-\bar{u}).
\end{align*}
Then by \eqref{250416commongr} and \eqref{250416-zz}, we prove \begin{align*}
| \partial_z\nabla _{\eta} (u_{n-1}-\bar{u})|\leq
\varepsilon^2 C\phi_{1,\alpha}\quad \text{in}\quad \Omega.
\end{align*} Similarly, we can prove
$| \partial_z\nabla _{\xi} (u_{n-1}-\bar{u})|\leq
\varepsilon^2 C\phi_{1,\alpha}$ in $\Omega.$  Then we complete the proof of \eqref{250416-znabla1}.

\textit{Step 3} We will prove
\begin{align}\label{250416-nabla22}
| \nabla _{\eta}\nabla _{\xi} (u_{n-1}-\bar{u})|\leq
\varepsilon^2 C\phi_{1,\frac{\alpha}{2}}\quad \text{in}\quad \Omega.
\end{align}
By \eqref{rlt}, we have
\begin{align*}\begin{split}
\nabla _{\eta} \nabla _{\xi}(u_{n-1}-\bar{u})=&\partial_{yx}^2(u_{n-1}-\bar{u})- \frac{\int_0^z \partial_{xy}^2u_{n-1}dz'}{u_{n-1}}\partial_z (u_{n-1}-\bar{u})
\\&+ \frac{\partial_yu_{n-1}\int_0^z \partial_{x}u_{n-1}dz'}{u_{n-1}^2}\partial_z (u_{n-1}-\bar{u})
\\&- \frac{\int_0^z \partial_{x}u_{n-1}dz'}{u_{n-1}}\partial_{zy}^2 (u_{n-1}-\bar{u})
- \frac{\int_0^z \partial_{y}v_{n-1}dz'}{v_{n-1}}\partial_z\nabla _{\xi}(u_{n-1}-\bar{u})
\end{split}\end{align*}
Then  by \eqref{n-1assumption}, \eqref{250416commongr} and \eqref{250416-znabla1}, we complete the proof of \eqref{250416-nabla22}.
The rest inequalities in this lemma can be proved similarly and then we skip the details.
\end{proof}
 \eqref{n-1assumption} and Lemma \ref{04160416e-v} imply  for some positive constant $C,$
\begin{align}\label{assumpn}\begin{split}
&|q_{n-1}|\leq\varepsilon^6 C\phi_{1,1+2\alpha},\quad
|\partial_zq_{n-1}|\leq \varepsilon^5C \phi_{1,\alpha},
\\&
|\partial_z\nabla_{\eta,\xi}q_{n-1}|\leq\varepsilon^2 C\phi_{1,\alpha},\quad
|\nabla_{\eta,\xi}\nabla_{\eta,\xi}q_{n-1}|\leq\varepsilon^2 C\phi_{1,\frac{\alpha}{2}},\\
&|u_{n-1}|+|v_{n-1}|\leq C ,\quad |\partial_zu_{n-1}|+|\partial_zv_{n-1}|\leq C e^{-\frac{(z+{\epsilon_0})^2}{x+1}\mu}
,\\&
|\nabla_{\eta,\xi}u_{n-1}|\leq C\min\{z+{\epsilon_0},1\}e^{-\frac{(z+{\epsilon_0})^2}{x+1}\mu},
\quad
|\partial_z\nabla_{\eta,\xi}u_{n-1}|\leq Ce^{-\frac{(z+{\epsilon_0})^2}{x+1}\mu},
\\& |\nabla_{\eta,\xi}\nabla_{\eta,\xi}u_{n-1}|+|\nabla_{\eta,\xi}\nabla_{\eta,\xi}v_{n-1}|\leq C \min\{(z+{\epsilon_0})^{\frac{\alpha}{2}},1\}e^{-\frac{(z+{\epsilon_0})^2}{x+1}\mu},
 \end{split}
\end{align}
  in $\Omega$.
In particular, by the boundary data on $z=0$ and $|\partial_z\nabla_{\eta,\xi}q_{n-1}|\leq\varepsilon^2 C\phi_{1,\alpha}$,
we  have
\begin{align}\label{0221qn-1}
|\nabla_{\eta,\xi}q_{n-1}|\leq\varepsilon^2 C\phi_{1,1+\alpha},\quad \text{in}\quad \Omega.
\end{align}
Thus we have the following estimates of vector field derivatives of  $\tilde{q},$
\begin{align}\label{tildq0221}\begin{split}
|\nabla_{\eta,\xi}\tilde{q}|=&|\nabla_{\eta,\xi}(\frac{q_{n-1}}{u_{n-1}})|
=|\frac{1}{u_{n-1}}\nabla_{\eta,\xi}q_{n-1}-\nabla_{\eta,\xi}u_{n-1}\frac{q_{n-1}}{u_{n-1}^2}|\leq\varepsilon^2 C\phi_{1,\alpha},\\
|\nabla_{\eta,\xi}\nabla_{\eta,\xi}\tilde{q}|=&|-\frac{\nabla_{\eta,\xi}u_{n-1}}{u_{n-1}^2}\nabla_{\eta,\xi}q_{n-1}+\frac{1}{u_{n-1}}\nabla_{\eta,\xi}\nabla_{\eta,\xi}q_{n-1}
-\nabla_{\eta,\xi}\nabla_{\eta,\xi}u_{n-1}\frac{q_{n-1}}{u_{n-1}^2}
\\&-\frac{\nabla_{\eta,\xi}u_{n-1}}{u_{n-1}}\nabla_{\eta,\xi}\tilde{q}+\frac{\nabla_{\eta,\xi}u_{n-1}\nabla_{\eta,\xi}u_{n-1}}{u_{n-1}^2}\tilde{q}|
\\ \leq&\varepsilon^2 C\frac{1}{u_{n-1}}\phi_{1,\frac{\alpha}{2}}\quad \quad \text{in}\quad \Omega.\end{split}
\end{align}

To complete the induction,
we will show that $u_{n}$ and $v_{n}$ satisfy
 \begin{align}\label{assumpn-2}\begin{split}
&|u_{n}-\bar{u}|,\,|v_{n}-\bar{u}|\leq d_0\varepsilon^6 \phi_{1,1+2\alpha},\\&
|\partial_zu_{n}-\partial_z\bar{u}|,\,|\partial_zv_{n}-\partial_z\bar{u}|\leq d_0\varepsilon^5 \phi_{1,\alpha},\\&
|\nabla_{\eta,\xi}u_{n}-\nabla_{\tau_1}\bar{u}|,\,|\nabla_{\eta,\xi}v_{n}-\nabla_{\tau_1}\bar{u}|\leq d_0\varepsilon^2 \phi_{1,1},
\\&
|\partial_z\nabla_{\eta,\xi}u_{n}-\partial_z\nabla_{\tau_1}\bar{u}|,\,|\partial_z\nabla_{\eta,\xi}v_{n}-\partial_z\nabla_{\tau_1}\bar{u}|\leq d_0\varepsilon^2 \phi_{1,\alpha},\quad
\\&|\nabla_{\eta,\xi}\nabla_{\eta,\xi}u_{n}-\nabla_{\tau_1}^2\bar{u}|,\,
|\nabla_{\eta,\xi}\nabla_{\eta,\xi}v_{n}-\nabla_{\tau_1}^2\bar{u}|\leq d_0\varepsilon^2 \phi_{1,\frac{\alpha}{2}},
\\&|\partial_{z}^2u_{n}-\partial_{z}^2\bar{u}|\leq \frac{\varepsilon e^{-\frac{A}{x+1}}}{1+\frac{\alpha}{7}}\min\{1,(z+\min_{[0,X]\times[0,Y]}\frac{\bar{u}(x,y,0)}{\partial_z\bar{u}(x,y,0)})^\alpha\}
,\\
&\partial_{z}u_{n},\,\,\partial_{z}v_{n}\geq c_0e^{-\frac{3}{2}\mu (z+\epsilon_0)^2},
 \end{split}
\end{align} for some positive constant  $d_0\ll 1$ independent of $\varepsilon$ and  $\varepsilon\ll d_0 $
 in $\Omega$, which implies \eqref{n-1assumption} with $n-1$ replaced by $n$ through the transformation formulas in \eqref{rlt}. ( Refer to Theorem \ref{lm6} in Subsection \ref{66} in the Appendix for detailed calculation of transformation.)
\subsubsection{Bootstrap argument}
To prove \eqref{assumpn-2}, we will employ a bootstrap argument as follows.

If \eqref{assumpn-2} holds in $\Omega$, then it is proved. Otherwise, there exists a
maximum  $Y^*\in(0,Y)$   such that
\begin{align}\label{defy}
Y^*=\max \{y_0\in[0,Y]|\eqref{assumpn-2}\,\, \text{holds}\,\, \text{in}\,\,\{0\leq y\leq y_0\} \cap\Omega\},
\end{align}
 by the boundary condition. To be rigorous, we will add $-\gamma$ and $\gamma$ to the right hand side of the last inequality and other inequalities in \eqref{assumpn-2} respectively
 with
 $0<\gamma \ll\epsilon^5_0 e^{-A}$  so that  $Y^*$ is well-defined and finally let $\gamma$ go to $0$ as in the proof of Theorem \ref{5.1}. Here, we write it in this form just for brevity. We will prove the inequalities in \eqref{assumpn-2} hold in $\Omega\cap\{0\leq y \leq Y^*\}$ with $d_0$ replaced by a constant strictly smaller than $d_0$ in the first 5 lines and $\frac{\varepsilon}{1+\frac{\alpha}{7}} , c_0$ replaced by $\frac{\varepsilon}{1+\frac{2\alpha}{13}} , 2c_0$ respectively in the last two lines.
 Then there is no $Y^*\in(0,Y)$  satisfying \eqref{defy}. Hence, \eqref{assumpn-2} holds in $\Omega$.

  In the following, we  present the estimates of $u_n$ and its derivatives, since the analysis on $v_n$ and its derivatives is the same.
\section{ Maximum principle in terms of   the vector fields}\label{s3}
For the function $\psi_{n-1}$ defined in \eqref{psi0125}, we have
\begin{align}\label{psin-1}\begin{split}
\nabla_\psi \psi_{n-1}=&\frac{1}{u_{n-1}}\partial_z\int_0^z u_{n-1}(x,y,z')dz'=1,\quad \nabla_\psi^2 \psi_{n-1}=0,\\
 -\nabla_\psi^2 \psi_{n-1}^\beta=&-\beta\nabla_\psi( \psi_{n-1}^{\beta-1})=-\beta(\beta-1)\psi_{n-1}^{\beta-2},\\
 \nabla_{\xi}\psi_{n-1}=&\int_0^z \partial_xu_{n-1}dz'-\int_0^z \partial_{x}u_{n-1}dz'=0,
 \\
 \nabla_{\eta}\psi_{n-1}=&\int_0^z \partial_yu_{n-1}dz'-\frac{u_{n-1}\int_0^z \partial_yv_{n-1}dz'}{v_{n-1}}\\
 =&\int_0^z \partial_yu_{n-1}dz'-\frac{\int_0^z \partial_yu_{n-1}+ \partial_yq_{n-1}dz'}{1+\tilde{q}}
 \\
 =&\frac{\tilde{q}}{1+\tilde{q}}\int_0^z \partial_yu_{n-1}dz'-\frac{ \int_0^z \partial_yq_{n-1}dz'}{1+\tilde{q}},\end{split}
\end{align} where  $ |\int_0^z \partial_yq_{n-1}dz'|, |\tilde{q}|\ll \varepsilon$ in $\Omega$ by \eqref{n-1assumption}. Hence, for $\beta\in \R_+,$
\begin{align}\label{4.2}\begin{split}
   &\nabla_\xi \psi_{n-1}^\beta +(1+\tilde{q})\nabla_{\eta}  \psi_{n-1}^\beta+b\nabla_{\psi}\psi_{n-1}^\beta-u_{n}\nabla_{\psi}^2\psi_{n-1}^\beta
 +c \psi_{n-1}^\beta \\ =&\beta(\tilde{q}\int_0^z \partial_yu_{n-1}dz'- \int_0^z \partial_yq_{n-1}dz'+b)\psi_{n-1}^{\beta-1}-\beta(\beta-1)u_{n}\psi_{n-1}^{\beta-2} +c \psi_{n-1}^\beta.
\end{split}\end{align}
 Note that  $u_{n-1}\psi_{n-1}^{\beta-2}\geq \lambda_0 \psi_{n-1}^{\beta-\frac{3}{2}}$ for some positive constant $\lambda_0$ when $z\ll1$.
  If $b$ and $c$ are  bounded functions, then  for  small $z$  such that $$\psi_{n-1}\leq \delta$$ with $\delta$ depending on $\|b\|_{L^\infty((0,X]\times(0,Y]\times(0,\infty))}$ and $\|c\|_{L^\infty((0,X]\times(0,Y]\times(0,\infty))}$, we have the following two inequalities. For $\beta\in(0,1)$, there exists a positive constant $\lambda_1$ such that \begin{align}\label{psin-1x1}\begin{split}
   \nabla_\xi \psi_{n-1}^\beta +(1+\tilde{q})\nabla_{\eta}  \psi_{n-1}^\beta+b\nabla_{\psi}\psi_{n-1}^\beta-u_{n}\nabla_{\psi}^2\psi_{n-1}^\beta
  +c \psi_{n-1}^\beta\geq \lambda_1 \psi_{n-1}^{\beta-\frac{3}{2}}>0, \quad \psi_{n-1}\leq \delta,\end{split}\end{align}
   and for $\beta\in(1,+\infty)$, there exists a positive constant $\lambda_2$  such that \begin{align}\label{psin-1x1-alphabig}\begin{split}
   \nabla_\xi \psi_{n-1}^\beta +(1+\tilde{q})\nabla_{\eta}  \psi_{n-1}^\beta+b\nabla_{\psi}\psi_{n-1}^\beta-u_{n}\nabla_{\psi}^2\psi_{n-1}^\beta
  +c \psi_{n-1}^\beta\leq -\lambda_2 \psi_{n-1}^{\beta-\frac{3}{2}}<0,\quad \psi_{n-1}\leq \delta.\end{split}\end{align}
On the other hand,
for $\psi_{n-1}\geq \delta$, by \eqref{n-1assumption}, we have
\begin{align}\label{psin-1x2}\begin{split}
  |(1+\tilde{q})\nabla_{\eta}\psi_{n-1}^{\beta}+b\nabla_{\psi}\psi_{n-1}^{\beta}| = &|\beta(\tilde{q}\int_0^z \partial_yu_{n-1}dz'- \int_0^z \partial_yq_{n-1}dz'+b) \psi_{n-1}^{\beta-1}|\\ \leq &\beta( \frac{1+\|b\|_{L^\infty((0,X]\times(0,Y]\times(0,\infty))}}{\delta})\psi_{n-1}^{\beta}
,\quad \psi_{n-1}\geq \delta.\end{split}\end{align}

With these estimates, we will prove the following two maximum principles.

\begin{lemma}[Maximum Principle in bounded domain]\label{bdMP}
For any positive constant $z_0,$ assume $f\in C^2\big((0,X]\times(0,Y]\times(0,z_0)\big)\cap C\big([0,X]\times[0,Y]\times[0,z_0]\big).$
  If
$f\leq 0 $ on $ [0,X]\times[0,Y]\times\{z=0,z_0\}\cup\{x=0\}\times[0,Y]\times[0,z_0]\cup[0,X]\times\{y=0\}\times[0,z_0],$ and
\begin{align}\label{bdmax}
    Lf\leq0 \quad\text{in}\quad(0,X]\times(0,Y]\times(0,z_0),
   \end{align} where $L f=\nabla_\xi f +(1+\tilde{q})\nabla_{\eta}  f+b\nabla_{\psi}f-u_{n}\nabla_{\psi}^2f +cf
$, $1+\tilde{q}>0$ and $b$ and $c$ are bounded functions,
then $f\leq 0$ in $[0,X]\times[0,Y]\times[0,z_0].$
\end{lemma}
\begin{proof}
Set $$G=fe^{-Bx},$$ where
\begin{align}\label{largeB}
 B>\|c\|_{L^\infty((0,X]\times(0,Y]\times(0,z_0))}.
\end{align}


 By the assumption,
$G\leq 0 $ on the boundary $ \big(\{x=0\}\cup\{y=0\}\cup\{z=0,z_0\}\big)\cap [0,X]\times[0,Y]\times[0,z_0].$ Hence, $G$ does not attain a positive maximum on the boundary.

We now prove $G$ does not attain any  positive maximum in the interior $(0,X]\times(0,Y]\times(0,z_0)$ by contradiction.
If $G$ attains its positive maximum at some interior point $p\in (0,X]\times(0,Y]\times(0,z_0),$ then
  \begin{align}\label{fp}
    f(p)>0,
\end{align}and
$\partial_zG(p)=0, \partial_xG(p)\geq0, \partial_yG(p)\geq0$ and $\partial_z^2G(p)\leq0.$ Hence
\begin{align}\label{xzx}\begin{split}
   \nabla_\eta G(p)\geq&0,\quad \nabla_\xi G(p)\geq0,\quad \nabla_{\psi}G(p)=0,
  \\  \nabla_{\psi}^2G=&\frac{1}{u_{n-1}}\partial_z(\frac{1}{u_{n-1}}\partial_zG)
    =\frac{1}{u_{n-1}}(
    \frac{1}{u_{n-1}}\partial_z^2G- \frac{\partial_zu_{n-1}}{u_{n-1}^2}\partial_zG)
     \leq0\quad \text{at}\quad p,\end{split}
\end{align}
which implies that
 $$LG-cG\geq 0\quad \text{at} \quad p.$$ However,  by \eqref{largeB} and \eqref{fp},
\begin{align*}
    LG-cG=e^{-Bx}Lf+(-B-c)e^{-Bx}f<0\quad \text{at} \quad p,
\end{align*} which leads to a contradiction.
Therefore, $G\leq 0$ in $[0,X]\times[0,Y]\times[0,z_0]$, so is  $f.$
\end{proof}

\begin{lemma}[Maximum Principle in unbounded domain]\label{infMPx}
If $f\leq M$ for some positive constant $M,$ $f\leq 0 $ on $ [0,X]\times[0,Y]\times\{z=0\}\cup\{x=0\}\times[0,Y]\times[0,+\infty)\cup[0,X]\times\{y=0\}\times[0,+\infty)$ and
\begin{align}\label{maxassumpx}
    Lf\leq0 \quad\text{in}\quad(0,X]\times(0,Y]\times(0,\infty),
   \end{align} where $L f=\nabla_\xi f +(1+\tilde{q})\nabla_{\eta}  f+b\nabla_{\psi}f-u_{n}\nabla_{\psi}^2f +cf
$, \eqref{psin-1x2} holds and $b$ and $c$ are bounded functions,
then $f\leq 0$ in $[0,X]\times[0,Y]\times[0,+\infty).$
\end{lemma}
\begin{remark}
As the standard maximum principle for parabolic equation in bounded domains, there is no requirement on the sign of $c.$
\end{remark}
\begin{proof}
Set
\begin{align}\label{hatpsi}
\hat{\psi}_{z}=\min_{(x,y)\in[0,X]\times[0,Y]}\psi_{n-1}(x,y,z), \quad z\in[0,+\infty),
\end{align}
 where $\psi_{n-1}$ is defined in \eqref{psi0125}. Then
 \begin{align}\label{hpsitoinf}
  \hat{\psi}_{z}\rightarrow+\infty, \quad z\rightarrow+\infty,
 \end{align}
because
     \begin{align*}
        \lim_{z\rightarrow+\infty}u_{n-1}(x,y,z)= 1, \quad (x,y)\in[0,X]\times[0,Y].
     \end{align*}

For any  $z_1\in(0,+\infty)$ and any  $z_2\in(z_1,+\infty),$  by the definition \eqref{hatpsi},
$\hat{\psi}_{z_2}$ is a constant such that $$0<\hat{\psi}_{z_2}\leq\psi_{n-1}(x,y,z_2),\quad (x,y)\in[0,X]\times[0,Y].$$ Set
\begin{align*}
    G(x,y,z)=&f(x,y,z)-\frac{M\psi_{n-1}^\alpha(x,y,z)}{\hat{\psi}_{z_2}^\alpha}e^{Bx}
    \quad \text{in} \quad [0,X]\times[0,Y]\times[0,z_2],
\end{align*}
where $\alpha\in(0,1)$ and $B$ is a large positive constant such that $$B>\|c\|_{L^{\infty}( [0,X]\times[0,Y]\times[0,+\infty))}+\frac{1+\|b\|_{L^\infty((0,X]\times(0,Y]\times(0,\infty))}}{\delta}.$$
 Then $G\leq 0 $ on $ \{z=0,z_2\}\cup\{x=0\}\cup\{y=0\}$ and  by \eqref{maxassumpx}, \eqref{4.2}, \eqref{psin-1x1} and \eqref{psin-1x2}, \begin{align*}
   L G(x,y,z)<&Lf(x,y,z)+( \frac{1+\|b\|_{L^\infty((0,X]\times(0,Y]\times(0,\infty))}}{\delta}-c-B)\frac{M\psi_{n-1}^\alpha(x,y,z)}{\hat{\psi}_{z_2}^\alpha}e^{Bx}
   \\ < &0 \quad \text{in} \quad (0,X]\times(0,Y]\times[0,z_2].
\end{align*}
     By Lemma \ref{bdMP},  we have $G\leq 0$ in $[0,X]\times[0,Y]\times[0,z_2].$ Hence $f(x,y,z_1)\leq \frac{M\psi_{n-1}^{\alpha}(x,y,z_1)}{\hat{\psi}_{z_2}^\alpha}e^{Bx}.$ Letting $z_2\rightarrow+\infty,$  we have $f(x,y,z_1)\leq0$  by \eqref{hpsitoinf}.
     Finally, since $z_1$ is arbitrary, we have $f\leq 0$ in $[0,X]\times[0,Y]\times[0,+\infty).$
\end{proof}
\subsection{Barrier functions for $P_1$ and $P_2$}
Corresponding to \eqref{p1uu} and \eqref{1121-6},  we denote
\begin{align}\label{p10106}\begin{split}
    P_1w=&\nabla_\xi w +(1+\tilde{q})\nabla_{\eta}w -\nabla_{\psi}(u_{n}\nabla_{\psi}w)\\ =&\partial_x w +(1+\tilde{q})\partial_{y}w -\tilde{b}\partial_{z} w -\frac{u_{n}}{u_{n-1}^2}\partial_{z}^2w-\frac{1}{u_{n-1}}\partial_z(\frac{u_{n}}{u_{n-1}})\,\partial_{z} w ,\end{split}\end{align}and \begin{align}\label{p20106}
     P_2 w=&\nabla_\xi w +(1+\tilde{q})\nabla_{\eta}  w-u_{n}\nabla_{\psi}^2w,
\end{align}  where
\begin{align}\label{linemethodb}
	\tilde{b}=\frac{\int_0^z \partial_{x}u_{n-1}dz'}{u_{n-1}}+(1+\tilde{q})\frac{\int_0^z \partial_{y}v_{n-1}dz'}{v_{n-1}}.
\end{align}
Note that 
\begin{align}\label{linemethodbsign}
-C\leq \tilde{b}=\frac{\int_0^z \partial_{x}u_{n-1}dz'}{u_{n-1}}+(1+\tilde{q})\frac{\int_0^z \partial_{y}v_{n-1}dz'}{v_{n-1}}\leq  C\min\{z,1\} \quad \text{in} \quad \Omega.
\end{align}Now we give the barrier functions that match the maximum principle for the vector fields.

By \eqref{pphi} in  Subsection \ref{bf} in the Appendix, we have  for positive constants $0<\alpha\ll1,$ $0<\beta<1,$
there exist positive constants $\delta_0$ and $c_2$ independent of $\varepsilon$ such that
\begin{align}\label{1118-3}\begin{split}
   & P_1 \phi_{1,\alpha}\geq c_2(\frac{\alpha(1-\alpha)}{u_{n-1} (z+\epsilon_0)^2}+A)\phi_{1,\alpha},\,\,
     P_2 \phi_{2,\beta}\geq c_2(\beta(1-\beta)\psi_{n-1}^{-\frac{3}{2 } }+A)\phi_{2,\beta},
     \\& P_2 \phi_{2,1}\geq c_2(\alpha\psi_{n-1}^{\alpha-\frac{3}{2 } }+A)\phi_{2,1} \quad  \text{in} \quad(0,X]\times(0,Y^*]\times[0,\delta_0],\\
    &  P_1 \phi_{1,\alpha}\geq c_2A\phi_{1,0} , P_2 \phi_{2,\beta}, P_2\phi_{2,0},  P_2 \phi_{2,1}\geq c_2A\phi_{2,0}  \\& \text{in} \quad(0,X]\times(0,Y^*]\times[0,+\infty)\setminus\{\text{ridges of the barrier functions}\},
\end{split}\end{align}
where $\phi_{1,\alpha}$ and $\phi_{1,0}$ are defined
in \eqref{phi3} and \eqref{phi1}, and for $\beta\in[0,1),$
\begin{equation}\label{phi2beta}\phi_{2,\beta}(x,y,z)= \left\{
  \begin{aligned}
  e^{-\frac{A}{x+1}}(\frac{\psi_{n-1}(x,y,z )}{\sqrt{x+1}})^{\beta},\quad\quad & 0\leq\frac{\psi_{n-1}(x,y,z )}{\sqrt{x+1}}\leq \delta,\\
    e^{-\frac{A}{x+1}}\delta^{\beta},\quad\quad & \delta\leq\frac{\psi_{n-1}(x,y,z )}{\sqrt{x+1}}\leq N ,\\
    e^{-\frac{A}{x+1}}\delta^{\beta}e^{N^2\frac{9}{10}\mu}e^{-\frac{\psi_{n-1}^2(x,y,z )}{x+1}\frac{9}{10}\mu},\quad\quad & \frac{\psi_{n-1}(x,y,z )}{\sqrt{x+1}}\geq N ,
\end{aligned}\right.\\
\end{equation}in particular,
\begin{equation}\label{phi20}\phi_{2,0}(x,y,z)= \left\{
  \begin{aligned}
   e^{-\frac{A}{x+1}},\quad\quad & 0\leq\frac{\psi_{n-1}(x,y,z )}{\sqrt{x+1}}\leq N ,\\
    e^{-\frac{A}{x+1}}e^{N^2\frac{9}{10}\mu}e^{-\frac{\psi_{n-1}^2(x,y,z )}{x+1}\frac{9}{10}\mu},\quad\quad & \frac{\psi_{n-1}(x,y,z )}{\sqrt{x+1}}\geq N ,
\end{aligned}\right.\\
\end{equation} and for  $0<\alpha\ll1,$
\begin{equation}\label{phi4}\phi_{2,1}= \left\{
  \begin{aligned}
  e^{-\frac{A}{x+1}}(\psi_{n-1}-\psi_{n-1}^{1+\alpha}),\quad & 0\leq\psi_{n-1}(x,y,z)\leq \delta,\\
   e^{-\frac{A}{x+1}}\delta(1-\delta^{\alpha}),\quad & \psi_{n-1}(x,y,z)\geq \delta.
\end{aligned}\right.\\
\end{equation}
Here the positive constant $\delta\leq \frac{1}{2}$ is independent of $\varepsilon.$
We should pay  attention on the growth rates in \eqref{1118-3} near $z=0$. In particular, $$\frac{1}{u_{n-1}}\sim\frac{1}{z+\epsilon_0},\,\,\frac{1}{\psi}\sim\frac{1}{z(z+\epsilon_0)}
\quad\text{near}\quad z=0.$$ The order of the growth rates is necessary to control the large terms with certain order in the equations resulted from the degeneracy of Prandtl equation at the boundary $z=0$.

On the other hand, since $\lim_{z\rightarrow+\infty}u_{n-1}=1$, $$\frac{\psi_{n-1}}{z}\rightarrow 1\quad \text{as} \quad z\rightarrow\infty,$$
which characterizes the decay rates of $\phi_{2,\beta}.$

\subsection{Generalized Maximum Principle for functions with ridges}

 Although these barrier functions are not $C^1,$ their left-hand derivatives with respect to $z$, i.e. $\partial_z^-$ are strictly bigger than their right-hand  derivatives with respect to $z$, i.e. $\partial_z^+$ at the ridges so that certain extrema of auxiliary function cannot be attained at the  ridges shown as follows.

 \begin{proposition}\label{26-06-13-2}
  If  a barrier function $\phi\in C^{1}(\R_+\setminus \{z_0\})$ satisfies
 \begin{align}\label{26-04-26}
    \partial_z^-\phi>\partial_z^+\phi\quad \text{at}\quad z_0,
 \end{align} where  $z_0$ is a positive constant,
 then for any function $f\in C^{1}(\R_+)$, $f-\phi$ cannot attain its maximum at $z_0.$
  \end{proposition}
  \begin{proof}
  We employ a proof by contradiction. If $f-\phi$ attains its maximum at $z_0,$
  then
  \begin{align*}
  & \partial_z^+(f-\phi)=\lim_{t\rightarrow0^+}\frac{(f-\phi)|_{z=z_0+t}-(f-\phi)|_{z=z_0}}{t}\leq 0,\\ &\partial_z^-(f-\phi)=\lim_{t\rightarrow0^+}\frac{(f-\phi)|_{z=z_0-t}-(f-\phi)|_{z=z_0}}{-t}\geq 0.
  \end{align*} Hence, $\partial_z^-(f-\phi)\geq\partial_z^+(f-\phi)$ at $z_0$, since $f\in C^{1}(\R_+)$. then $\partial_z^-\phi\leq\partial_z^+\phi$ at $z_0,$
 which contradicts to \eqref{26-04-26}. The proof is completed.

  \end{proof}

Based on this observation, we can modify the proof and establish generalized  maximum principles in terms of vector fields for barrier functions with ridges as shown in Example \ref{1118-4}. cf.  Serrin \cite{Serrin}.
When we use these barrier functions, we will further divide $\Omega$ into two domains $\Omega_s=\{0\leq z\leq \delta_\varepsilon\}\cap\Omega$
and $\Omega_b=\{ z\geq \delta_\varepsilon\}\cap\Omega$
where $\delta_\varepsilon=\varepsilon^m\ll \delta$, $m\in (0,+\infty)$. In particular, we take $A>C\varepsilon^{-2m}$ to control the remainder in $\Omega_b=\{ z\geq \delta_\varepsilon\}\cap\Omega.$

{\centering
\begin{tikzpicture}[scale=9]

\draw[->,thick] (-0.1,0) -- (0.7,0) node[right] {$\psi_{n-1}$};
\draw[->,thick] (0,-0.1) -- (0,0.3) node[above] {$\phi_{2,1}$};

\draw[domain=0:0.3] plot (\x,{(\x-\x*\x)}) ;

\draw[domain=0.3:0.6] plot (\x,{0.3-0.09}) ;
\draw[dotted] (0.3,0.21) -- (0.3,0) node[below]{$\delta$};
\draw[dotted] (0.1,0.09) -- (0.1,0) node[below]{$\varepsilon^{2m}$};
\end{tikzpicture}
\begin{center}
	Figure 1
\end{center}
\par}

\vspace{0.2cm}

Here is a simple example for better explanation.
\begin{example}\label{1118-4} If a smooth function $f$ satisfies
$L_{simple} f\leq \phi_{2,0}$ in $\Omega$, then $f-\varepsilon^{6.5}\phi_{2,1}$ cannot attain its maximum in  $\Omega,$
 where $ L_{simple}=\nabla_\xi -\nabla_{\psi}^2.$
\end{example}
\begin{proof} First, auxiliary function $f-\varepsilon^{6.5}\phi_{2,1}$ cannot attain its maximum
on the ridge $\{\psi_{n-1}(x,y,z)=\delta\}$ by Proposition \ref{26-06-13-2}, since $\partial_z^-\phi_{2,1}>\partial_z^+\phi_{2,1}$ at this ridge by noting $\partial_z \psi_{n-1}=u_{n-1}$. Next,
we claim \begin{align*}
    L_{simple}(\varepsilon^{6.5}\phi_{2,1})> \phi_{2,0} \quad \text{in}\quad \Omega\setminus\{\psi_{n-1}(x,y,z)=\delta\}.
\end{align*}
Then $ L_{simple}(f-\varepsilon^{6.5}\phi_{2,1})<0$ in $\Omega\setminus\{\psi_{n-1}(x,y,z)=\delta\}.$ Hence, $f-\varepsilon^{6.5}\phi_{2,1}$ cannot attain its maximum in $\Omega\setminus\{\psi_{n-1}(x,y,z)=\delta\}$ by the same argument in the proof of the maximum principles in terms of vector fields.

We now prove the Claim as follows. See Figure 1.
Take
\begin{align*}
\Omega_s=&\{0\leq z\leq \varepsilon^m\}\cap\Omega\subset \{0\leq \psi_{n-1}(x,y,z)\leq \varepsilon^{2m}C\},\\
\Omega_b=&\{ z\geq  \varepsilon^m\}\cap\Omega\subset \{ \psi_{n-1}(x,y,z)\geq \varepsilon^{2m}c\},
\end{align*}
 where $c$ and $C $ are independent of $\varepsilon.$
 Then by the definition in \eqref{phi4}, we have
\begin{equation}L_{simple}\phi_{2,1}= \left\{
  \begin{aligned}e^{-\frac{A}{x+1}}\alpha(1+\alpha)\psi_{n-1}^{\alpha-1}+
 \frac{ A}{(1+x)^2}e^{-\frac{A}{x+1}}(\psi_{n-1}-\psi_{n-1}^{1+\alpha}),\quad & \psi_{n-1}(x,y,z)\leq \delta,\\
  \frac{A}{(1+x)^2} e^{-\frac{A}{x+1}}\delta(1-\delta^{\alpha}),\quad & \psi_{n-1}(x,y,z)\geq \delta.
\end{aligned}\right.\\
\end{equation}
Then
\begin{align*}
    L_{simple}(\varepsilon^{6.5}\phi_{2,1})>  e^{-\frac{A}{x+1}} \quad \text{in}\quad \Omega_s
\end{align*}
by taking $m>\frac{6.5}{2(1-\alpha)}$ so that $\varepsilon^{2m(1-\alpha)}<\varepsilon^{6.5}$ with $\varepsilon$ being sufficiently small. Since for $\delta$ small independent of $\varepsilon,$ we have
$$\psi_{n-1}-\psi_{n-1}^{1+\alpha}\geq \frac{1}{2}\psi_{n-1}\quad  \text{for} \quad \psi_{n-1}(x,y,z)\leq \delta,$$  for $A\geq C\varepsilon^{-2m-6.5}$  with $C$ being  independent of $\varepsilon,$ we have
\begin{align*}
    L_{simple}(\varepsilon^{6.5}\phi_{2,1})\geq \varepsilon^{6.5} \frac{A}{(x+1)^2}e^{-\frac{A}{x+1}}\frac{c}{2} \varepsilon^{2m}> \phi_{2,0} \quad \text{in} \quad\Omega_b\setminus\{\psi_{n-1}(x,y,z)=\delta\}.
\end{align*}
\end{proof}
\section{Estimates on $u_n$ and first order derivatives}\label{s4}

In this section, from time to time we  use $w_n$ to stand for $u_n$  in order to have a better presentation  of some equations. Since the methods  for estimating $u_n, v_n$ and their derivatives are the same,   we only present the  estimation on  $u_n$ and its derivatives.

\subsection{Estimates on  $u_n,\partial_z u_n$ }
\subsubsection{Equations}
It holds that
\begin{align}\label{qn0}\begin{split}
   P_1 w_n= \nabla_{\xi} w_n +(1+\tilde{q})\nabla_{\eta}w_n  -\nabla_{\psi}(u_{n}\nabla_{\psi}w_n)
      =0 .\end{split}\end{align}
Then
\begin{align}\label{dpsi3un}\begin{split}
\frac{1}{u_n }(\nabla_{\xi} u_n +(1+\tilde{q})\nabla_{\eta}u_n  -(\nabla_{\psi}u_{n})^2)
      =&\nabla_{\psi}^2u_{n},\\
      \nabla_{\psi}\big(\frac{1}{u_n }(\nabla_{\xi} u_n +(1+\tilde{q})\nabla_{\eta}u_n  -(\nabla_{\psi}u_{n})^2)\big)
      =&\nabla_{\psi}^3u_{n},\end{split}
\end{align}
and
\begin{align}\label{ppun2}
 P_2 u_n^2= \nabla_{\xi} u_{n}^2 +(1+\tilde{q})\nabla_{\eta}u_{n}^2  -u_{n}\nabla_{\psi}^2u_{n}^2
      =0 ,\end{align}
  which implies that
\begin{align}\label{0323-446}\begin{split}
0= & \nabla_{\xi} \nabla_{\psi}u_{n}^2 +(1+\tilde{q})\nabla_{\eta}\nabla_{\psi}u_{n}^2
  -u_{n}\nabla_{\psi}^2(\nabla_{\psi}u_{n}^2)
   -(\nabla_{\psi}u_{n})\nabla_{\psi}(\nabla_{\psi}u_{n}^2)  \\
   & -\nabla_{\eta}\tilde{q}\nabla_{\psi}u_{n}^2
   +\nabla_{\psi}\tilde{q}\nabla_{\eta}u_{n}^2 \\
   = & \nabla_{\xi} \nabla_{\psi}u_{n}^2 +(1+\tilde{q})\nabla_{\eta}\nabla_{\psi}u_{n}^2
  -u_{n}\nabla_{\psi}^2(\nabla_{\psi}u_{n}^2)
   -(\nabla_{\psi}u_{n}^2)\frac{1}{2u_{n}}\nabla_{\psi}^2u_{n}^2  \\
   & -\nabla_{\eta}\tilde{q}\nabla_{\psi}u_{n}^2
   +\nabla_{\psi}\tilde{q}\nabla_{\eta}u_{n}^2 .\end{split}\end{align}
 On the other hand, since  \begin{align}\label{0823}
 0= &\nabla_{\tau_1}    \bar{u}^2+   \nabla_{\tau_1} \bar{u}^2- \bar{u}\nabla_{n}^2  \bar{u}^2,
\end{align} we have
\begin{align}\label{u-baru}
  \nabla_{\xi} \big(u_{n}^2-\bar{u}^2 \big) +(1+\tilde{q})\nabla_{\eta} \big(u_{n}^2-\bar{u}^2 \big) -u_{n}\nabla_{\psi}^2 \big(u_{n}^2-\bar{u}^2 \big)
      =R_0 ,\end{align}
      where  \begin{align}\label{r0-0106}
       R_0:=&
       (2\nabla_{\tau_1}  -\nabla_{\xi}-(1+\tilde{q})\nabla_{\eta} ) \bar{u}^2+(-\bar{u}\nabla_{n}^2+u_{n}\nabla_{\psi}^2 )\bar{u}^2
       \\=&
       (\nabla_{\tau_1}+\nabla_{\tau_2}  -\nabla_{\xi}-(1+\tilde{q})\nabla_{\eta} ) \bar{u}^2+(-\bar{u}\nabla_{n}^2+u_{n}\nabla_{\psi}^2 )\bar{u}^2,\end{align}
 and
\begin{align*}
 0= &\nabla_{\tau_1}  \nabla_{n}  \bar{u}^2+   \nabla_{\tau_1} \nabla_{n} \bar{u}^2-\nabla_{n} \bar{u}\nabla_{n}^2  \bar{u}^2- \bar{u}\nabla_{n}^3 \bar{u}^2.
\end{align*} Then
\begin{align*}
 0= &\nabla_{\xi}\nabla_{n}\bar{u}^2 +(1+\tilde{q})\nabla_{\eta} \nabla_{n}\bar{u}^2-u_{n} \nabla_{\psi}^2 (\nabla_{n}\bar{u}^2)
 -(\nabla_{n} \bar{u}^2)\frac{1}{2\bar{u}}\nabla_{\psi}^2u_n^2
 \\
 & +(2\nabla_{\tau_1}  -\nabla_{\xi}-(1+\tilde{q})\nabla_{\eta} ) \nabla_{n}\bar{u}^2+(-\bar{u}\nabla_{n}^2+u_{n}\nabla_{\psi}^2 )\nabla_{n}\bar{u}^2 \\&+(\nabla_{n} \bar{u}^2)\frac{1}{2\bar{u}}(\nabla_{\psi}^2u_n^2-\nabla_{n}^2\bar{u}^2) ,
\end{align*}
which implies that
 \begin{align}\label{6.7}\begin{split}
\nabla_{\xi} (\nabla_{\psi}u_{n}^2 -\nabla_{n}\bar{u}^2) +(1+\tilde{q})\nabla_{\eta}(\nabla_{\psi}u_{n}^2 -\nabla_{n}\bar{u}^2)
  -u_{n}\nabla_{\psi}^2(\nabla_{\psi}u_{n}^2 -\nabla_{n}\bar{u}^2)
  = R_1,\end{split}\end{align}
 where
\begin{align*}
  R_1=&\nabla_{\eta}\tilde{q}\nabla_{\psi}u_{n}^2
   -\nabla_{\psi}\tilde{q}\nabla_{\eta}u_{n}^2
   \\& +(2\nabla_{\tau_1}  -\nabla_{\xi}-(1+\tilde{q})\nabla_{\eta} ) \nabla_{n}\bar{u}^2+(-\bar{u}\nabla_{n}^2+u_{n}\nabla_{\psi}^2 )\nabla_{n}\bar{u}^2 \\&+(\nabla_{n} \bar{u}^2)\frac{1}{2\bar{u}}(\nabla_{\psi}^2u_n^2-\nabla_{n}^2\bar{u}^2) +\nabla_{\psi}^2u_{n}^2( \frac{1}{2u_{n}} \nabla_{\psi}u_{n}^2-\frac{1}{2\bar{u}}\nabla_{n}\bar{u}^2)
   .
\end{align*} Here we note $(\nabla_{\tau_1}  -(1+\tilde{q})\nabla_{\eta} ) \nabla_{n}\bar{u}^2=(\nabla_{\tau_2}  -(1+\tilde{q})\nabla_{\eta} ) \nabla_{n}\bar{u}^2.$
\subsubsection{Estimates on $u_n-\bar{u}$}
First, we estimate $ | u_{n}^2-\bar{u}^2 |$ as follows.
\begin{lemma}\label{xz4} For any  constant $\beta\in(\frac{1}{2},1),$ it holds
 \begin{align*}
    | u_{n}^2-\bar{u}^2 |\leq \varepsilon^{8} C{\epsilon_0}^3\phi_{2,0}+\varepsilon^7 \phi_{2,\beta} \quad \text{in} \quad \Omega\cap\{0\leq y\leq Y^*\},
 \end{align*}
 where
 $\phi_{2,0}$ and $\phi_{2,\beta}$ are defined in \eqref{phi20}-\eqref{phi2beta}. \end{lemma}

 \begin{proof} Our goal is to prove $g=u_{n}^2-\bar{u}^2-\varepsilon^{8} C{\epsilon_0}^3\phi_{2,0} -\varepsilon^7 \phi_{2,\beta}\leq 0$ in $\Omega\cap\{0\leq y\leq Y^*\}$ by generalized maximum principle for functions with ridges.

\textit{ Step 1} We will prove  $g$ can not attain positive maximum on $\Omega\cap\{0\leq y\leq Y^*\}\setminus \{\text{ridges}\}.$

\textit{We will prove by contradiction.}
 By \eqref{1118-3} and \eqref{u-baru}, we have in $\Omega\cap\{0\leq y\leq Y^*\}$,
 \begin{align*}
  P_2 \big(u_{n}^2-\bar{u}^2-\varepsilon^{8} C{\epsilon_0}^3\phi_{2,0} \big)
      \leq R_0 ,\end{align*}
  where  $R_0 $ is defined in \eqref{r0-0106}.  By \eqref{gg}, \eqref{cy} and \eqref{cy-1} in Lemma \ref{6100205}, we have in $\Omega\cap\{0\leq y\leq Y^*\}$,\begin{align}
  |(\nabla_{\tau_1}+\nabla_{\tau_2}  -\nabla_{\xi}-(1+\tilde{q})\nabla_{\eta} ) \bar{u}^2+(-\bar{u}\nabla_{n}^2+u_{n}\nabla_{\psi}^2 )\bar{u}^2|\leq C\bar{u}^2\phi_{1,0}+C\phi_{1,0}.
\end{align}
This implies  $|R_0|\leq C\phi_{1,0}$ where we note that by \eqref{26-06-10-ncg},
\begin{align}\label{lb0126g}
c_0\epsilon_0\leq u_{n-1}\quad \text{in}\quad\Omega\cap\{0\leq y\leq Y^*\}.
\end{align}
        Then by \eqref{1118-3} and taking $A\gg \varepsilon^{-7}$,
        it holds
      $$P_2 (u_{n}^2-\bar{u}^2-\varepsilon^{8} C{\epsilon_0}^3\phi_{2,0} -\varepsilon^7 \phi_{2,\beta} )<0\quad \text{in} \quad \Omega\cap\{0\leq y\leq Y^*\}\setminus \{\frac{\psi_{n-1}(x,y,z)}{\sqrt{x+1}}=\delta, N\},$$
       as explained in Example \ref{1118-4}.

        However, if $g$ attains its positive maximum at some point $p_0\in (0,X]\times(0,Y^*]\times(0,+\infty),$ then by the argument in the  proof of maximum principle, we have $$P_2(g)|_{p_0}\geq 0,$$
        which leads to a contradiction.
        Then we complete step 1.

\textit{ Step 2} We will prove    $g$ can not attain positive maximum on  ridges.

      $g$ cannot attain its positive maximum on the ridges $\{\frac{\psi_{n-1}(x,y,z)}{\sqrt{x+1}}=\delta, N\}$ by Proposition \ref{26-06-13-2}, since for any fixed $x$, $$\partial_z^-(\varepsilon^{8} C{\epsilon_0}^3\phi_{2,0}+\varepsilon^7 \phi_{2,\beta})>\partial_z^+(\varepsilon^{8} C{\epsilon_0}^3\phi_{2,0}+\varepsilon^7 \phi_{2,\beta})\quad\text{at the ridges}.$$

    Combining Step 1 and Step 2,   $g$ can not attain positive maximum on  $\Omega\cap\{0\leq y\leq Y^*\}$.

 Next, by  the
       boundary condition including the compatible condition \eqref{flw0106}, i.e.
      \begin{align*}
       |u_{n}^2-\bar{u}^2|\leq \varepsilon^{8} C{\epsilon_0}^3\phi_{2,0}\quad \text{at}\quad \{z=0\}\cap\overline{\Omega},
      \end{align*} we have $g\leq 0$ on the boundary.

      In summary,  by the argument in the proof of Lemma \ref{infMPx}[Maximum Principle in unbounded domain], we have
      $g=u_{n}^2-\bar{u}^2-\varepsilon^{8} C{\epsilon_0}^3\phi_{2,0} -\varepsilon^7 \phi_{2,\beta}\leq 0$ in $\Omega\cap\{0\leq y\leq Y^*\}$.

      Similarly, we can prove
      $$[-(u_{n}^2-\bar{u}^2)-\varepsilon^{8} C{\epsilon_0}^3\phi_{2,0}] -\varepsilon^7 \phi_{2,\beta}\leq 0\quad \text{in} \quad \Omega\cap\{0\leq y\leq Y^*\}.$$
 \end{proof}
 Based on Lemma \ref{xz4}, we can refine the decay estimate of $u_n-\bar{u}$ in $z$ for $z$ large given in the following lemma.
 \begin{lemma}\label{xz40316}It holds
 \begin{align*}
    | u_{n}^2-\bar{u}^2 |\leq\varepsilon^{6.5} \phi_{1,0} \quad \text{in} \quad \Omega\cap\{0\leq y\leq Y^*\}.
 \end{align*}
 \end{lemma}
 \begin{proof}
 By Lemma \ref{xz4} and the prescribed boundary conditions on $\{x=0\}\cup\{y=0\},$
 $$\pm(u_{n}^2-\bar{u}^2)-\varepsilon^{6.5} \phi_{1,0}\leq 0\quad \text{on}\quad [0,X]\times[0,Y^*]\times[N,+\infty)\cap(\{x= 0\}\cup\{y=0\}\cup\{z= N\}),$$
 for small $\varepsilon.$ Moreover, by taking $A>>\varepsilon^{-6.5}$, \eqref{u-baru} and \eqref{0323-p2phi}, we have
 \begin{align*}
 P_2\big(\pm(u_{n}^2-\bar{u}^2)-\varepsilon^{6.5} \phi_{1,0}\big)
      =R_0 -\frac{1}{2}A\varepsilon^{6.5} \phi_{1,0}<0\quad \text{in}\quad (0,X]\times(0,Y^*]\times(N,+\infty)\end{align*}
where $|R_0|\leq C\phi_{1,0}$ by the estimates in the proof of Lemma \ref{xz4}.
Then by applying the maximum principle in the domain $ [0,X]\times[0,Y^*]\times[N,+\infty),$ we have the desired result.

 \end{proof}

Moreover, based on Lemma \ref{xz4}, we can refine the growth estimate of $u_n-\bar{u}$ in $z$ near $z=0$ given in the following theorem.

\begin{theorem} \label{thm1} For some constant  $\alpha\in(0,\frac{1}{4})$,  it holds
 \begin{align}\label{sm1}
    | u_{n}^2-\bar{u}^2 |\leq  \varepsilon^{8} C{\epsilon_0}^3\phi_{2,0}+ \varepsilon^{6.5} \phi_{2,1}\psi_{n-1}^\alpha \quad \text{in} \quad \Omega\cap\{0\leq y\leq Y^*\}\cap\{0\leq z\leq \delta_4\},
 \end{align}  for some small positive constant $\delta_4 $, which implies
 \begin{align}\label{0821jl}
    | u_{n}-\bar{u} |\leq \frac{d_0\varepsilon^6}{4} \phi_{1,1+2\alpha} \quad \text{in} \quad \Omega\cap\{0\leq y\leq Y^*\},
 \end{align} by  Lemma \ref{xz40316},
 where   $d_0$  is defined in \eqref{assumpn-2} and $\phi_{2,1}$ in \eqref{phi4}.

 \end{theorem}
 \begin{proof} For convenience, we take $\delta_4$ small enough such that $\psi_{n-1}(x,y,\delta_4)\leq \delta$ where  $\delta$ and $\phi_{2,1}$ are defined in \eqref{phi4}.

 First, we will prove by contradiction that
  \begin{align}\label{sm10126step1}
     u_{n}^2-\bar{u}^2 \leq  \varepsilon^{8} C{\epsilon_0}^3\phi_{2,0}+ \varepsilon^{6.5} \phi_{2,1}\psi_{n-1}^\alpha \quad \text{in} \quad \Omega\cap\{0\leq y\leq Y^*\}\cap\{0\leq z\leq \delta_4\}.
 \end{align}

  If \eqref{sm10126step1} holds in $\Omega\cap\{0\leq y\leq Y^*\}\cap\{0\leq z\leq \delta_4\},$ then it is proved.

  Otherwise, there exists an interior point $p_0\in(0,X]\times(0,Y^*]\times(0,\delta_4)$ such that
 $(u_{n}^2-\bar{u}^2-\varepsilon^{8} C{\epsilon_0}^3\phi_{2,0})\psi_{n-1}^{-\alpha}-\varepsilon^{6.5} \phi_{2,1}$ attains its positive maximum in $ \Omega\cap\{0\leq y\leq Y^*\}\cap\{0\leq z\leq \delta_4\}$ at $p_0,$  since
 \eqref{sm1} holds on $x=0,$ $y=0$ and $ z=0,\delta_4,$  by the boundary data and Lemma \ref{xz4}.
Then by applying the argument used for  \eqref{xzx}, we have
\begin{align}\label{p0sign}
P_2[(u_{n}^2-\bar{u}^2-\varepsilon^{8} C{\epsilon_0}^3\phi_{2,0})\psi_{n-1}^{-\alpha}-\varepsilon^{6.5} \phi_{2,1}](p_0)\geq 0,
\end{align} and
\begin{align}\label{pp0}
 u_{n}^2-\bar{u}^2 - \varepsilon^{8} C{\epsilon_0}^3\phi_{2,0}>0\quad \text{at}\quad p_0.
\end{align}
 On the other hand,  for small $\varepsilon,$ by \eqref{xz2}, \eqref{xz3}, \eqref{psin-1} and \eqref{pp0}, we have
  \begin{align*}
  P_2 \big((u_{n}^2-\bar{u}^2 -\varepsilon^{8} C{\epsilon_0}^3\phi_{2,0})\psi_{n-1}^{-\alpha}\big)
      \leq& |R_0|\psi_{n-1}^{-\alpha}-\alpha(\alpha+1)(u_{n}^2-\bar{u}^2- \varepsilon^{8} C{\epsilon_0}^3\phi_{2,0} ) u_n \psi_{n-1}^{-\alpha-2}
      \\&+2\alpha u_n \nabla_\psi (u_{n}^2-\bar{u}^2 -\varepsilon^{8} C{\epsilon_0}^3\phi_{2,0}) \psi_{n-1}^{-\alpha-1}\\&+\varepsilon^2C(u_{n}^2-\bar{u}^2- \varepsilon^{8} C{\epsilon_0}^3\phi_{2,0} )\psi_{n-1}^{-\alpha}
      \\ \leq& C\psi_{n-1}^{-\alpha-1}u_{n-1}^2\phi_{1,0}\quad \text{at}\quad p_0.\end{align*}
In fact, we have used the following:
 \begin{align*}
 &|\nabla_\eta \psi_{n-1}^{-\alpha}|\leq \varepsilon^2C \psi_{n-1}^{-\alpha},
 \quad\epsilon_0\leq Cu_{n-1}\quad \text{in } \quad \Omega\cap\{0\leq y\leq Y^*\}\cap\{0\leq z\leq \delta_4\} ,
  \\&P_2 \big(u_{n}^2-\bar{u}^2-\varepsilon^{8} C{\epsilon_0}^3\phi_{2,0} \big)
      \leq R_0\quad \text{in } \quad \Omega\cap\{0\leq y\leq Y^*\}\cap\{0\leq z\leq \delta_4\} ,\end{align*}
   with
 \begin{align}\label{xz3}
 |R_0|\leq C\phi_{1,0}
 \end{align}
from the proof of Lemma \ref{xz4}.
      Then by \eqref{1118-3} and argument as explained in Example \ref{1118-4},  taking $A\gg \varepsilon^{-7}$ and $\delta_4$ sufficiently small,  for  small $\alpha$ with   $2\alpha<\frac{1}{2}$, it holds that
      $$P_2 \big((u_{n}^2-\bar{u}^2-\varepsilon^{8} C{\epsilon_0}^3\phi_{2,0} )\psi_{n-1}^{-\alpha} -\varepsilon^{6.5}  \phi_{2,1} \big)(p_0)<0,$$  which contradicts to \eqref{p0sign}. Therefore,
      $(u_{n}^2-\bar{u}^2-\varepsilon^{8} C{\epsilon_0}^3\phi_{2,0})-\varepsilon^{6.5}  \phi_{2,1}\psi_{n-1}^{\alpha}\leq 0$  in $\Omega\cap\{0\leq y\leq Y^*\}\cap\{0\leq z\leq \delta_4\}$.

      Similarly, we can prove
      \begin{align}
    -( u_{n}^2-\bar{u}^2 )\leq  \varepsilon^{8} C{\epsilon_0}^3\phi_{2,0}+ \varepsilon^{6.5} \phi_{2,1}\psi_{n-1}^\alpha \quad \text{in} \quad \Omega\cap\{0\leq y\leq Y^*\}\cap\{0\leq z\leq \delta_4\}.
 \end{align}
 Hence, \eqref{sm1} holds in $\Omega\cap\{0\leq y\leq Y^*\}\cap\{0\leq z\leq \delta_4\}.$

      Finally,  \eqref{psin-1un-1} and \eqref{lb0126g} give \eqref{0821jl}.
 \end{proof}

\begin{remark}\label{rk4.2}We can refine the size of growth rate of $u_k-\bar{u}$ in $z$ near $z=0$ for each $k=1,\cdots, n-1$ as given in  \eqref{0821jl}.
\end{remark}

\subsubsection{Estimates on $\partial_z u_n$}

 \begin{theorem}\label{dzunphi0} It holds that
 \begin{align}\label{changchun0106}\begin{split}
    |\frac{u_n}{u_{n-1}}\partial_z u_n-\partial_z\bar{u}|\leq & \frac{d_0\varepsilon^6}{4} \phi_{2,1}+C\varepsilon^7{\epsilon_0}\phi_{2,0},\quad \text{in} \quad \Omega\cap\{0\leq y\leq Y^*\},\\
     |\frac{u_n}{u_{n-1}}\partial_z u_n-\partial_z\bar{u}|\leq & \frac{d_0\varepsilon^6}{4} \phi_{1,0},\quad \text{in} \quad \Omega\cap\{0\leq y\leq Y^*\},\end{split}
 \end{align}
where  $\phi_{2,1} $ is  defined in \eqref{phi4}. Hence,
  \begin{align}\label{1012}
    |\partial_z u_n-\partial_z\bar{u}|\leq \frac{d_0\varepsilon^5}{4} \phi_{1,\alpha} \quad \text{in} \quad \Omega\cap\{0\leq y\leq Y^*\},
 \end{align}
 by induction assumption, Remark \ref{rk4.2} and \eqref{0821jl}.
 \end{theorem}
 \begin{proof}By \eqref{6.7},
  \begin{align}\begin{split}
 P_2(\nabla_{\psi}u_{n}^2 -\nabla_{n}\bar{u}^2)
  = R_1.\end{split}\end{align}

  \textit{Step 1}: We will prove $| R_1|\leq C\phi_{1,0}.$

\noindent
 Note that
 $$\frac{1}{2}\nabla_{\psi}u_{n}^2=\frac{u_{n}}{u_{n-1}}\partial_z u_{n}.$$
 By induction assumption and the definition of $Y^*$, we have
 \begin{align*}
&| \nabla_{\eta}\tilde{q}|\leq \varepsilon^2 e^{-\frac{A}{x+1}}C,\quad |\nabla_{\psi}\tilde{q}|
 = |\frac{\partial_z\tilde{q}}{u_{n-1}}|\leq\frac{\varepsilon^5C}{u_{n-1}^2}\phi_{1,\alpha},
 \\&|\nabla_{\eta}u_{n}^2|\leq Cu_{n-1}\min\{z+\epsilon_0,1\}e^{-\frac{(z+{\epsilon_0})^2}{x+1}\mu}.
 \end{align*}
Then by \eqref{gg} in Appendix,
\begin{align*}
|\nabla_{\eta}\tilde{q}\nabla_{\psi}u_{n}^2
   -\nabla_{\psi}\tilde{q}\nabla_{\eta}u_{n}^2
    +(\nabla_{\tau_1} + \nabla_{\tau_2}-\nabla_{\xi}-(1+\tilde{q})\nabla_{\eta} ) \nabla_{n}\bar{u}^2
    | \leq C \phi_{1,0}.
\end{align*}
Since $ \partial_z\nabla_{n}\bar{u}^2=2\partial_z^2\bar{u}$,
by \eqref{nnbaru}, $|\partial_z\nabla_{n}^2 \bar{u}^2|\leq Ce^{-\frac{(z+{\epsilon_0})^2}{x+1}\mu}$ and $|\frac{2}{\bar{u}}\partial_z^2\bar{u}|=|\nabla_{n}^2 \bar{u}^2|\leq C\bar{u}e^{-\frac{(z+{\epsilon_0})^2}{x+1}\mu}.$ Hence, by \eqref{cy}
$$|(-\bar{u}\nabla_{n}^2+u_{n}\nabla_{\psi}^2 )\nabla_{n}\bar{u}^2 |\leq C\phi_{1,0}.$$
Then by \eqref{ppun2}, induction assumption and definition of $Y^*,$ it holds that
\begin{align*}
   | \nabla_{\psi}^2u_n^2-\nabla_{n}^2\bar{u}^2|=|
   2
  \nabla_{\xi} u_{n}+2(1+\tilde{q})\nabla_{\eta}u_{n}
     - 4\nabla_{\tau_1}\bar{u}|\leq \varepsilon^2 C\phi_{1,1}.
\end{align*}
Thus, we have
\begin{align*}
|(\nabla_{n} \bar{u}^2)\frac{1}{2\bar{u}}(\nabla_{\psi}^2u_n^2-\nabla_{n}^2\bar{u}^2)|
\leq |\frac{2\partial_z\bar{u}}{2\bar{u}}(\nabla_{\psi}^2u_n^2-\nabla_{n}^2\bar{u}^2)|
\leq \varepsilon^2  C\phi_{1,0}. \end{align*}
By \eqref{ppun2}, $|\nabla_{\psi}^2u_{n}^2|\leq Cu_{n-1}e^{-\frac{(z+{\epsilon_0})^2}{x+1}\mu}.$ Then
\begin{align*}
   | \nabla_{\psi}^2u_{n}^2( \frac{1}{2u_{n}} \nabla_{\psi}u_{n}^2-\frac{1}{2\bar{u}}\nabla_{n}\bar{u}^2)|\leq  Cu_{n-1}
  | \frac{1}{u_{n-1}}\partial_zu_{n}-\frac{1}{\bar{u}}\partial_z\bar{u}|\leq C\phi_{1,0}.
\end{align*}
In summary, $|R_1|\leq C\phi_{1,0},$ where  $R_1$ is defined in \eqref{6.7}.

 \textit{Step 2}: We will derive \eqref{changchun0106}.

 Employing \eqref{0323-p2phi} and taking $A\gg \varepsilon^{-6}$, by a similar argument on \eqref{6.7} as in the previous proof of Lemma \ref{xz4}
 and Lemma   \ref{xz40316}, we have
 \begin{align*}
    \frac{1}{2} | \nabla_{\psi}u_{n}^2-\nabla_{n}\bar{u}^2|=|\frac{u_n}{u_{n-1}}\partial_z u_n-\partial_z\bar{u}|\leq& \frac{d_0\varepsilon^6}{4} \phi_{2,1}+C\varepsilon^7\epsilon_0\phi_{2,0}\quad \text{in} \quad \Omega\cap\{0\leq y\leq Y^*\},\\
     \frac{1}{2} | \nabla_{\psi}u_{n}^2-\nabla_{n}\bar{u}^2|=|\frac{u_n}{u_{n-1}}\partial_z u_n-\partial_z\bar{u}|\leq& \frac{d_0\varepsilon^6}{4} \phi_{1,0}\quad \text{in} \quad \Omega\cap\{0\leq y\leq Y^*\}.
 \end{align*}
 This completes the proof of  the theorem.
 \end{proof}

\begin{remark}\label{jili}
	The first estimate  in \eqref{changchun0106} shows the growth rates in $z$ near $z=0$ and the second line of \eqref{changchun0106} shows the  decay rates in $z$ near infinity.
	Here is a direct consequence which will be used later.
	\begin{align}\label{1026}
		|\partial_z u_n-\partial_z\bar{u}|\leq \varepsilon^4C\phi_{1,0}\quad \text{in}\quad \Omega\cap\{0\leq y\leq Y^*\}\cap\{z\geq (d_0\varepsilon)^{\frac{1}{1-\alpha}}\}.
	\end{align}
\end{remark}

\begin{theorem}
 It holds
 \begin{align}\label{mon0315dj}
     \partial_z u_n\geq 2c_0e^{-\frac{3}{2}\mu (z+\epsilon_0)^2}\quad \text{in}\quad \Omega\cap\{0\leq y\leq Y^*\},
 \end{align}
 where $c_0$ defined in \eqref{n-1assumption} is a positive constant independent of $\varepsilon.$
 \end{theorem}
\begin{proof}
First, by \eqref{0323-446} and \eqref{p10106}, \begin{align}\begin{split}
P_1(\nabla_{\psi}u_{n}^2)
   =\bar{R}
   ,\end{split}\end{align} where $$|\bar{R}|=| -\nabla_{\eta}\tilde{q}\nabla_{\psi}u_{n}^2
   +\nabla_{\psi}\tilde{q}\nabla_{\eta}u_{n}^2| \leq \varepsilon^2 C\phi_{1,0}e^{-\mu \frac{(z+\epsilon_0)^2}{1+x}}\quad\text{in}\quad(0,X]\times(0,Y^*]\times[0,\infty).$$
Take $N_1$ large depending on $N$ such that
\begin{align}\label{N10316-1}
 |\frac{2u_n}{u_{n-1}}-2 |<\frac{1}{10}
   \quad\text{in}\quad(0,X]\times(0,Y^*]\times(N_1,\infty),
\end{align} and
\begin{align}\label{N10316}
 - c_0e^{-\frac{3}{2}\mu (z+\epsilon_0)^2}+|\bar{R}|<0
   \quad\text{in}\quad(0,X]\times(0,Y^*]\times(N_1,\infty),
\end{align}
where $N$ is the constant in \eqref{phi3}.
Then  by \eqref{N10316} and Lemma \ref{mon-3-15},
\begin{align}\label{0316intsign}
    P_1(5c_0e^{-\frac{3}{2}\mu(z+\epsilon_0)^2}-\nabla_{\psi}u_{n}^2 )<0 \quad\text{in}\quad(0,X]\times(0,Y^*]\times(N_1,\infty).
   \end{align}

   Second,  we analyze the boundary data. by \eqref{1012}, for $\varepsilon$ small depending on $N_1,$
   \begin{align}\label{0316psiz}
      \nabla_\psi u_n^2=\frac{2u_n}{u_{n-1}}\partial_z u_n
   \end{align} satisfies
    \begin{align}\label{mon0315bd}
     \nabla_\psi u_n^2\geq 5c_0e^{-\frac{3}{2}\mu (z+\epsilon_0)^2}\quad \text{on}\quad \Omega\cap\{0\leq y\leq Y^*\}\cap\{z= N_1\},
 \end{align}  and
   \begin{align}\label{mon0315bd-2-1}
    \partial_z u_n\geq 2c_0e^{-\frac{3}{2}\mu (z+\epsilon_0)^2}\quad \text{in}\quad \Omega\cap\{0\leq y\leq Y^*\}\cap\{z\leq N_1\},
 \end{align} with $c_0$ being small enough.
 Moreover, by \eqref{mnt} and \eqref{bddata},
 \begin{align}\label{djdl0316ch}
 \nabla_\psi u_n^2\geq 5c_0e^{-\frac{3}{2}\mu (z+\epsilon_0)^2}\,\,\text{on} \,\,[0,X]\times[0,Y^*]\times[N_1,+\infty)\cap(\{x= 0\}\cup\{y=0\}).\end{align}
Then by \eqref{mon0315bd} and \eqref{djdl0316ch},   $
     \nabla_\psi u_n^2\geq 5c_0e^{-\frac{3}{2}\mu (z+\epsilon_0)^2}$ on
     $  [0,X]\times[0,Y^*]\times[N_1,+\infty)\cap(\{x= 0\}\cup\{y=0\}\cup\{z= N_1\}).$

Finally, by the boundary condition, \eqref{0316intsign} and applying the maximum principle in $[0,X]\times[0,Y^*]\times[N_1,+\infty),$  we have
\begin{align}\label{mon0315bd-1}
5c_0e^{-\frac{3}{2}\mu (z+\epsilon_0)^2}-\nabla_{\psi}u_{n}^2\leq 0 \quad \text{in}\quad \Omega\cap\{0\leq y\leq Y^*\}\cap\{z\geq N_1\}.
\end{align}
By \eqref{mon0315bd-1}, \eqref{0316psiz} and \eqref{N10316-1}, we have
  \begin{align}\label{mon0315bd-2-1-final}
    \partial_z u_n\geq 2c_0e^{-\frac{3}{2}\mu (z+\epsilon_0)^2}\quad \text{in}\quad \Omega\cap\{0\leq y\leq Y^*\}\cap\{z\geq N_1\}.
 \end{align}
Combining \eqref{mon0315bd-2-1} and \eqref{mon0315bd-2-1-final}, we have  the conclusion.
\end{proof}

\subsection{Estimates on 1st order tangential vector field derivatives}
\subsubsection{ Equation for $\nabla_{\eta}w_n=\nabla_{\eta}u_n$}It holds
\begin{align}\label{qn1}\begin{split}
   0=&\nabla_{\xi} \nabla_{\eta} w_n +(1+\tilde{q})\nabla_{\eta}^2w_n  -\nabla_{\psi}(u_{n}\nabla_{\psi}\nabla_{\eta} w_n)+\nabla_{\eta}\tilde{q}\nabla_{\eta}w_n\\& -K\partial_z w_n
   -\frac{\nabla_{\eta}\tilde{q}}{1+\tilde{q}}\nabla_{\psi}(u_{n}\nabla_{\psi} w_n)-\nabla_{\psi}(\nabla_{\eta}u_{n}\nabla_{\psi} w_n)-\nabla_{\psi}(\frac{\nabla_{\eta}\tilde{q}}{1+\tilde{q}}u_{n}\nabla_{\psi} w_n)\\=&P_1\nabla_{\eta} w_n+\nabla_{\eta}\tilde{q}\nabla_{\eta}w_n
   -\nabla_{\psi}(\nabla_{\eta}u_{n}\nabla_{\psi} w_n)+f_1,\end{split}\end{align} where
      \begin{align}
        f_1=&-K\partial_z u_n
   -\frac{\nabla_{\eta}\tilde{q}}{1+\tilde{q}}\frac{1}{u_{n-1}}\partial_z( \frac{u_{n}}{u_{n-1}} \partial_z u_n)-\frac{1}{u_{n-1}}\partial_z(\frac{\nabla_{\eta}\tilde{q}}{1+\tilde{q}}\frac{u_{n}}{u_{n-1}}\partial_z u_n).   \end{align}

 By induction assumption and \eqref{Ksmallinfty}, we have
   \begin{align}
  |f_1 |\leq \frac{\varepsilon^2C}{u_{n-1}^{2-\alpha}}\phi_{1,0},
\end{align}where we have used \eqref{0221qn-1}.
\subsubsection{Equation for $\nabla_{\xi}w_n=\nabla_{\xi}u_n$}
First, we multiply \eqref{qn0}  by $\frac{1}{1+\tilde{q}}$ and obtain
\begin{align*}
   \frac{1}{1+\tilde{q}} \nabla_{\xi} w_n +\nabla_{\eta}w_n  -\frac{1}{1+\tilde{q}}\nabla_{\psi}(u_{n}\nabla_{\psi}w_n)
      =0.
\end{align*}
Then we have
\begin{align}\begin{split}
 0=& \frac{1}{1+\tilde{q}} \nabla_{\xi}^2 w_n +\nabla_{\eta}\nabla_{\xi}w_n  -\frac{1}{1+\tilde{q}}\nabla_{\psi}(u_{n}\nabla_{\psi}\nabla_{\xi}w_n)
  - \frac{ \nabla_{\xi}\tilde{q}}{(1+\tilde{q})^2} \nabla_{\xi} w_n   \\&+K\partial_z w_n+\frac{ \nabla_{\xi}\tilde{q}}{(1+\tilde{q})^2}\nabla_{\psi}(u_{n}\nabla_{\psi}w_n) -\frac{1}{1+\tilde{q}}\nabla_{\psi}(\nabla_{\xi}u_{n}\nabla_{\psi}w_n)  .\end{split}
\end{align}
        Hence, \begin{align}\label{qn2}\begin{split}
 0=&  P_1\nabla_{\xi}w_n - \frac{ \nabla_{\xi}\tilde{q}}{1+\tilde{q}} \nabla_{\xi} w_n       -\nabla_{\psi}(\nabla_{\xi}u_{n}\nabla_{\psi}w_n)
   +f_2,\end{split}
\end{align}where
\begin{align*}
        f_2 =&(1+\tilde{q})K\partial_z w_n+
         \frac{ \nabla_{\xi}\tilde{q}}{1+\tilde{q}}\nabla_{\psi}(u_{n}\nabla_{\psi}w_n)
         =(1+\tilde{q})K\partial_z u_n+ \frac{ \nabla_{\xi}\tilde{q}}{1+\tilde{q}}\frac{1}{u_{n-1}} \partial_z(\frac{u_n}{u_{n-1}}\partial_z u_{n})
         .   \end{align*}
     By induction assumption and \eqref{Ksmallinfty}, we have
   \begin{align}\label{4.17}
  |f_2 |\leq \frac{\varepsilon^2C}{u_{n-1}^{2-\alpha}}\phi_{1,0}.
\end{align}

\subsubsection{Equation for $\nabla_{\tau_1}    \bar{u}$}
\begin{align*}
 0= &\nabla_\xi \nabla_{\tau_1}\bar{u} +(1+\tilde{q})\nabla_{\eta}\nabla_{\tau_1}\bar{u}  -\nabla_{\psi}(\nabla_{\tau_1} \bar{u}\nabla_{\psi}u_n) -\nabla_{\psi} ( u_n\nabla_{\psi}\nabla_{\tau_1}\bar{u}) \\& +(2\nabla_{\tau_1}-\nabla_\xi -(1+\tilde{q})\nabla_{\eta})  \nabla_{\tau_1} \bar{u}
 \\&+ \nabla_{\psi}(\nabla_{\tau_1} \bar{u}\nabla_{\psi}u_n) -\nabla_{n}( \nabla_{\tau_1}\bar{u}\nabla_{n}  \bar{u})+\nabla_{\psi} ( u_n\nabla_{\psi}\nabla_{\tau_1}\bar{u})-\nabla_{n}( \bar{u}\nabla_{n}  \nabla_{\tau_1}\bar{u}).
\end{align*}
\subsubsection{Estimates on  $\nabla_{\eta,\xi}u_n-\nabla_{\tau_1}\bar{u}$}
Note that the tangential vector field derivative $w=\nabla_{\eta,\xi}u_n$  satisfies
\begin{align}\begin{split}
  0=  P_1w  -\nabla_{\psi} (w\nabla_{\psi} u_n) +\bar{c}w+F,
  \end{split}
\end{align}where $ |\bar{c}|\leq \varepsilon^2 C\phi_{1,0}$. (Refer to \eqref{qn1} and \eqref{qn2} for the specific expression of $\bar{c}$ and $F.$) Hence, we have
\begin{align}\label{P1111}\begin{split}
  P_1(w-\nabla_{\tau_1}\bar{u})  -\nabla_{\psi} \big ((w-\nabla_{\tau_1}\bar{u})\nabla_{\psi} u_n\big)=R_2
  ,
  \end{split}
\end{align}
where
\begin{align}\label{r2xi}\begin{split}
    R_2=&-(\bar{c}w+F)+(2\nabla_{\tau_1}-\nabla_\xi -(1+\tilde{q})\nabla_{\eta})  \nabla_{\tau_1} \bar{u}
 \\&+ \nabla_{\psi}(\nabla_{\tau_1} \bar{u}\nabla_{\psi}u_n) -\nabla_{n}( \nabla_{\tau_1}\bar{u}\nabla_{n}  \bar{u})+\nabla_{\psi} ( u_n\nabla_{\psi}\nabla_{\tau_1}\bar{u})-\nabla_{n}( \bar{u}\nabla_{n}  \nabla_{\tau_1}\bar{u}).
\end{split}
\end{align}
\begin{theorem}\label{r2imp} For a small positive constant $\alpha_0,$ it holds that
 \begin{align}\label{0323-toimp}
    |\nabla_{\eta,\xi}u_n-\nabla_{\tau_1}\bar{u} |\leq \varepsilon^5 \phi_{2,\alpha_0}+C\varepsilon^6\epsilon_0^2\phi_{2,0} \quad \text{in} \quad \Omega\cap\{0\leq y\leq Y^*\},
 \end{align} which implies
 \begin{align}\label{20250206}
 |\nabla_{\eta,\xi}u_n-\nabla_{\tau_1}\bar{u} |\leq C\varepsilon^5 \phi_{1,0}\quad  \text{in} \quad \Omega\cap\{0\leq y\leq Y^*\}.
 \end{align}Here $\alpha_0>\alpha.$
 \end{theorem}
 \begin{proof} \textit{Step 1} We will prove \eqref{0323-toimp}.
 By \eqref{P1111}, $w=\nabla_{\eta,\xi}u_n$ satisfies
  \begin{align}\label{426}\begin{split}
 ( P_2+ P_3)(w-\nabla_{\tau_1}\bar{u})  =R_2
  ,
  \end{split}
\end{align}
where $P_3$ contains 0th-order term and 1st-order derivative with respect to normal direction as follows, \begin{align}\begin{split}
  P_3 g:=- 2(\nabla_{\psi} u_n)\nabla_{\psi} g -\nabla_{\psi}^2 u_n g.
  \end{split}
\end{align}
By \eqref{xitau},   \eqref{dpsi3un} for $\nabla_{\psi}^2 u_n$ and $\nabla_{\psi}^3 u_n$, \eqref{taunu} for $\bar{u}$, \eqref{cy}, Theorem \ref{thm1} and Theorem \ref{dzunphi0},  we have
\begin{align}\label{r2sheng}
   | R_2|\leq \frac{C}{u_{n-1}^{2}}\phi_{1,0}.
   \end{align}
Now  we deal with $P_3$ as follows. First, 
for $\delta_1\in(0,\delta_0)$ small enough, \begin{align*}
  -\nabla_{\psi}^2u_{n}=&-\frac{1}{u_{n-1}}\partial_z(\frac{1}{u_{n-1}}\partial_zu_n)
  =-\frac{1}{u_{n-1}^2}\partial_z^2u_n+\frac{1}{u_{n-1}^3}\partial_zu_{n-1}\partial_zu_n
 \\ \geq &\frac{c_0^2}{2u_{n-1}^3}\quad \text{in}\quad \Omega\cap\{0\leq y\leq Y^*\}\cap\{z\leq
 \delta_1\} ,\end{align*}  where  $\delta_0$ and $c_0$ are defined in \eqref{pphi} and \eqref{positivelbu} respectively which are independent of $\alpha_0.$
Then for small positive constants $\alpha_0, \delta_1\ll 1$, it holds
\begin{align}\begin{split}
&P_3\phi_{2,\alpha_0}
 +\frac{c_2\alpha_0(1-\alpha_0)}{8}\psi_{n-1}^{-\frac{3}{2 }}\phi_{2,\alpha_0}\\=&-2\nabla_{\psi}u_{n}\nabla_{\psi}\phi_{2,\alpha_0}
  -(\nabla_{\psi}^2u_{n})\phi_{2,\alpha_0}+\frac{c_2\alpha_0(1-\alpha_0)}{8}\psi_{n-1}^{-\frac{3}{2 }}\phi_{2,\alpha_0}\\=&(
  -2\alpha_0\frac{\partial_zu_n}{u_{n-1}\psi_{n-1}}
  -\frac{1}{u_{n-1}^2}\partial_z^2u_n
  +\frac{1}{u_{n-1}^3}\partial_zu_{n-1}\partial_zu_n
  +\alpha_0\frac{c_2(1-\alpha_0)}{8}\psi_{n-1}^{-\frac{3}{2 }})\phi_{2,\alpha_0}
  \\ \geq & 0 \quad \text{in}\quad \Omega\cap\{0\leq y\leq Y^*\}\cap\{z\leq
 \delta_1\} ,
  \end{split}
\end{align}
where we have used the fact that  for $z\in[0,
 \delta_1] $, if $\frac{2\partial_z u_n}{u_{n-1}\psi_{n-1}}\geq \frac{c_2(1-\alpha_0)}{8}\psi_{n-1}^{-\frac{3}{2 }},$
 then for $\alpha_0\ll 1,$ it holds
 \begin{align*}
 |\alpha_0\frac{\partial_zu_n}{u_{n-1}\psi_{n-1}}|\leq \alpha_0\frac{C_0}{u_{n-1}\psi_{n-1}}\leq \alpha_0 (\frac{16C_0}{c_2(1-\alpha_0)})^2\frac{C_0}{u_{n-1}^3}\leq
 \frac{c_0^2}{2u_{n-1}^3}.
 \end{align*}
Note that  $c_2$ is defined in \eqref{pphi} and $c_0, C_0$ are defined in \eqref{positivelbu}.
Moreover,  for  some positive constants $\lambda$ and $c,$
 \begin{align*}
    \phi_{2,\alpha_0} \psi_{n-1}^{-\frac{3}{2}}\geq  \lambda e^{-\frac{A}{1+x}} \psi_{n-1}^{\alpha_0-\frac{3}{2}}\geq ce^{-\frac{A}{1+x}} u_{n-1}^ {2\alpha_0-3}\quad \text{in}\quad \Omega\cap\{0\leq y\leq Y^*\}\cap\{z\leq
 \delta_1\}
 \end{align*} and $2\alpha_0-3<-2$ for $\alpha_0<\frac{1}{2}.$

 On the other hand, in $\Omega\cap\{0\leq y\leq Y^*\}\cap\{z\geq\delta_1\} $, by   \eqref{psin-1}, we have
 \begin{align*}
    |\nabla_{\psi}u_{n}\nabla_{\psi}\phi_{2,\alpha_0}|
    =&|\frac{\partial_zu_n}{u_{n-1}}\nabla_{\psi}\phi_{2,\alpha_0}|
    =|-\frac{9}{5}\frac{\psi_{n-1}}{x+1}\mu\frac{\partial_zu_n}{u_{n-1}}\phi_{2,\alpha_0}|\\
 \leq &C|\psi_{n-1}\partial_zu_n|\phi_{2,\alpha_0}
 \leq C\phi_{2,\alpha_0} \quad\text{in}\quad \Omega\cap\{0\leq y\leq Y^*\}\cap\{\frac{\psi_{n-1}}{\sqrt{x+1}}\geq N\}  ,
 \end{align*}and
   \begin{align*}
    |\nabla_{\psi}u_{n}\nabla_{\psi}\phi_{2,\alpha_0}|
    =&|\frac{\partial_zu_n}{u_{n-1}}\nabla_{\psi}\phi_{2,\alpha_0}|
 \leq C\phi_{2,\alpha_0} \,\,\text{in}\,\, \Omega\cap\{0\leq y\leq Y^*\} \cap\{z\geq\delta_1\}\cap\{\frac{\psi_{n-1}(x,y,z)}{\sqrt{x+1}}\leq N\} ,
 \end{align*}
  where we have used the facts that  $\nabla_\psi x=\frac{1}{u_{n-1}}\partial_z x=0$ and for some positive constants $c$ and $C,$
 $$u_{n-1}\geq c_0\delta_1,\quad |\partial_zu_n|\leq Ce^{-c\frac{\psi_{n-1}^2(x,y,z )}{x+1}\mu}\quad \text{in}\quad \Omega\cap\{0\leq y\leq Y^*\}\cap\{z\geq\delta_1\} ,$$
by \eqref{assumpn-2}, \eqref{defy} and  the relation $c\leq  \frac{\psi_{n-1}}{z+\epsilon_0}\leq C$ in $\Omega\cap\{0\leq y\leq Y^*\}\cap\{z\geq\delta_1\} .$ Then $|P_3\phi_{2,\alpha_0}|\leq C\phi_{2,\alpha_0}$ in $\Omega\cap\{0\leq y\leq Y^*\}\cap\{z\geq\delta_1\} .$

 Hence, by taking  $A\gg \frac{1}{\varepsilon^5}$ and by \eqref{426} and \eqref{pphi},
    \begin{align*}
 ( P_2+ P_3)(\pm(\nabla_{\eta,\xi}u_n-\nabla_{\tau_1}\bar{u})-C\varepsilon^6\epsilon_0^2\phi_{2,0} -\varepsilon^5 \phi_{2,\alpha_0})\leq 0,\end{align*} in $\Omega\cap\{0\leq y\leq Y^*\}\setminus\{\text{ridges of barrier functions}\}.$
 Then by a similar argument as in the previous proof,
  we obtain \eqref{0323-toimp}.

 \textit{Step 2} We will prove \eqref{20250206}. By  \eqref{0323-toimp}, we only need to prove it for $z$ large.

  By  \eqref{0323-toimp} and the prescribed boundary conditions on $\{x=0\}\cup\{y=0\},$
 $$\pm(\nabla_{\eta,\xi}u_n-\nabla_{\tau_1} \bar{u})-C\varepsilon^{5} \phi_{1,0}\leq 0\quad \text{on}\quad [0,X]\times[0,Y^*]\times[N,+\infty)\cap(\{x= 0\}\cup\{y=0\}\cup\{z= N\}),$$
 for small $\varepsilon.$
Next, by \eqref{426},
  we have
  \begin{align}\begin{split}
  P_2(\nabla_{\eta,\xi}u_n-\nabla_{\tau_1}\bar{u})  =\hat{R}_2
  \quad\text{in}\quad(0,X]\times(0,Y^*]\times(N,+\infty),
  \end{split}
\end{align}
where we note $|\hat{R}_2|\leq C\phi_{1,0}$ in $(0,X]\times(0,Y^*]\times(N,+\infty).$ Hence,  by \eqref{0323-p2phi} and taking $A>>\varepsilon^{-5},$ we have \begin{align}\begin{split}
  P_2(\nabla_{\eta,\xi}u_n-\nabla_{\tau_1}\bar{u}-C\varepsilon^{5} \phi_{1,0})  <0
  \quad\text{in}\quad(0,X]\times(0,Y^*]\times(N,+\infty).
  \end{split}
\end{align}
Then applying the maximum principle in the domain $ [0,X]\times[0,Y^*]\times[N,+\infty),$ we obtain \eqref{20250206}.
   \end{proof}

\begin{remark} The  growth rates near $z=0$  will be improved to those given in  \eqref{1116-2} later.
\end{remark}

         \section{Second  order derivative estimates}\label{s5}

         \subsection{Estimates on the 2nd order tangential vector field derivatives}

For the second order tangential vector field derivatives, we will first prove the following theorem.

\begin{theorem}\label{5.1}
	It holds that
\begin{align*}
|\nabla_{\eta}\nabla_{\xi}u_n-\nabla_{\tau_1}^2\bar{u}|\leq \frac{\varepsilon^2}{2}d_0^2\phi_{1,\frac{\alpha}{2}} \quad \text{in} \quad \Omega\cap\{0\leq y\leq Y^*\},
\end{align*}
where $d_0$ is defined in \eqref{assumpn-2}.
\end{theorem}

         \subsubsection{Equations for $\nabla_{\eta}\nabla_{\xi}u_n$ and $\nabla_{\tau_1}^2 \bar{u}$ } In the following discussion, we some times use $w_n$ for  $u_n $  for clear presentation  of the structure of the
         equations.
 Applying $\nabla_{\eta}$ to \eqref{qn2}, we have

\begin{align}\label{wtan2}\begin{split}
    0=&P_1 \nabla_{\eta}\nabla_{\xi} w_n+(\nabla_{\eta}\tilde{q}-\frac{ \nabla_{\xi}\tilde{q}}{1+\tilde{q}} )\nabla_{\eta}\nabla_{\xi} w_n\\&+(1+\tilde{q})\nabla_{\eta}(K\partial_zw_n)-\frac{1}{u_{n-1}}\partial_z(\frac{\nabla_{\eta} \nabla_{\xi} u_{n}}{u_{n-1}}\partial_zw_{n} ) +f_{12},\end{split}
\end{align}
where
\begin{align}\begin{split}
 f_{12}=&K\partial_zw_n\nabla_{\eta}\tilde{q}-K\partial_z\nabla_{\xi} w_{n} -\nabla_{\eta}(\frac{ \nabla_{\xi}\tilde{q}}{1+\tilde{q}}) \nabla_{\xi} w_{n}\\&+\nabla_{\eta}\big(\frac{ \nabla_{\xi}\tilde{q}}{1+\tilde{q}}\frac{1}{u_{n-1}} \partial_z(\frac{u_n}{u_{n-1}}\partial_zw_{n})\big)
   -\frac{ \nabla_{\eta}\tilde{q}}{1+\tilde{q}} \frac{\partial_z(\frac{u_{n}}{u_{n-1}}\partial_z\nabla_{\xi} w_{n})}{u_{n-1}}
   \\& -\frac{1}{u_{n-1}}\partial_z(\frac{ \nabla_{\eta}\tilde{q}}{1+\tilde{q}}\frac{u_{n}}{u_{n-1}} \partial_z \nabla_{\xi} w_{n} )
   \\& -\frac{ \nabla_{\eta}\tilde{q}}{1+\tilde{q}} \frac{1}{u_{n-1}}\partial_z(\frac{\nabla_{\xi} u_{n}}{u_{n-1}}\partial_zw_{n} )  -\frac{1}{u_{n-1}}\partial_z(\frac{\nabla_{\xi}u_{n}}{u_{n-1}}\frac{ \nabla_{\eta}\tilde{q}}{1+\tilde{q}} \partial_z w_{n} )
    \\& -\frac{1}{u_{n-1}}\partial_z(\frac{\nabla_{\xi}u_{n}}{u_{n-1}}\partial_z\nabla_{\eta}  w_{n} )-\frac{1}{u_{n-1}}\partial_z(\frac{\nabla_{\eta}u_{n}}{u_{n-1}}\partial_z \nabla_{\xi} w_{n} ) .\end{split}
\end{align}
In particular, by \eqref{tildq0221},  it holds in $[0,Y^*],$
\begin{align}\label{alpha2}\begin{split}
&|-\frac{\nabla_{\eta} \nabla_{\xi}\tilde{q}}{1+\tilde{q}} \nabla_{\xi} w_{n}+\frac{ \nabla_{\eta}\nabla_{\xi}\tilde{q}}{1+\tilde{q}}\frac{1}{u_{n-1}} \partial_z(\frac{u_n}{u_{n-1}}\partial_zu_{n})|
\\=&|-\frac{\nabla_{\eta} \nabla_{\xi}\tilde{q}}{1+\tilde{q}} \nabla_{\xi} w_{n}+\frac{ \nabla_{\eta}\nabla_{\xi}\tilde{q}}{1+\tilde{q}}\nabla_{\psi}(u_n\nabla_{\psi}u_{n})
|\\=&|-\frac{\nabla_{\eta} \nabla_{\xi}\tilde{q}}{1+\tilde{q}} \nabla_{\xi} w_{n}+\frac{ \nabla_{\eta}\nabla_{\xi}\tilde{q}}{1+\tilde{q}}(\nabla_{\xi} u_n +(1+\tilde{q})\nabla_{\eta}u_n )|
  \leq C  \varepsilon^2\phi_{1,\frac{\alpha}{2}}.\end{split}
\end{align}

Since  $\nabla_{\tau_1} \bar{u}=\nabla_{\tau_2} \bar{u}$  and
\begin{align*}
 0= &\nabla_{\tau_1}   \nabla_{\tau_1}^2 \bar{u}+ \nabla_{\tau_2}   \nabla_{\tau_1}^2 \bar{u}-\nabla_{n}( \bar{u}\nabla_{n}  \nabla_{\tau_1}^2\bar{u})
  -\nabla_{n}(  \nabla_{\tau_1}^2\bar{u}\nabla_{n} \bar{u})
  -2\nabla_{n}(  \nabla_{\tau_1}\bar{u}\nabla_{n} \nabla_{\tau_1} \bar{u}),
\end{align*} we have \begin{align}\label{021455}\begin{split}
 0=&
  P_1 \nabla_{\tau_1}^2 \bar{u}+(\nabla_{\eta}\tilde{q}-\frac{ \nabla_{\xi}\tilde{q}}{1+\tilde{q}} )\nabla_{\tau_1}^2 \bar{u} -\frac{1}{u_{n-1}}\partial_z(\nabla_{\tau_1}^2 \bar{u}\frac{\partial_zu_{n}}{u_{n-1}} ) +f_{12}^{\bar{u}}
  ,\end{split}
\end{align} where
\begin{align*}
    f_{12}^{\bar{u}}=&(\nabla_{\tau_1}-\nabla_{\xi} ) \nabla_{\tau_1}^2 \bar{u} +(\nabla_{\tau_2}-(1+\tilde{q})\nabla_{\eta}) \nabla_{\tau_1}^2 \bar{u}-(\nabla_{\eta}\tilde{q}-\frac{ \nabla_{\xi}\tilde{q}}{1+\tilde{q}} )\nabla_{\tau_1}^2 \bar{u}\\
   & +(\nabla_{\psi}-\nabla_n)(u_{n}\nabla_{\psi} \nabla_{\tau_1}^2 \bar{u})
   +\nabla_n((u_{n}-\bar{u})\nabla_{\psi} \nabla_{\tau_1}^2 \bar{u})
   +\nabla_n(\bar{u}(\nabla_{\psi}-\nabla_n) \nabla_{\tau_1}^2 \bar{u})
   \\
   & +(\nabla_{\psi}-\nabla_n)(\nabla_{\tau_1}^2 \bar{u} \frac{\partial_z u_n}{u_{n-1}})
   +\nabla_n(\nabla_{\tau_1}^2 \bar{u}( \frac{\partial_z u_n}{u_{n-1}}- \frac{\partial_z \bar{u}}{\bar{u}}))  -2\nabla_{n}(  \nabla_{\tau_1}\bar{u}\nabla_{n} \nabla_{\tau_1} \bar{u}).
\end{align*}
Note that here the last term corresponds to the last two terms of $f_{12}.$

Next, we estimate
\begin{align}\label{02142025}
f_{12}^{u-\bar{u}}:=f_{12}- f_{12}^{\bar{u}}.
\end{align}
By \eqref{P1111},  $w=\nabla_{\eta,\xi }u_n$ satisfies
\begin{align}\label{5.5-1}\begin{split}
 \nabla_{\psi}^2 (w-\nabla_{\tau_1}\bar{u})=&\frac{1}{u_n }[\nabla_{\xi} (w-\nabla_{\tau_1}\bar{u}) +(1+\tilde{q})\nabla_{\eta}(w-\nabla_{\tau_1}\bar{u})\\& -2 \nabla_{\psi} (w-\nabla_{\tau_1}\bar{u})\nabla_{\psi} u_n
 -(w-\nabla_{\tau_1}\bar{u})\nabla_{\psi}^2 u_n-R_2]
  .
  \end{split}
\end{align}
Then by the definition of $Y^*$ in \eqref{defy}, \eqref{psi-n}, Theorem \ref{thm1} and Theorem \ref{dzunphi0}, we have
\begin{align}\label{5.5-2}\begin{split}
 &| \nabla_{\psi}^2 (w-\nabla_{\tau_1}\bar{u})|+|\nabla_{\psi} (\nabla_{\psi}w-\nabla_{n}\nabla_{\tau_1}\bar{u}) |\\ =& | \nabla_{\psi}^2 (w-\nabla_{\tau_1}\bar{u})|+ | \nabla_{\psi}^2 (w-\nabla_{\tau_1}\bar{u})+ \nabla_{\psi}( \nabla_{\psi}-\nabla_{n})\nabla_{\tau_1}\bar{u}|
 \\ \leq &\frac{C}{u_{n-1}^{3-\alpha}}\phi_{1,0},\end{split}
\end{align} where we have used the fact that for $z$ near $0,$
\begin{align}\label{1tder}
|w-\nabla_{\tau_1}\bar{u}|\leq C(z+{\epsilon_0})^{1+\alpha} e^{-\frac{A}{x+1}},
\end{align}
by the definition of $Y^*$ in \eqref{defy}.
By the induction assumption and  \eqref{alpha2}, we have $f_{12}^{u-\bar{u}}=f_{12}- f_{12}^{\bar{u}}$ satisfies
\begin{align}\label{f12baru-u}
|f_{12}^{u-\bar{u}}|=|f_{12}- f_{12}^{\bar{u}}|\leq \frac{\varepsilon^2 C\phi_{1,0}}{\bar{u}^{3-\alpha}}.
\end{align}

\subsubsection{Proof of Theorem \ref{5.1}}

\begin{proof}[Proof of Theorem \ref{5.1}]

\noindent\textit{Step 1  Equation of $\nabla_{\eta}\nabla_{\xi} u_n-\nabla_{\tau_1}^2\bar{u}$ }

By \eqref{wtan2} and \eqref{021455}, we have
\begin{align}\label{1104-1}\begin{split}
    0=&P_1( \nabla_{\eta}\nabla_{\xi} u_n-\nabla_{\tau_1}^2\bar{u})-\frac{\partial_zu_{n} }{u_{n-1}^2}\partial_z( \nabla_{\eta}\nabla_{\xi} u_n-\nabla_{\tau_1}^2\bar{u})+\tilde{c}( \nabla_{\eta}\nabla_{\xi} u_n-\nabla_{\tau_1}^2\bar{u})\\&+(1+\tilde{q})\partial_{y}(K\partial_zu_n) +f_{12}^u,\end{split}
\end{align}  where by \eqref{dzk},
\begin{align}\label{f12u}\begin{split}
f_{12}^u=&f_{12}^{u-\bar{u}}-(1+\tilde{q})\frac{\int_0^z \partial_{y}v_{n-1}dz'}{v_{n-1}}\partial_{z}(K\partial_zu_n)\\
=&f_{12}^{u-\bar{u}}-(1+\tilde{q})\frac{\int_0^z \partial_{y}v_{n-1}dz'}{v_{n-1}}K\partial_z^2u_n-(1+\tilde{q})\frac{\int_0^z \partial_{y}v_{n-1}dz'}{v_{n-1}}\partial_{z}K\partial_zu_n
\\
=&f_{12}^{u-\bar{u}}-(1+\tilde{q})\frac{\int_0^z \partial_{y}v_{n-1}dz'}{v_{n-1}}K\partial_z^2u_n
\\&-(1+\tilde{q})\frac{\int_0^z \partial_{y}v_{n-1}dz'}{v_{n-1}}\partial_zu_n(-\nabla_{\xi}(\frac{\nabla_{\eta}\tilde{q}}{1+\tilde{q}})
-K\frac{\partial_z u_{n-1}}{u_{n-1}}).\end{split}
\end{align}
Recall $f_{12}^{u-\bar{u}}$ in \eqref{02142025} and
\begin{align}\label{utan2c}\begin{split}
\tilde{c}=&(\nabla_{\eta}\tilde{q}-\frac{ \nabla_{\xi}\tilde{q}}{1+\tilde{q}} )-\frac{1}{u_{n-1}}\partial_z(\frac{\partial_zu_{n} }{u_{n-1}})\\
=&(\nabla_{\eta}\tilde{q}-\frac{ \nabla_{\xi}\tilde{q}}{1+\tilde{q}} )-\frac{\partial_z^2u_{n} }{u_{n-1}^2}+\frac{\partial_zu_{n} \partial_zu_{n-1} }{u_{n-1}^3}
\\:=&c_1+\frac{\partial_zu_{n} \partial_zu_{n-1} }{u_{n-1}^3}.\end{split}
\end{align}Then   by \eqref{Ksmallinfty}, we have
\begin{align}\label{cf12u}
|c_1+\frac{\partial_z^2u_{n} }{u_{n-1}^2}|\leq  \frac{C \phi_{1,0}}{u_{n-1}^2}, \quad |f_{12}^{u}|\leq C\varepsilon^2\phi_{1,0}+\frac{\varepsilon^2C\phi_{1,0}}{u_{n-1}^{3-\alpha}}.
\end{align}

\noindent\textit{Step 2  Auxiliary function $f$ and the goal}

For any small positive constant
\begin{align}\label{02072025}
\gamma\ll \varepsilon^3e^{-A}\epsilon_0\leq \varepsilon^3e^{-\frac{A}{x+1}}\epsilon_0,
\end{align}
 set
\begin{align}\label{fplus}
 f= ( \nabla_{\eta}\nabla_{\xi} u_n-\nabla_{\tau_1}^2\bar{u})-\frac{\varepsilon^2}{2}d_0^2\phi_{1,\frac{\alpha}{2}}-\gamma
 \quad
 \text{in} \quad[0,X]\times[0,Y]\times[0,+\infty).
\end{align}
 \textit{Goal}: We will prove
  \begin{align}\label{0323goalfneg}
    f\leq 0  \quad \text{in} \quad [0,X]\times[0,Y^*]\times[0,+\infty).
  \end{align}
  Then letting $\gamma$ go to $0$, we have
  \begin{align}\label{0324fremovegam}
 ( \nabla_{\eta}\nabla_{\xi} u_n-\nabla_{\tau_1}^2\bar{u})-\frac{\varepsilon^2}{2}d_0^2\phi_{1,\frac{\alpha}{2}}
 \leq 0
 \quad
 \text{in} \quad[0,X]\times[0,Y^*]\times[0,+\infty).
  \end{align}

 We will employ \textit{a proof by contradiction}. i.e. If \eqref{0323goalfneg} fails, then there exists a
maximum ``time"
\begin{align}\label{ystar0220}
y_*:=\max \{y_0\in[0,Y^*]|f(x,y,z)\leq 0,\quad \text{in} \quad [0,X]\times[0,y_0]\times[0,+\infty)\},
\end{align} such that
\begin{align}\label{0323-ymaxt}
y_*\in(0,Y^*),
\end{align}
 by the boundary condition. In the following, we will prove that this leads to a contradiction.

  As mentioned in the introduction, $\partial_y (K\partial_z u_n)$ in \eqref{1104-1} contains \textit{loss of tangential derivative which is one order higher than the induction assumptions.} Hence, instead of bounding $\partial_y K$, our analysis takes advantage of the sign of $\partial_y f$ as follows.

\noindent\textit{Step 3  The sign of  $f$ and $\partial_y f$ on specific domains}

Firstly, by the boundary condition, we have $f\leq -\gamma$ on $z=0$. By continuity, there exists a positive constant $z_0$ such that
\begin{align}\label{tldM0126}
    f\leq-\frac{\gamma}{2} \quad \text{in} \quad [0,X]\times[0,Y^*]\times[0,z_0].
\end{align}Moreover, since $\nabla_{\eta}\nabla_{\xi} u_n-\nabla_{\tau_1}^2\bar{u}\rightarrow 0$ as $z$ tends to infinity, there exists a large positive constant $N_1$ such that
\begin{align}\label{Ngamm}
f+\frac{\gamma}{2}<0\quad
 \text{in} \quad[0,X]\times[0,Y]\times[N_1,+\infty).
\end{align}



In addition, by the boundary condition,
 $$|\nabla_{\eta}\nabla_{\xi} u_n-\nabla_{\tau_1}^2\bar{u}|-\frac{\varepsilon^2}{8}d_0^2(\phi_{1,\frac{\alpha}{2}})^2\leq 0\quad \text{on}\quad \{x=0\}\times[0,Y]\times[0,N_1].$$

   Note
\begin{align*}
   \phi_{1,\frac{\alpha}{2}} (x,z) \geq \min_{[0,X]\times[0,N_1]}\phi_{1,\frac{\alpha}{2}} (x,z)>0, \quad (x,z)\in [0,X]\times[0,N_1].
\end{align*}
  Then for the positive constant $l=\min_{[0,X]\times[0,N_1]}\frac{\varepsilon^2}{16} d_0^2 \phi_{1,\frac{\alpha}{2}}>0,$ there exists a constant $x_0\in(0,X)$
such that $|\nabla_{\eta}\nabla_{\xi} u_n-\nabla_{\tau_1}^2\bar{u}|-\frac{\varepsilon^2}{8}d_0^2(\phi_{1,\frac{\alpha}{2}})^2<l$ in $[0,x_0]\times[0,Y]\times[0,N_1]. $ Hence,
$|\nabla_{\eta}\nabla_{\xi} u_n-\nabla_{\tau_1}^2\bar{u}|-\frac{\varepsilon^2}{2}d_0^2\phi_{1,\frac{\alpha}{2}}< l+\frac{\varepsilon^2}{8}d_0^2(\phi_{1,\frac{\alpha}{2}})^2-\frac{\varepsilon^2}{2}d_0^2\phi_{1,\frac{\alpha}{2}}\leq-\frac{\varepsilon^2}{4}d_0^2 \phi_{1,\frac{\alpha}{2}}$. Then
 \begin{align}\label{x0}
 f<-\frac{\varepsilon^2}{4}d_0^2 \phi_{1,\frac{\alpha}{2}}<0\quad \text{in}
 \quad [0,x_0]\times[0,Y]\times[0,N_1].
\end{align}

In this step, we will study the properties of the function $f$. First of all, we  consider the
signs of $f$ and $\partial_y f$ on the plane $\{y=y_*\}.$

  Set $$D_{y_*}=[x_0,X]\times\{y=y_*\}\times[z_0,N_1],$$
  where $z_0$ is defined in \eqref{tldM0126}.
Then by the definition of $y_*$ in \eqref{ystar0220}, we have
$$f(p)\leq0,\quad p\in D_{y_*}.$$

For every $p\in D_{y_*}$ , there are  two possibilities.

(i) If $f(p)<0,$ then there exists a positive constant $r_p<\min\{z_0,x_0\}$ such that  $$ f\leq \frac{f(p)}{2} <0 \quad \text{in}\quad B_{r_p}(p)\cap D_{y_*},$$ where $B_{r_p}(p)\subseteq \R^2_{x,z}\times\{y=y_*\}$ and $\R^2_{x,z}=\{(x,z)\in\R^2\}$.

(ii) If $f(p)=0,$ then $\partial_y f(p)\geq 0 $ by the definition of $y_*$ in \eqref{ystar0220}. Then there exists  a positive constant $r_p<\min\{z_0,x_0\}$ such that  $$\partial_y f\geq -\frac{1}{100} \gamma \min_{[0,X]\times[0,N_1]} \phi_{1,\frac{\alpha}{2}} \quad \text{in}\quad B_{r_p}(p)\cap D_{y_*},$$ where $B_{r_p}(p)\subseteq \R^2_{x,z}\times\{y=y_*\}$.

Since $\{B_{r_p}(p)\}_{p\in D_{y_*}}$ is an open covering  of the compact set $ D_{y_*},$  there exists a finite subcover  of $D_{y_*}$ such that $$D_{y_*}\subseteq\big(\cup_{i\in\{1,\cdots,n_0\}}B_{r_{p_i^-}}(p_i^-)\big)\bigcup\big(\cup_{i\in\{1,\cdots,n_1\}}B_{r_{p_i^+}}(p_i^+)\big),$$
  where $p_i^-$ denotes the point where $f(p_i^-)<0$ and $p_i^+$ denotes the point where $\partial_y f(p_i^+)\geq  0$.

  According to the choice of $r_{p},$ we have, for some $i_0\in\{1,\cdots,n_0\},$
\begin{align*}
f\leq \max_{i\in\{1,\cdots,n_0\}}\{\frac{f(p_i^-)}{2} \}=\frac{f(p_{i_0}^-)}{2}<0, \quad &  D_{y_*}\cap\big( \cup_{i\in\{1,\cdots,n_0\}}B_{r_{p_i^-}}(p_i^-)\big),\\
\partial_y f\geq -\frac{1}{100} \gamma  \min_{[0,X]\times[0,N_1]} \phi_{1,\frac{\alpha}{2}}, \quad  & D_{y_*}\cap\big(\cup_{i\in\{1,\cdots,n_1\}}B_{r_{p_i^+}}(p_i^+)\big).
\end{align*}
Then take a constant $y_2$ with $y_2-y_*$ small such that $0<y_2-y_*\ll \min\{1,Y^*-y_*\}.$ By the continuity,  \begin{align}\label{sginf}\begin{split}
f+\gamma  \phi_{1,\frac{\alpha}{2}}(y-y_*)\leq& \max\{ \frac{f(p_{i_0}^-)}{4},\,\,-\min_{[0,x_0]\times[0,N_1]}\frac{\varepsilon^2d_0^2}{4} \phi_{1,\frac{\alpha}{2}},\,\,-\frac{\gamma}{2}\}+\gamma  \phi_{1,\frac{\alpha}{2}}(y_2-y_*)
\\ <&0 \quad  \quad\quad\quad\quad \text{in} \quad D_1 \times[y_*,y_2],\\
\partial_y (f+\gamma  \phi_{1,\frac{\alpha}{2}}(y-y_*))\geq&0 \quad\quad\quad\quad\quad   \text{in} \quad D_2\times[y_*,y_2],\end{split}
\end{align} where  the negative upper bound of $f$ in  $[0,x_0]$ is given in \eqref{x0}, the  negative upper bound of $f$ in  $[0,z_0]$ is given in \eqref{tldM0126}. Here,
\begin{align*}
    D_1=&\big(D_{y_*}\cap\cup_{i\in\{1,\cdots,n_0\}}B_{r_{p_i^-}}(p_i^-)\big)\bigcup
    [0,x_0]\times\{y=y_*\}\times[0,N_1]\bigcup
    [0,X]\times\{y=y_*\}\times[0,z_0],
    \\D_2=&\big(D_{y_*}\cap\cup_{i\in\{1,\cdots,n_1\}}B_{r_{p_i^+}}(p_i^+)\big).
\end{align*}

Note that
\begin{align}\label{d1d2}
   D_1\cup D_2=[0,X]\times[0,N_1]\subseteq \R^2_{x,z},
\end{align} and $D_i \times[y_*,y_2],$ $i=1,2$ are the cylinders with  bottom on the plane $\{y=y_*\}$ and  height being $y_2-y_*.$

\noindent\textit{Step 4  Equation of  $f$}

We now turn to estimate
 $$P_1 f-\frac{\partial_zu_{n} }{u_{n-1}^2}\partial_zf+\frac{\partial_zu_{n} \partial_zu_{n-1} }{u_{n-1}^3}f+c_1f.$$
By the definition of $f$ in \eqref{fplus} and straightforward calculation, we have
\begin{align}\label{520}\begin{split}
    &P_1 f-\frac{\partial_zu_{n} }{u_{n-1}^2}\partial_zf+\frac{\partial_zu_{n} \partial_zu_{n-1} }{u_{n-1}^3}f+c_1f
    \\
    =& -(1+\tilde{q})\partial_{y}(K\partial_zu_n)
 -f_{12}^u-\gamma(\frac{\partial_zu_{n} \partial_zu_{n-1} }{u_{n-1}^3}+c_1)
 \\&-[P_1 \phi_{1,\frac{\alpha}{2}}-\frac{\partial_zu_{n} }{u_{n-1}^2}\partial_z\phi_{1,\frac{\alpha}{2}}+\frac{\partial_zu_{n} \partial_zu_{n-1} }{u_{n-1}^3}\phi_{1,\frac{\alpha}{2}}+c_1\phi_{1,\frac{\alpha}{2}}]\frac{\varepsilon^2}{2}d_0^2
    \\
    =& -(1+\tilde{q})\partial_{y}(K\partial_zu_n)
 -f_{12}^u-\gamma(\frac{\partial_zu_{n} \partial_zu_{n-1} }{u_{n-1}^3}+c_1)
 \\&-[P_1 \phi_{1,\frac{\alpha}{2}}-\frac{\partial_zu_{n} }{u_{n-1}^2}\partial_z\phi_{1,\frac{\alpha}{2}}+\frac{\partial_zu_{n} \partial_zu_{n-1} }{u_{n-1}^3}\phi_{1,\frac{\alpha}{2}}+c_1\phi_{1,\frac{\alpha}{2}}](\frac{\varepsilon^2}{4}+\frac{\varepsilon^2}{4})d_0^2
    ,\end{split}
\end{align}
 where $c_1$ is given in \eqref{utan2c}.
Taking $A$ sufficiently large and  the positive constant $\alpha$ sufficiently small such that
\begin{align*}
   0< \partial_z\phi_{1,\frac{\alpha}{2}}=\frac{\alpha}{2}\phi_{1,\frac{\alpha}{2}}\frac{1}{z+\epsilon_0} \leq \frac{ c_0 }{2u_{n-1}}\phi_{1,\frac{\alpha}{2}}\leq \frac{ \partial_z u_{n-1}}{2u_{n-1}}\phi_{1,\frac{\alpha}{2}}, \quad z\in[0,\delta-\epsilon_0],\end{align*}
where  $\delta$ and  $c_0$ are defined in  \eqref{phi1} and \eqref{positivelbu} respectively, it holds
\begin{align}\label{52302072025}\begin{split}
    &\frac{\varepsilon^2}{4}d_0^2[P_1 \phi_{1,\frac{\alpha}{2}}-\frac{\partial_zu_{n} }{u_{n-1}^2}\partial_z\phi_{1,\frac{\alpha}{2}}+\frac{\partial_zu_{n} \partial_zu_{n-1} }{u_{n-1}^3}\phi_{1,\frac{\alpha}{2}}+c_1\phi_{1,\frac{\alpha}{2}}
     ]
     \\ \geq&\varepsilon^2d_0^2(\frac{3\lambda}{u_{n-1}^3}+\frac{A}{5}c_2)\phi_{1,\frac{\alpha}{2}}
     >0,\quad z\in[0,+\infty),\end{split}
\end{align}where $\lambda$ is a constant with $0<3\lambda\leq \frac{c_0^2}{8}$, $c_2$ is defined in \eqref{pphi}. And we note that $\partial_z^-\phi_{1,\frac{\alpha}{2}}>\partial_z^+\phi_{1,\frac{\alpha}{2}}$ at the ridges.
  By \eqref{cf12u} and by taking $\delta_2$ small enough such that $\frac{\delta_2^{\frac{\alpha}{2}}}{\varepsilon }\ll d_0^2,$ we have,
 \begin{align}
|f_{12}^{u}|\leq \frac{\varepsilon^3Cd_0^2\phi_{1,0}}{u_{n-1}^{3-\frac{\alpha}{2}},}\quad z\in[0,\delta_2].
\end{align}

 Then for  $
\gamma\ll \varepsilon^3e^{-A}\epsilon_0\leq\varepsilon^3 e^{-\frac{A}{x+1}}Cu_{n-1},
$ $\varepsilon\ll d_0^2\ll1 $ and  $A\gg \frac{1}{\delta_2^{3}}$, we have
\begin{align}\label{7.14}\begin{split}
    &\partial_{x}f+(1+\tilde{q})\partial_{y}f
  -\frac{1}{u_{n-1}}\partial_{z} (\frac{u_n}{u_{n-1}}\partial_{z}f) \\&+(-\tilde{b}-\frac{\partial_zu_{n} }{u_{n-1}^2})\partial_zf+\frac{\partial_zu_{n} \partial_zu_{n-1} }{u_{n-1}^3}f+c_1f \\=&P_1 f-\frac{\partial_zu_{n} }{u_{n-1}^2}\partial_zf+\frac{\partial_zu_{n} \partial_zu_{n-1} }{u_{n-1}^3}f+c_1f \\
    \leq &-\varepsilon^2d_0^2(\frac{2\lambda}{u_{n-1}^3}+\frac{A}{6}c_2)\phi_{1,\frac{\alpha}{2}} -(1+\tilde{q})\partial_{y}(K\partial_zu_n+(2-k_1)\varepsilon^{2} e^{-\frac{A}{x+1}}C_1)
  \\&-[P_1 \phi_{1,\frac{\alpha}{2}}-\frac{\partial_zu_{n} }{u_{n-1}^2}\partial_z\phi_{1,\frac{\alpha}{2}}+\frac{\partial_zu_{n} \partial_zu_{n-1} }{u_{n-1}^3}\phi_{1,\frac{\alpha}{2}}+c_1\phi_{1,\frac{\alpha}{2}}]\frac{\varepsilon^2}{4}d_0^2
   \end{split}
\end{align}
 where the function $\tilde{b}$ is given in \eqref{linemethodb} and the two constants $C_1$
 and $k_1$ are chosen as follows.
 For $\varepsilon$ small enough, we have
 \begin{align}\label{k0}
 	2 (1+\tilde{q})\geq 2-k_0
 \end{align} for some small positive constant $k_0\ll 1$.
$C_1$ independent of $\varepsilon$ and $k_0< k_1\ll 1$ are two positive constants chosen such that $C_1\gg C_2$ where the constant $C_2$ is defined in \eqref{Ksmallinfty} and
 \begin{align}\label{Kplus}
 	(2-k_2)\varepsilon^{2}e^{-\frac{A}{x+1}}C_1< (1+\tilde{q})( K\partial_zu_n+(2-k_1)\varepsilon^{2}e^{-\frac{A}{x+1}}C_1)\leq\varepsilon^{2}e^{-\frac{A}{x+1}}C_1
 	(2-k_0),\end{align}  with $0<k_0< k_1<k_2\ll 1$. 
 \\

\noindent\textit{Step 5  Auxiliary function $F$ and proof of \eqref{toy_1}}

Set \begin{align}\label{fmod}\begin{split}
 F=& \nabla_{\eta}\nabla_{\xi} u_n-\nabla_{\tau_1}^2 \bar{u}-\frac{\varepsilon^2}{2}d_0^2\phi_{1,\frac{\alpha}{2}}-\gamma -2\varepsilon^{2}e^{-\frac{A}{x+1}}C_1+2 \varepsilon^{2}e^{-\frac{A}{x+1}}C_1\zeta(y)
 \\&+\gamma  \phi_{1,\frac{\alpha}{2}}(y-y_*), \end{split}
\end{align} in $[0,X]\times[y_*, y_2]\times[0,+\infty),$ where  \begin{equation}\label{zeta}\zeta(y)= \left\{
  \begin{aligned}
 1,\quad & y_*\leq y\leq y_1,\\
  1-\frac{y-y_1}{y_2-y_1},\quad & y_1< y\leq y_2.
\end{aligned}\right.\\
\end{equation}
 Here $ y_2<Y^*$ is defined in \eqref{sginf} and  $y_1$ will be  determined  as follows.

 Take  $y_1\in(y_*,y_2)$ to be a positive constant so that  $y_2-y_1$ is  sufficiently small to satisfy
\begin{align}\label{y2-y1}
   2 \varepsilon^{2}e^{-A}C_1\frac{1}{y_2-y_1}\gg\max_{[0,X]\times[0,Y^*]\times[0,N_1]} |\partial_y f|+\gamma |\phi_{1,\frac{\alpha}{2}}|.
\end{align}
This  implies
\begin{align}\label{k3dyF}
1-k_3\leq \frac{-\partial_y F}{d(x)}\leq 1+k_3\quad \text{in} \quad (y_1,y_2],
\end{align} by $e^{-\frac{A}{x+1}}\geq e^{-A},$  where
\begin{align}\label{ddyzeta}
d(x):=\frac{2 \varepsilon^{2}e^{-\frac{A}{x+1}}C_1}{y_2-y_1}.
\end{align}
 $k_3\ll 1$ is chosen small enough such that
\begin{align}\label{kcomb}
   \frac{(1+k_3)(2-k_0)}{2}\leq (1-k_3)(2-k_2).
\end{align}

 Moreover, since $|\phi_{1,\frac{\alpha}{2}}|$ is bounded, we have for small $y_2-y_*$,
  $$-2\varepsilon^{2} e^{-\frac{A}{x+1}}C_1+2 \varepsilon^{2} e^{-\frac{A}{x+1}}C_1\zeta(y)
 +\gamma  \phi_{1,\frac{\alpha}{2}}(y-y_*)\leq \frac{\gamma}{2},\quad y\in[y_*,y_2].$$ Thus, by \eqref{Ngamm} we have
 \begin{align}\label{capfn1}
 F<0 \quad \text{ in} \quad [0,X]\times[y_*,y_2]\times[N_1,+\infty).
 \end{align}
 In addition, by the definition of $y_*$ in \eqref{ystar0220},
there exists a point $(x^{1},y^{1},z^{1})\in \Omega$  such that $f(x^{1},y^{1},z^{1})>0$ with $y_*<y^{1}\leq y_1.$

In this step, our \textit{goal} is to  prove $F\leq 0$ in $[y_*,y_2]$ which implies
$$\nabla_{\eta}\nabla_{\xi} u_n-\nabla_{\tau_1}^2 \bar{u}-\frac{\varepsilon^2}{2}d_0^2\phi_{1,\frac{\alpha}{2}}-\gamma
 +\gamma  \phi_{1,\frac{\alpha}{2}}(y-y_*)\leq 0 \quad \text{in} \quad [0,X]\times[y_*,y_1]\times[0,N_1],$$
  because
  \begin{align*}
F=f+\gamma  \phi_{1,\frac{\alpha}{2}}(y-y_*) \quad   \text{in} \quad [0,X]\times[y_*,y_1]\times[0,N_1].
\end{align*}
 In particular, we have
 $\nabla_{\eta}\nabla_{\xi} u_n-\nabla_{\tau_1}^2 \bar{u}-\frac{\varepsilon^2}{2}d_0^2\phi_{1,\frac{\alpha}{2}}-\gamma
\leq 0$ in $[y_*,y_1]$ which leads to a contradiction to the definition of $y_*.$

Since $-2\varepsilon^2 e^{-\frac{A}{x+1}} C_1+2\varepsilon^2  e^{-\frac{A}{x+1}}C_1\zeta \leq 0,$ we have
by \eqref{sginf} and \eqref{y2-y1} that
\begin{align}\begin{split}
F<0 \quad & \text{in} \quad D_1 \times[y_*,y_2],\\
\partial_y F\geq0 \quad  & \text{in} \quad D_2\times[y_*,y_1),\\
\partial_y F<0\quad  & \text{in} \quad [0,X]\times(y_1,y_2]\times[0,N_1].\end{split}
\end{align}
Hence,  by \eqref{d1d2},
\begin{align}\label{capF}\begin{split}
\partial_y F_+\geq &0\quad  \text{in} \quad [0,X]\times[y_*,y_1)\times[0,N_1],\\
\partial_y F_+\leq &0\quad  \text{in} \quad [0,X]\times(y_1,y_2]\times[0,N_1].
\end{split}
\end{align}
By the definition of $F$ in \eqref{fmod},  using \eqref{52302072025} and \eqref{7.14}, we have, for $\varepsilon\ll d_0^2$   and $A\geq\frac{ CC_1}{d_0^2}$ that
\begin{align}\label{Feq}\begin{split}
    &\partial_{x}F+(1+\tilde{q})\partial_{y}F
  -\frac{1}{u_{n-1}}\partial_{z} (\frac{u_n}{u_{n-1}}\partial_{z}F) +(-\tilde{b}-\frac{\partial_zu_{n} }{u_{n-1}^2})\partial_zF+\frac{\partial_zu_{n} \partial_zu_{n-1} }{u_{n-1}^3}F+c_1F \\
    \leq &-\varepsilon^2d_0^2(\frac{2\lambda}{u_{n-1}^3}+\frac{A}{6}c_2)\phi_{1,\frac{\alpha}{2}} -(1+\tilde{q})\partial_{y}(K\partial_zu_n+(2-k_1)\varepsilon^2 e^{-\frac{A}{x+1}}C_1)
\\&-[P_1 \phi_{1,\frac{\alpha}{2}}-\frac{\partial_zu_{n} }{u_{n-1}^2}\partial_z\phi_{1,\frac{\alpha}{2}}+\frac{\partial_zu_{n} \partial_zu_{n-1} }{u_{n-1}^3}\phi_{1,\frac{\alpha}{2}}+c_1\phi_{1,\frac{\alpha}{2}}](\frac{\varepsilon^2}{4}d_0^2-\gamma  (y-y_*))
   \\&-2\varepsilon^2 e^{-\frac{A}{x+1}}C_1(1-\zeta(y))(\frac{\partial_zu_{n} \partial_zu_{n-1} }{u_{n-1}^3}+c_1)
+2\varepsilon^2 e^{-\frac{A}{x+1}}C_1(1+\tilde{q})\partial_y \zeta
\\&-2\varepsilon^2 e^{-\frac{A}{x+1}}C_1(1-\zeta(y))\frac{A}{(x+1)^2}
+(1+\tilde{q})\gamma  \phi_{1,\frac{\alpha}{2}}
\\
\leq &-\varepsilon^2d_0^2(\frac{\lambda}{u_{n-1}^3}+\frac{A}{7}c_2)\phi_{1,\frac{\alpha}{2}} -(1+\tilde{q})\partial_{y}(K\partial_zu_n+(2-k_1)\varepsilon^{2}e^{-\frac{A}{x+1}}C_1)
\\&
+(2-k_0)\varepsilon^{2}e^{-\frac{A}{x+1}}C_1\partial_y \zeta,
  \end{split}
\end{align}
where $d_0$ is defined in \eqref{assumpn-2} which is independent of $\varepsilon$ and
$\gamma\ll \varepsilon^3 e^{-A}\epsilon_0$ by \eqref{02072025}.
Note that
 \begin{align*}
    -2\varepsilon^2e^{-\frac{A}{x+1}} C_1(1-\zeta(y))\frac{\partial_zu_{n} \partial_zu_{n-1} }{u_{n-1}^3}\leq 0,
\end{align*} and
$\frac{\partial_zu_{n} \partial_zu_{n-1} }{u_{n-1}^3}+c_1>0$ for $z$ near $0$.  By \eqref{cf12u}, we have
\begin{align}\label{etaeta}
\frac{\partial_zu_{n} \partial_zu_{n-1} }{u_{n-1}^3}+c_1\geq -C e^{-\frac{(z+{\epsilon_0})^2}{x+1}\mu}\quad \text{in }\quad \Omega\cap\{0\leq y\leq Y^*\}.
\end{align}
 By the definition of $y_*,$ we have $F_+(x,y_*,z)=0.$ Then by \eqref{Kplus} and \eqref{capF}, i.e. $\partial_y F_+\geq 0$ in $[y_*,y_1)$ and $\partial_y F_+\leq 0$ in $(y_1,y_2],$ we have
\begin{align*}
   &\int_{y_*}^{y_2} (1+\tilde{q})\partial_{y}(K\partial_zu_n+(2-k_1)\varepsilon^{2}e^{-\frac{A}{x+1}}C_1) F_+ dy\\
   =&(1+\tilde{q})(K\partial_zu_n+(2-k_1)\varepsilon^{2}e^{-\frac{A}{x+1}}C_1) F_+(y_2) -\int_{y_*}^{y_2} \partial_{y}\tilde{q}(K\partial_zu_n+(2-k_1)\varepsilon^{2}e^{-\frac{A}{x+1}}C_1) F_+ dy\\&- \int_{y_*}^{y_1 } (1+\tilde{q})(K\partial_zu_n+(2-k_1)\varepsilon^{2}e^{-\frac{A}{x+1}}C_1) \partial_{y}F_+ dy- \int_{y_1}^{y_2 } (1+\tilde{q})(K\partial_zu_n+(2-k_1)\varepsilon^{2}e^{-\frac{A}{x+1}}C_1) \partial_{y}F_+ dy
   \\ \geq& -\int_{y_*}^{y_2} \varepsilon^4 C\phi_{1,\frac{\alpha}{2}} F_+dy- \int_{y_*}^{y_1 } (2-k_0)\varepsilon^{2}e^{-\frac{A}{x+1}}C_1 \partial_{y}F_+ dy- \int_{y_1}^{y_2 } (2-k_2)\varepsilon^{2}e^{-\frac{A}{x+1}}C_1\partial_{y}F_+ dy
   \\ \geq& -\int_{y_*}^{y_2} \varepsilon^4C \phi_{1,\frac{\alpha}{2}} F_+dy- (2-k_0)\varepsilon^{2}e^{-\frac{A}{x+1}}C_1 F_+ (y_1)+(1-k_3)(2-k_2)\varepsilon^{2}e^{-\frac{A}{x+1}}C_1\,d(x)\,(y_2 -y_1),
\end{align*}
where $k_0$, $k_1$,  $k_2$, $k_3$
and $d(x)$ are defined in \eqref{k0}, \eqref{Kplus}, \eqref{k3dyF} and  \eqref{ddyzeta}
respectively.
Hence, by \eqref{kcomb},
\begin{align}\label{dyfkdzu}\begin{split}
      &\int_{y_*}^{y_2} - \frac{A}{7}c_2d_0^2\varepsilon^2\phi_{1,\frac{\alpha}{2}} F_+-(1+\tilde{q})\partial_{y}F F_+dy
      \\&+\int_{y_*}^{y_2} [-(1+\tilde{q})\partial_{y}(K\partial_zu_n+(2-k_1)\varepsilon^{2}e^{-\frac{A}{x+1}}C_1) +(2-k_0)\varepsilon^{2}e^{-\frac{A}{x+1}}C_1\partial_y \zeta] F_+  dy
      \\ \leq  &
      C\int_{y_*}^{y_2} F_+ ^2 dy,\end{split}
\end{align} where we have used
\begin{align*}
   \int_{y_*}^{y_2}  -(1+\tilde{q})\partial_{y}FF_+dy
   =-\frac{1}{2} (1+\tilde{q})F_+^2(y_2)+\int_{y_*}^{y_2} \frac{1}{2} \partial_{y}\tilde{q}F_+^2dy
    \leq  C\int_{y_*}^{y_2} F_+^2dy,
\end{align*}
and
\begin{align*}
    \int_{y_*}^{y_2} (2-k_0)\varepsilon^{2}e^{-\frac{A}{x+1}}C_1\partial_y \zeta F_+  dy \leq &-(2-k_0)\varepsilon^{2}e^{-\frac{A}{x+1}}C_1
    \frac{1}{y_2-y_1}\int_{y_1}^{y_2}F_+ dy\\
   \leq &-(2-k_0)\varepsilon^{2}e^{-\frac{A}{x+1}}C_1
    \frac{1}{y_2-y_1}\int_{y_1}^{y_2}  (F_+(y_1)+\int_{y_1}^y\partial_yF_+ dy')dy
    \\
   \leq &-(2-k_0)\varepsilon^{2}e^{-\frac{A}{x+1}}C_1F_+(y_1)+(1+k_3)(2-k_0)\varepsilon^{2}e^{-\frac{A}{x+1}}C_1\,d(x)\,
    \frac{y_2-y_1}{2}
   .
\end{align*}
Moreover,  by straightforward calculation, we have
 \begin{align}\label{dz0n1}\begin{split}
 &\int_0^{N_1}[-\frac{1}{u_{n-1}}\partial_{z} (\frac{u_n}{u_{n-1}}\partial_{z}F) +(-\tilde{b}-\frac{\partial_zu_{n} }{u_{n-1}^2})\partial_zF+\frac{\partial_zu_{n} \partial_zu_{n-1} }{u_{n-1}^3}F+c_1F]F_+dz
 \\=& \int_0^\infty-\frac{1}{u_{n-1}}\partial_{z} (\frac{u_n}{u_{n-1}}\partial_{z}F) F_+ +\frac{1}{2}(-\tilde{b}-\frac{\partial_zu_{n} }{u_{n-1}^2})\partial_zF_+^2+(\frac{\partial_zu_{n} \partial_zu_{n-1} }{u_{n-1}^3}+c_1)F_+^2dz
 \\
 =&\int_0^\infty \frac{u_n}{u_{n-1}^2}(\partial_{z}F_+)^2dz-\int_0^\infty \frac{\partial_zu_{n-1}u_n}{2u_{n-1}^3}\partial_{z}(F_+^2)dz
 \\
 &+ \int_0^\infty\frac{1}{2}\partial_z(\tilde{b}+\frac{\partial_zu_{n} }{u_{n-1}^2})F_+^2+(\frac{\partial_zu_{n} \partial_zu_{n-1} }{u_{n-1}^3}+c_1)F_+^2dz
 \\ \geq & -C_{{\epsilon_0}}\int_0^{N_1}F_+^2dz,
\end{split}
 \end{align}
where $C_{{\epsilon_0}}$ is a positive constant depending on ${\epsilon_0}$ and  $N_1$
is defined in \eqref{Ngamm}.
Then combining \eqref{dyfkdzu} and \eqref{dz0n1}, we have
\begin{align*}
\partial_x\int_{y_*}^{y_2} \int_0^{N_1} F_+^2dzdy\leq C_{{\epsilon_0}} \int_{y_*}^{y_2} \int_0^{N_1} F_+^2dzdy \quad \text{in} \quad [0,X].
\end{align*}
 Therefore, we have $\int_{y_*}^{y_2} \int_0^{N_1} F_+^2dydz=0$ in $[0,X]$ because  $F_+(0,y,z)=0$ by Grownwall inequality.
 Therefore,  we prove the goal \eqref{0323goalfneg} and \eqref{0324fremovegam}. We can prove the other direction
\begin{align*}
 -( \nabla_{\eta}\nabla_{\xi} u_n-\nabla_{\tau_1}^2\bar{u})-\frac{\varepsilon^2}{2}d_0^2\phi_{1,\frac{\alpha}{2}}
 \leq 0
 \quad
 \text{in} \quad[0,X]\times[0,Y^*]\times[0,+\infty),
  \end{align*}in a same way.
Then we complete the proof of the theorem.\end{proof}

Similarly, we have the following theorem.

\begin{theorem}\label{1105}
\begin{align*}
 |\nabla_{\eta}^2u_n-\nabla_{\tau_1}^2\bar{u}|\leq &\frac{\varepsilon^2}{2}d_0^2\phi_{1,\frac{\alpha}{2}},\quad|\nabla_{\xi}\nabla_{\eta}u_n-\nabla_{\tau_1}^2\bar{u}|\leq \frac{\varepsilon^2}{2}d_0^2\phi_{1,\frac{\alpha}{2}},
\\|\nabla_{\xi}^2u_n-\nabla_{\tau_1}^2\bar{u}|\leq& \frac{\varepsilon^2}{2}d_0^2\phi_{1,\frac{\alpha}{2}} \quad \text{in} \quad \Omega\cap\{0\leq y\leq Y^*\}.
\end{align*}
\end{theorem}\begin{proof}

\noindent\textit{Step 1} First, we estimate $|\nabla_{\eta}^2u_n-\nabla_{\tau_1}^2\bar{u}|.$

Compared to the proof of Theorem \ref{5.1}, we only need to consider the term \begin{align}\label{03-26-30-1}
\nabla_{\eta}(\frac{1}{u_{n-1}}\partial_z(\frac{\nabla_{\eta}\tilde{q}}{1+\tilde{q}}\frac{u_{n}}{u_{n-1}}\partial_z u_n))
\end{align}
coming from $\nabla_{\eta}f_1$  when we estimate $|\nabla_{\eta}^2u_n-\nabla_{\tau_1}^2\bar{u}|,$ where $f_1$  is defined in \eqref{qn1}.  Since $$|\nabla_{\eta}(\frac{1}{u_{n-1}}\partial_z(\frac{\nabla_{\eta}\tilde{q}}{1+\tilde{q}}\frac{u_{n}}{u_{n-1}}\partial_z u_n)) -(\nabla_{\eta}\partial_z\nabla_{\eta}\tilde{q})\frac{1}{1+\tilde{q}}\frac{u_{n}}{u_{n-1}^2}\partial_z u_n |\leq \frac{C\phi_{1,0}}{u_{n-1}^{3-\frac{3\alpha}{2}}}.$$
We will handle the term
$(\nabla_{\eta}\partial_z\nabla_{\eta}\tilde{q})\frac{1}{1+\tilde{q}}\frac{u_{n}}{u_{n-1}^2}\partial_z u_n  $.

We consider this new structural term because by definition of $\nabla_\eta$, it contains the term
\begin{align}\label{26-03-30-1}
   \frac{ \partial_y\partial_z\nabla_{\eta}\tilde{q}}{1+\tilde{q}}\frac{u_{n}}{u_{n-1}^2}\partial_z u_n=\partial_y(\frac{\partial_z\nabla_{\eta}\tilde{q}}{1+\tilde{q}}\frac{u_{n}}{u_{n-1}^2}\partial_z u_n)+Remainder.
\end{align}The order of $\partial_y\partial_z\nabla_{\eta}\tilde{q}$ is one order higher than  the orders of derivatives in the  induction assumption. We will  handle it similarly as we handled the term $\partial_y(K\partial_z u_n)$ in Theorem  \ref{5.1}. However $$|\frac{\partial_z\nabla_{\eta}\tilde{q}}{1+\tilde{q}}\frac{u_{n}}{u_{n-1}^2}\partial_z u_n|\leq \frac{C}{u_{n-1}^2}\phi_{1,0}.$$
 An extra multiplier $\frac{C}{u_{n-1}^2}$ increasing the growth rate near $z=0$ should be taken into account. Then we modify the auxiliary function $F$  where an extra $u_n^2$ was multiplied to $f$. i.e.
Take $y_2-y_*$ small enough such that
\begin{align}\label{y2u2}
\gamma (1+\|\phi_{1,\frac{\alpha}{2}}\|_{L^\infty})(y_2-y_*)\leq \gamma u^2_{n},
\end{align} and
\begin{align}\label{26-06-14-1}
    f\leq \epsilon_0\quad\text{in}\quad[0,X]\times[y_*,y_2]\times[0,N_1],
\end{align} where we recall the definition of $y_*$ defined in \eqref{ystar0220}.
 Set
$$f=\nabla_{\eta}^2 u_n-\nabla_{\tau_1}^2 \bar{u}-\frac{\varepsilon^2}{2}d_0^2\phi_{1,\frac{\alpha}{2}}-\gamma\quad \text{in}\quad [0,X]\times[0,Y]\times[0,+\infty),$$  and in $[y_*,y_2]$,
 \begin{align}\label{26-03-29-2}
	\begin{split}
 F:=&u_n^2f -2\varepsilon^{2}e^{-\frac{A}{x+1}}C_1+2 \varepsilon^{2}e^{-\frac{A}{x+1}}C_1\zeta(y)
 +\gamma  \phi_{1,\frac{\alpha}{2}}(y-y_*),
 \end{split}
\end{align}where $\zeta(y)$ is defined in \eqref{zeta}.
Then
\begin{align*}
&P_1 F-\frac{\partial_zu_{n} }{u_{n-1}^2}\partial_zF+\frac{\partial_zu_{n} \partial_zu_{n-1} }{u_{n-1}^3}F+c_1F\\=&  \partial_y(Ku_n^2\partial_z u_n)+
\partial_y(\frac{\partial_z\nabla_{\eta}\tilde{q}}{1+\tilde{q}}\frac{u_{n}^3}{u_{n-1}^2}\partial_z u_n)+\cdots,
\end{align*} where the second term on the right hand side corresponds to \eqref{26-03-30-1} multiplied by $u_n^2$ and the first two terms on the right hand side can be estimated like $\partial_y (K\partial_z u_n)$ in the previous proof, since
 $$|Ku_n^2\partial_z u_n|+|\frac{\partial_z\nabla_{\eta}\tilde{q}}{1+\tilde{q}}\frac{u_{n}^3}{u_{n-1}^2}\partial_z u_n|\leq \varepsilon^2 e^{-\frac{A}{x+1}}C.$$
 Next we handle the new terms appearing in the equation of $F$ when we  use  $u_n^2f.$

\textit{Step 1.1 Equation of $f$}

Comparing \eqref{qn1} with \eqref{qn2}, similar to \eqref{520} we have 
\begin{align}\label{26-03-30-4}\begin{split}
    &P_1 f-\frac{\partial_zu_{n} }{u_{n-1}^2}\partial_zf+\frac{\partial_zu_{n} \partial_zu_{n-1} }{u_{n-1}^3}f+c_1f
    \\
    =&\frac{1}{u_n^2}\partial_{y}(Ku_n^2\partial_zu_n)+ \frac{1}{u_n^2}\partial_y( \frac{\partial_z\nabla_{\eta}\tilde{q}}{1+\tilde{q}}\frac{u_{n}^3}{u_{n-1}^2}\partial_z u_n)
 -f_{11}^u-\gamma(\frac{\partial_zu_{n} \partial_zu_{n-1} }{u_{n-1}^3}+c_1)
 \\&-[P_1 \phi_{1,\frac{\alpha}{2}}-\frac{\partial_zu_{n} }{u_{n-1}^2}\partial_z\phi_{1,\frac{\alpha}{2}}+\frac{\partial_zu_{n} \partial_zu_{n-1} }{u_{n-1}^3}\phi_{1,\frac{\alpha}{2}}+c_1\phi_{1,\frac{\alpha}{2}}]\frac{\varepsilon^2}{2}d_0^2
    ,\end{split}
\end{align}  with
\begin{align}
|c_1|\leq  \frac{C \phi_{1,0}}{u_{n-1}^2}, \quad |f_{11}^{u}|\leq C\varepsilon^2\phi_{1,0}+\frac{\varepsilon^2C\phi_{1,0}}{u_{n-1}^{3-\alpha}}.
\end{align}
Here we have used
\begin{align*}
\frac{1}{u_n^2}\partial_{y}(K\partial_zu_n)u_n^2=&
\frac{1}{u_n^2}\partial_{y}(Ku_n^2\partial_zu_n)-\frac{1}{u_n^2}K\partial_{y}u_n^2\partial_zu_n,\\
  \frac{1}{u_n^2}\frac{\partial_y\partial_z\nabla_{\eta}\tilde{q}}{1+\tilde{q}}\frac{u_{n}^3}{u_{n-1}^2}\partial_z u_n= & \frac{1}{u_n^2}\partial_y( \frac{\partial_z\nabla_{\eta}\tilde{q}}{1+\tilde{q}}\frac{u_{n}^3}{u_{n-1}^2}\partial_z u_n)-\frac{1}{u_n^2}\partial_z\nabla_{\eta}\tilde{q}\partial_y  ( \frac{u_{n}^3}{(1+\tilde{q})u_{n-1}^2}\partial_z u_n),
   \\ \nabla_{\eta}\partial_z\nabla_{\eta}\tilde{q}=&\partial_y\partial_z\nabla_{\eta}\tilde{q}
    -\frac{\int_0^z \partial_yv_{n-1}dz'}{v_{n-1}}\partial_z^2\nabla_{\eta}\tilde{q},
\end{align*}and
\begin{align}\label{sm3}
|\frac{\int_0^z \partial_yv_{n-1}dz'}{v_{n-1}}\partial_z^2\nabla_{\eta}\tilde{q}|\leq
\frac{C}{u_{n-1}^{1-\alpha}}\phi_{1,0} \quad\text{in}\quad \Omega\cap\{0\leq y\leq Y^*\},
\end{align} which imply $$|(\nabla_{\eta}\partial_z\nabla_{\eta}\tilde{q}-\partial_y\partial_z\nabla_{\eta}\tilde{q})\frac{1}{1+\tilde{q}}\frac{u_{n}}{u_{n-1}^2}\partial_z u_n  |\leq \frac{C}{u_{n-1}^{2-\alpha}}\phi_{1,0},$$
because $|\frac{1}{1+\tilde{q}}\frac{u_{n}}{u_{n-1}^2}\partial_z u_n|\leq \frac{C}{u_{n-1}}. $
\eqref{sm3} comes from   \begin{align}\label{26-03-17-1}
    |\partial_z^2\nabla_{\eta}\tilde{q}|= |\partial_z^2(\partial_y\tilde{q}-\frac{\int_0^z \partial_yv_{n-1}dz'}{v_{n-1}}\partial_z\tilde{q})|
     \leq \frac{C}{u_{n-1}^{2-\alpha}}\phi_{1,0},
\end{align} which can be derived from the equations of $u_{n-1},$ $v_{n-1}$
and  the induction assumption in \eqref{n-1assumption}.

\textit{Step 1.2 Equation of $g=u_n^2f$}

By definition of $P_1$ in \eqref{p10106}, we have $g=u_{n}^2f$ satisfies

\begin{align*}\begin{split}
    P_1g-\frac{\partial_zu_{n} }{u_{n-1}^2}\partial_zg =&f[P_1u_{n}^2 -\frac{\partial_zu_{n} }{u_{n-1}^2}\partial_zu_n^2] +u_n^2[P_1f-\frac{\partial_zu_{n} }{u_{n-1}^2}\partial_zf] \\& -2\frac{u_{n}}{u_{n-1}^2}\partial_{z}u_n^2\partial_{z}f\\
     =&f[\frac{-2u_{n}(\partial_zu_{n} )^2 }{u_{n-1}^2}-\frac{2u_n(\partial_zu_{n})^2 }{u_{n-1}^2}] +u_n^2[P_1f-\frac{\partial_zu_{n} }{u_{n-1}^2}\partial_zf] \\& -4\frac{u_{n}^2}{u_{n-1}^2}\partial_{z}u_n\partial_{z}f\\
     =&u_n^2[P_1f-\frac{\partial_zu_{n} }{u_{n-1}^2}\partial_zf] -fI_1 -\partial_{z}f I_2,\end{split}\end{align*} with
     \begin{equation*}
        0\leq I_1\leq C\frac{u_n}{u_{n-1}^2}\leq C\frac{1}{u_{n-1}},\quad
        0\leq I_2\leq C,
     \end{equation*}
     where we have used by \eqref{qn0},
      $$P_1 u_n^2=\frac{-2u_{n}(\partial_zu_{n} )^2 }{u_{n-1}^2} .$$
     Then
  \begin{align}\label{26-04-24-1}\begin{split}  &P_1 g-\frac{\partial_zu_{n} }{u_{n-1}^2}\partial_zg+\frac{\partial_zu_{n} \partial_zu_{n-1} }{u_{n-1}^3}g+c_1g\\
  \leq &u_n^2[P_1f-\frac{\partial_zu_{n} }{u_{n-1}^2}\partial_zf+\frac{\partial_zu_{n} \partial_zu_{n-1} }{u_{n-1}^3}f+c_1f]
 \\&  +I_3\partial_zF+I_4 F
   +Ce^{-\frac{(z+\epsilon_0)^2}{x+1}}+Cu_{n-1}^2\phi_{1,\frac{\alpha}{2}}(y-y_*)
    ,\end{split}\end{align} with $|I_3|+|\partial_zI_3|+|I_4|\leq C_{\epsilon_0}$
   Note that
      $$P_1f-\frac{\partial_zu_{n} }{u_{n-1}^2}\partial_zf+\frac{\partial_zu_{n} \partial_zu_{n-1} }{u_{n-1}^3}f+c_1f$$
   follows from  the equation of $f$ in \eqref{26-03-30-4}. 
   Take
$$\gamma\leq e^{-A}\epsilon_0^5,$$ which implies
$\gamma\leq e^{-A}u_{n-1}^5\ll u_{n-1}^5.$ Then by  the  definition of $F$ in \eqref{26-03-29-2}, the following inequalities hold in $[y_*,y_2]$,
\begin{align*}
f\geq &\frac{1}{u_n^2}F-\frac{1}{u_n^2}\gamma  \phi_{1,\frac{\alpha}{2}}(y-y_*)
   \geq \frac{1}{u_n^2}F-u_{n-1}^3\phi_{1,\frac{\alpha}{2}}(y-y_*),
   \\ \frac{u_n^2}{u_{n-1}^2}\partial_z f\geq &\frac{1}{u_{n-1}^2}[\partial_zF-f\partial_zu_n^2-\gamma  \partial_z\phi_{1,\frac{\alpha}{2}}(y-y_*)]\\
   \geq &\frac{1}{u_{n-1}^2}\partial_zF-f\frac{2u_n\partial_zu_n}{u_{n-1}^2}-u_{n-1}^2\phi_{1,\frac{\alpha}{2}}(y-y_*)
 \\
 \geq&\frac{1}{u_{n-1}^2}\partial_zF-\frac{\epsilon_0C}{u_{n-1}}e^{-\frac{(z+\epsilon_0)^2}{x+1}}-u_{n-1}^2\phi_{1,\frac{\alpha}{2}}(y-y_*)
 \\
 \geq&\frac{1}{u_{n-1}^2}\partial_zF-Ce^{-\frac{(z+\epsilon_0)^2}{x+1}}-u_{n-1}^2\phi_{1,\frac{\alpha}{2}}(y-y_*)
   \end{align*}
where we have used  $$-\partial_z\phi_{1,\frac{\alpha}{2}}\geq-\frac{\alpha}{2(z+{\epsilon_0})} \phi_{1,\frac{\alpha}{2}}\geq -\frac{C}{u_{n-1}}\phi_{1,\frac{\alpha}{2}},$$
 which can be derived by  the definition
of $\phi_{1, \frac{\alpha}{2}}$  in \eqref{phi1} and the following calculation,
\begin{equation}\partial_z\phi_{1,\frac{\alpha}{2}}= \left\{
  \begin{aligned}
 \frac{\alpha}{2(z+{\epsilon_0})} \phi_{1,\frac{\alpha}{2}},\quad & 0\leq -\frac{z+{\epsilon_0}}{\sqrt{x+1}}\leq \delta,\\
 0,\quad &  \delta\leq \frac{z+{\epsilon_0}}{\sqrt{x+1}}\leq N,\\
  -\frac{2(z+{\epsilon_0})}{x+1}\mu\phi_{1,\frac{\alpha}{2}} ,\quad\quad & \frac{z+{\epsilon_0}}{\sqrt{x+1}}\geq N .
\end{aligned}\right.\\
\end{equation}

Then the proof proceeds similarly as the proof of   Theorem \ref{5.1}.

\noindent\textit{Step 2} Next, we prove $|\nabla_{\xi}\nabla_{\eta,\xi}u_n-\nabla_{\tau_1}^2\bar{u}|\leq \frac{\varepsilon^2}{2}d_0^2\phi_{1,\frac{\alpha}{2}}$.

 Recall \eqref{qn1} and \eqref{qn2} for the equations of $\nabla_{\eta} u_n$ and $\nabla_{\xi } u_n$. Then the key term  $(1+\tilde{q})\partial_{y}(K\partial_zu_n) $ in \eqref{1104-1} is replaced by $ \partial_x(\mp  K\partial_z u_n )$ and $(1+\tilde{q})\partial_x(\pm K\partial_z u_n)$ when we estimate $\pm(\nabla_{\xi}\nabla_{\eta,\xi}u_n-\nabla_{\tau_1}^2\bar{u})$.
 The proof is similar to the proof of Theorem \ref{5.1} and Step 1. For example, we will estimate $\nabla_{\xi}^2u_n-\nabla_{\tau_1}^2\bar{u}$  where $(1+\tilde{q})\partial_{y}(K\partial_zu_n) $ is replaced by $(1+\tilde{q})\partial_x( K\partial_z u_n)$. 

Set \begin{align*}
 f= ( \nabla_{\xi}^2 u_n-\nabla_{\tau_1}^2\bar{u})-\frac{\varepsilon^2}{2}d_0^2\phi_{1,\frac{\alpha}{2}}-\gamma
 \quad
 \text{in} \quad[0,X]\times[0,Y]\times[0,+\infty).
\end{align*}
 As in the proof of Theorem \ref{5.1}, we apply  proof by contradiction.
 Instead of \eqref{ystar0220}, set $$x_*:=\max \{x_0\in[0,X]|f(x,y,z)\leq 0,\,\,\{0\leq x\leq x_0\} \times[0,Y^*]\times[0,+\infty)\}<X.$$
  By the boundary condition, for some positive constant $y_0<Y^*,$
  \begin{align}\label{d32025}
f<0\quad \text{in}\quad[0,X]\times[0,y_0]\times[0,N_1]\cup[0,X]\times[0,Y^*]\times[0,z_0].
\end{align} Similar
 estimates as in the proof of Theorem \ref{5.1} hold for the sign of $f$ and $\partial_x f$ on $$D_{x_*}=\{x=x_*\}\times[y_0,Y^*]\times[z_0,N_1].$$
Take a constant  $x_2$  such that  $x_2-x_*$ is sufficiently small  and $0<x_2-x_*\ll\min\{1,X-x_*\}$. Then by the continuity, for some domains $D_1$ and
$D_2$ with $ D_1\cup D_2=[0,Y^*]\times[0,N_1]$, \begin{align}\label{26-03-15-3}\begin{split}
f+\gamma  \phi_{1,\frac{\alpha}{2}}(x-x_*) <&0 \quad  \quad\quad\quad\quad \text{in} \quad [x_*,x_2]\times D_1 ,\\
\partial_x (f+\gamma  \phi_{1,\frac{\alpha}{2}}(x-x_*))\geq&0 \quad\quad\quad\quad\quad   \text{in} \quad [x_*,x_2]\times D_2,\end{split}
\end{align}
 with  noting $\partial_x \phi_{1,\frac{\alpha}{2}}\geq 0$. Furthermore, take  $x_2-x_*$   sufficiently small such that
 \begin{align}\label{2026-03-15-1}
    |\frac{ e^{-\frac{A}{x_2+1}}}{ e^{-\frac{A}{x_*+1}}}-1|\ll 1,
 \end{align} which implies
 \begin{align*}
  e^{-\frac{A}{x_2+1}}e^{-\frac{(z+{\epsilon_0})^2}{x+1}\mu}
\leq C\phi_{1,0}\quad \text{in}\quad[x_*,x_2]\times[0,Y^*]\times[0,+\infty),
 \end{align*}
 and instead of \eqref{y2-y1},
  \begin{align}\label{26-03-15-2}
   2 \varepsilon^{2}e^{-A}C_1\frac{1}{x_2-x_1}\gg\max_{[0,X]\times[0,Y^*]\times[0,N_1]} |\partial_x f|+\gamma |\partial_x\big(\phi_{1,\frac{\alpha}{2}}(x-x_*)\big)|.
\end{align} Next, instead of \eqref{fmod}, set$$F= f -2\varepsilon^{2}e^{-\frac{A}{x_2+1}}C_1(1-\zeta)
 +\gamma  \phi_{1,\frac{\alpha}{2}}(x-x_*),
 \quad\text{in} \quad[x_*,x_2]\times[0,Y^*]\times[0,+\infty),$$
where \begin{equation*}\zeta(x)= \left\{
  \begin{aligned}
 1,\quad & x_*\leq x\leq x_1,\\
  1-\frac{x-x_1}{x_2-x_1},\quad & x_1< x\leq x_2,
\end{aligned}\right.\\
\end{equation*}
By \eqref{26-03-15-2}, it holds
\begin{align}\label{26-03-15-4}
 1-k_3 \leq \frac{ -\partial_x F}{d}\leq 1+k_3\quad\text{in} \quad(x_1,x_2],
\end{align}
 with
\begin{align*}
    d=  \frac{2 \varepsilon^{2}e^{-\frac{A}{x_2+1}}C_1}{x_2-x_1},
\end{align*} which corresponds to \eqref{k3dyF}. 
Then by $1-\zeta\geq 0,$ \eqref{26-03-15-3} and \eqref{26-03-15-4}, we have \begin{align*}\begin{split}
\partial_x F_+\geq &0\quad  \text{in} \quad [x_*,x_1)\times[0,Y^*]\times[0,N_1],\\
\partial_x F_+\leq &0\quad  \text{in} \quad (x_1,x_2]\times[0,Y^*]\times[0,N_1].
\end{split}
\end{align*} Then instead of  \eqref{dyfkdzu}, by integrating in $x$ and using  a similar argument as in the proof of Theorem \ref{5.1}, we have \begin{align*}\begin{split}
      &\int_{x_*}^{x_2} - \frac{A}{7}c_2d_0^2\varepsilon^2\phi_{1,\frac{\alpha}{2}} F_+-\partial_{x}F F_+dx
      \\&+\int_{x_*}^{x_2} [-(1+\tilde{q})\partial_{x}(K\partial_zu_n+(2-k_1)\varepsilon^{2}e^{-\frac{A}{x_2+1}}C_1) +(2-k_0)\varepsilon^{2}e^{-\frac{A}{x_2+1}}C_1\partial_x \zeta] F_+  dx
      \\ \leq  &
      C\int_{x_*}^{x_2} F_+ ^2 dx.\end{split}
\end{align*}
Then we have
\begin{align*}
    \int_{x_*}^{x_2}\int_0^{N_1}(1+\tilde{q})\partial_{y}F F_+  dzdx\leq C_{\epsilon_0}\int_{x_*}^{x_2}\int_0^{N_1} F_+ ^2 dzdx
\end{align*}
similar to  \eqref{dz0n1}. Thus
 \begin{align*}\begin{split}
  \partial_{y} \int_{x_*}^{x_2}\int_0^{N_1}(1+\tilde{q})F_+^2dzdx  = & \int_{x_*}^{x_2}\int_0^{N_1}(1+\tilde{q})\partial_{y}F_+^2
  +\partial_{y}\tilde{q}\,F_+^2dzdx  \\ = &\int_{x_*}^{x_2}\int_0^{N_1}2(1+\tilde{q})\partial_{y}F F_+  +\partial_{y}\tilde{q}\,F_+^2dzdx
      \\ \leq  &
      C_{\epsilon_0}\int_{x_*}^{x_2}\int_0^{N_1} F_+ ^2 dzdx
      \\ \leq  &
      C_{\epsilon_0}\int_{x_*}^{x_2}\int_0^{N_1} (1+\tilde{q})F_+ ^2 dzdx\quad\text{in}\quad [0,Y^*].\end{split}
\end{align*}
Hence,  we have  the conclusion $\int_{x_*}^{x_2}\int_0^{N_1}(1+\tilde{q})F_+^2dzdx=0$ in $[0,Y^*]$ because $F_+(x,0,z)=0$ by Grownwall inequality.
This leads to a  contradiction to the definition of $x_*$ and the proof of the theorem is complete.
\end{proof}
\begin{Corollary}\label{d2tan}
It holds that
\begin{align}\label{15}
|\nabla_{\xi,\eta}\nabla_{\xi,\eta}u_n-\nabla_{\xi,\eta}\nabla_{\xi,\eta}\bar{u}|\leq \frac{\varepsilon^2}{2}d_0^2\phi_{1,\frac{\alpha}{2}} \quad \text{in} \quad \Omega\cap\{0\leq y\leq Y^*\}.
\end{align}
\end{Corollary}
\begin{proof}We will  prove \begin{align}\label{1106}
|\nabla_{\eta}\nabla_{\xi}u_n-\nabla_{\eta}\nabla_{\xi}\bar{u}|\leq \frac{\varepsilon^2}{2}d_0^2\phi_{1,\frac{\alpha}{2}} \quad \text{in} \quad \Omega\cap\{0\leq y\leq Y^*\},
\end{align} and the other inequalities  can be derived similarly.

By \eqref{1104-1},
\begin{align}\label{26-03-29-1}\begin{split}
    0=&P_1( \nabla_{\eta}\nabla_{\xi} u_n-\nabla_{\eta}\nabla_{\xi}\bar{u})-\frac{\partial_zu_{n} }{u_{n-1}^2}\partial_z( \nabla_{\eta}\nabla_{\xi} u_n-\nabla_{\eta}\nabla_{\xi}\bar{u})+\tilde{c}( \nabla_{\eta}\nabla_{\xi} u_n-\nabla_{\eta}\nabla_{\xi}\bar{u})\\&+(1+\tilde{q})\partial_{y}(K\partial_zu_n) +F^{new},\end{split}
\end{align} where
$$F^{new}:=-[P_1( \nabla_{\tau_1}^2\bar{u} -\nabla_{\eta}\nabla_{\xi}\bar{u})-\frac{\partial_zu_{n} }{u_{n-1}^2}\partial_z( \nabla_{\tau_1}^2\bar{u}-\nabla_{\eta}\nabla_{\xi}\bar{u})+\tilde{c}( \nabla_{\tau_1}^2\bar{u}-\nabla_{\eta}\nabla_{\xi}\bar{u})]+f_{12}^u,$$
where
$f_{12}^u$ and $\tilde{c}$ are the functions defined in \eqref{f12u} and \eqref{utan2c} respectively.

\noindent\textit{Step 1 Decompose $F^{new}$ into several structural components}

Based on the following claims, we have
\begin{align}\label{26-03-30-5}
    F^{new}:=\partial_yQ +I+f_{12}^u,
\end{align}with
    \begin{align}\label{26-03-30-6}
   |Q|+ |I|\leq \frac{C}{u_{n-1}^2}\phi_{1,0}\quad \text{in} \quad \Omega\cap\{0\leq y\leq Y^*\}.
\end{align}


\textit{Claim 1.}
The first two terms for  $P_1( \nabla_{\tau_1}^2\bar{u} -\nabla_{\eta}\nabla_{\xi}\bar{u})$ with $P_1$  defined in \eqref{p10106} satisfy\begin{align*}
   & \partial_x ( \nabla_{\tau_1}^2\bar{u} -\nabla_{\eta}\nabla_{\xi}\bar{u}) +(1+\tilde{q})\partial_{y}( \nabla_{\tau_1}^2\bar{u} -\nabla_{\eta}\nabla_{\xi}\bar{u})
   \\=&\partial_{y}( Q_1)\partial_z \bar{u}+(1+\tilde{q})\partial_{y}( Q_2)\partial_z \bar{u}+\tilde{I}_1
   \\=&\partial_{y}( Q_3)+\tilde{I}_2,
    \end{align*} where $Q_3=(Q_1+(1+\tilde{q}) Q_2)\partial_z \bar{u}$,
     \begin{align*}
     Q_1&=-\frac{\int_0^z \partial_{x}\partial_x\bar{u}dz'}{\bar{u}}+\frac{\int_0^z \partial_{x}\partial_xu_{n-1}dz'}{u_{n-1}},\\
     Q_2&=-\frac{\int_0^z \partial_{y}\partial_x\bar{u}dz'}{\bar{u}}+\frac{\int_0^z \partial_{y}\partial_xu_{n-1}dz'}{u_{n-1}},
\end{align*}
and $\tilde{I}_1, \,\tilde{I}_2$ do not contain  third-order tangential derivatives which are one order higher than the orders of derivatives in the induction assumption. In particular, by \eqref{n-1assumption}, $Q_3$ satisfies
   \begin{align*}
    |Q_3|\leq 2\varepsilon^2 e^{-\frac{A}{x+1}}C,\quad|\tilde{I}_2|\leq C\frac{1}{u_{n}^2}\phi_{1,0}  \quad \text{in} \quad \Omega\cap\{0\leq y\leq Y^*\},
   \end{align*} for some constant $C$ independent of $\varepsilon.$

\textit{Claim 2.}
The  fourth and fifth terms of
    $P_1( \nabla_{\tau_1}^2\bar{u} -\nabla_{\eta}\nabla_{\xi}\bar{u})$  are
     \begin{align}\begin{split}
   &-\frac{u_{n}}{u_{n-1}^2}\partial_{z}^2(( \nabla_{\tau_1}^2\bar{u} -\nabla_{\eta}\nabla_{\xi}\bar{u}))-\frac{1}{u_{n-1}}\partial_z(\frac{u_{n}}{u_{n-1}})\,\partial_{z} ( \nabla_{\tau_1}^2\bar{u} -\nabla_{\eta}\nabla_{\xi}\bar{u})\\
  =& \partial_y( Q_6)+I_6 ,\end{split}\end{align} with
    \begin{align*}
   |Q_6|+ |I_6|\leq \frac{C}{u_{n-1}^2}\phi_{1,0}\quad \text{in} \quad \Omega\cap\{0\leq y\leq Y^*\}.
\end{align*}

\textit{Claim 3.}
\begin{align*}
    -\frac{\partial_zu_{n} }{u_{n-1}^2}\partial_z( \nabla_{\tau_1}^2\bar{u}-\nabla_{\eta}\nabla_{\xi}\bar{u})+\tilde{c}( \nabla_{\tau_1}^2\bar{u}-\nabla_{\eta}\nabla_{\xi}\bar{u})
    = \partial_y( Q_9)+I_9,
\end{align*}with
    \begin{align*}
   |Q_9|+ |I_9|\leq \frac{C}{u_{n-1}^2}\phi_{1,0}\quad \text{in} \quad \Omega\cap\{0\leq y\leq Y^*\}.
\end{align*}

  \textit{Proof of the claims:}

  Growth rates of $\bar{u}$ and its derivatives are given in subsection \ref{1027}.

  For Claim 1,
    by the  definitions,
    \begin{align*}
      \nabla_{\eta}=\partial_y-\frac{\int_0^z \partial_yv_{n-1}dz'}{v_{n-1}}\partial_z,
      \quad \nabla_{\xi}&=\partial_x-\frac{\int_0^z \partial_xu_{n-1}dz'}{u_{n-1}}\partial_z.
    \end{align*}Since
    \begin{align*}\begin{split}
       \partial_{x,y}\partial_{y}(-\frac{\int_0^z \partial_xu_{n-1}dz'}{u_{n-1}})\partial_z \bar{u}=&
        \partial_{y}\partial_{x,y}(-\frac{\int_0^z \partial_xu_{n-1}dz'}{u_{n-1}})\partial_z \bar{u} \\
        :=& \partial_{y}( -\frac{\int_0^z \partial_{x,y}\partial_xu_{n-1}dz'}{u_{n-1}})\partial_z \bar{u}+I_0,\end{split}
    \end{align*} we have
      \begin{align}\label{26-03-15-5-1-1}\begin{split}
       \partial_{x,y}\nabla_{\eta}\nabla_{\xi}\bar{u}:=&
       \partial_{x,y} \partial_{y}\partial_x\bar{u}+\partial_{y}( -\frac{\int_0^z \partial_{x,y}\partial_xu_{n-1}dz'}{u_{n-1}})\partial_z \bar{u}+I_1
       ,\end{split}
    \end{align}
    where $I_0$ and $I_1$ do not contain the third-order tangential derivatives $\partial_{x,y}^3u_{n-1}.$
    By the definitions, $\nabla_{\tau_1} =  \partial_x-\frac{\int_0^z \partial_{x}\bar{u}dz'}{\bar{u}}\partial_z$ and $\nabla_{\tau_2} =  \partial_y-\frac{\int_0^z \partial_{y}\bar{u}dz'}{\bar{u}}\partial_z.$ By the symmetry of $\bar{u},$ we have $$\nabla_{\tau_1}^2\bar{u}=\nabla_{\tau_2}\nabla_{\tau_1}\bar{u}.$$
    Hence, similarly to \eqref{26-03-15-5-1-1}, it holds
    \begin{align*}
       \partial_{x,y} \nabla_{\tau_1}^2\bar{u} :=  \partial_{x,y}\partial_{y}\partial_x\bar{u}+\partial_{y}( -\frac{\int_0^z \partial_{x,y}\partial_x\bar{u}dz'}{\bar{u}})\partial_z \bar{u}+I_2,
    \end{align*} where $I_2$ does not contain the third-order tangential derivatives $\partial_{x,y}^3\bar{u}.$
    In summary,
    \begin{align*}
       \partial_{x,y} ( \nabla_{\tau_1}^2\bar{u} -\nabla_{\eta}\nabla_{\xi}\bar{u}):=  \partial_{y}( -\frac{\int_0^z \partial_{x,y}\partial_x\bar{u}dz'}{\bar{u}}+\frac{\int_0^z \partial_{x,y}\partial_xu_{n-1}dz'}{u_{n-1}})\partial_z \bar{u}+I_3,
    \end{align*}where $I_3$ does not contain the third-order tangential derivatives. Then we can prove  claim 1 by straight forward calculation and growth estimates.

    Similarly, for Claim 2, note that
    \begin{align*}\begin{split}
       \partial_{z}\partial_{y}(-\frac{\int_0^z \partial_xu_{n-1}dz'}{u_{n-1}})\partial_z \bar{u}=&
        \partial_{y}\partial_{z}(-\frac{\int_0^z \partial_xu_{n-1}dz'}{u_{n-1}})\partial_z \bar{u} \\
        =& \partial_{y}( -\frac{ \partial_xu_{n-1}}{u_{n-1}}+\frac{\int_0^z \partial_xu_{n-1}dz'}{u_{n-1}^2}\partial_zu_{n-1})\partial_z \bar{u}
        \\
        =& \partial_{y}Q_4-( -\frac{ \partial_xu_{n-1}}{u_{n-1}}+\frac{\int_0^z \partial_xu_{n-1}dz'}{u_{n-1}^2}\partial_zu_{n-1})\partial_z\partial_{y} \bar{u},
        \\
        \partial_{z}^2\partial_{y}(-\frac{\int_0^z \partial_xu_{n-1}dz'}{u_{n-1}})\partial_z \bar{u}
        =& \partial_{y}\partial_z( -\frac{ \partial_xu_{n-1}}{u_{n-1}}+\frac{\int_0^z \partial_xu_{n-1}dz'}{u_{n-1}^2}\partial_zu_{n-1})\partial_z \bar{u}\\
        =&\partial_{y} Q_5-\partial_z( -\frac{ \partial_xu_{n-1}}{u_{n-1}}+\frac{\int_0^z \partial_xu_{n-1}dz'}{u_{n-1}^2}\partial_zu_{n-1})\partial_z \partial_{y}\bar{u},\end{split}
    \end{align*} where
     \begin{align*}
       Q_4=& ( -\frac{ \partial_xu_{n-1}}{u_{n-1}}+\frac{\int_0^z \partial_xu_{n-1}dz'}{u_{n-1}^2}\partial_zu_{n-1})\partial_z \bar{u},\\
       Q_5=&\partial_z( -\frac{ \partial_xu_{n-1}}{u_{n-1}}+\frac{\int_0^z \partial_xu_{n-1}dz'}{u_{n-1}^2}\partial_zu_{n-1})\partial_z \bar{u}.
\end{align*}
Here,  the order of  differentiation in $\partial_{y}\partial_z \partial_xu_{n-1}$ is one order higher than the order of derivatives in the induction assumption. Then by straight forward calculation,  we have
      \begin{align}\label{26-03-15-5-1}\begin{split}
       \partial_{z}( \nabla_{\tau_1}^2\bar{u}-\nabla_{\eta}\nabla_{\xi}\bar{u}):=&\partial_y Q_4+I_4,\\
       \partial_{z}^2( \nabla_{\tau_1}^2\bar{u}-\nabla_{\eta}\nabla_{\xi}\bar{u}):=&\partial_y Q_5+I_5,\end{split}
    \end{align} with
    \begin{align}\label{26-03-30-2}
      |I_4|+u_{n-1}|I_5| \leq C\phi_{1,0},\quad|Q_4|\leq C\phi_{1,0},\quad |Q_5|\leq\frac{C}{u_{n-1}}\phi_{1,0}\quad \text{in} \quad \Omega\cap\{0\leq y\leq Y^*\},
    \end{align} by the induction assumptions and in particular, by noting $$|\partial_y u_{n-1}|\leq Cu_{n-1}.$$  Moreover, since $|\partial_y Q_4|\leq\frac{C}{u_{n-1}}\phi_{1,0}$ by induction assumption,
    \begin{align}\label{26-03-30-3}\begin{split}
       u_{n-1}|\partial_{z}( \nabla_{\tau_1}^2\bar{u}-\nabla_{\eta}\nabla_{\xi}\bar{u})|\leq&C\phi_{1,0},
   \quad \text{in} \quad \Omega\cap\{0\leq y\leq Y^*\}. \end{split}\end{align}

Then by \eqref{p10106} and \eqref{26-03-15-5-1}, the  fourth and fifth terms of
    $P_1( \nabla_{\tau_1}^2\bar{u} -\nabla_{\eta}\nabla_{\xi}\bar{u})$  are
     \begin{align}\begin{split}
   &-\frac{u_{n}}{u_{n-1}^2}\partial_{z}^2( \nabla_{\tau_1}^2\bar{u} -\nabla_{\eta}\nabla_{\xi}\bar{u})-\frac{1}{u_{n-1}}\partial_z(\frac{u_{n}}{u_{n-1}})\,\partial_{z} ( \nabla_{\tau_1}^2\bar{u} -\nabla_{\eta}\nabla_{\xi}\bar{u})\\
   =& \partial_y (-\frac{u_{n}}{u_{n-1}^2}Q_5)-Q_5\partial_y (-\frac{u_{n}}{u_{n-1}^2})+\partial_y[(\frac{u_{n}}{u_{n-1}^3}\partial_zu_{n-1}-\frac{\partial_zu_{n}}{u_{n-1}^2})Q_4]
   \\&-Q_4\partial_y(\frac{u_{n}}{u_{n-1}^3}\partial_zu_{n-1}-\frac{\partial_zu_{n}}{u_{n-1}^2})
   -\frac{u_{n}}{u_{n-1}^2}I_5+(\frac{u_{n}}{u_{n-1}^3}\partial_zu_{n-1}-\frac{\partial_zu_{n}}{u_{n-1}^2})I_4\\
  =& \partial_y( Q_6)+I_6 ,\end{split}\end{align} with
    \begin{align*}
   |Q_6|+ |I_6|\leq \frac{C}{u_{n-1}^2}\phi_{1,0}\quad \text{in} \quad \Omega\cap\{0\leq y\leq Y^*\},
\end{align*}
by the induction assumptions and in particular, by noting $|\partial_y u_{n-1}|\leq Cu_{n-1}.$

For Claim 3, by the same method in the proof of claim 2, we can decompose $-\frac{\partial_zu_{n} }{u_{n-1}^2}\partial_z( \nabla_{\tau_1}^2\bar{u}-\nabla_{\eta}\nabla_{\xi}\bar{u})$ as follows.
Now we decompose the term $\tilde{c}( \nabla_{\tau_1}^2\bar{u}-\nabla_{\eta}\nabla_{\xi}\bar{u}).$

Since  \begin{align*}\begin{split}
     \nabla_{\eta}\nabla_{\xi}\bar{u}=&
       \partial_{y}\partial_x\bar{u}+\partial_{y}( -\frac{\int_0^z \partial_xu_{n-1}dz'}{u_{n-1}}\partial_z \bar{u})
       \\&-\frac{\int_0^z \partial_yv_{n-1}dz'}{v_{n-1}}\partial_z\partial_x\bar{u}-\frac{\int_0^z \partial_yv_{n-1}dz'}{v_{n-1}}\partial_z( -\frac{\int_0^z \partial_xu_{n-1}dz'}{u_{n-1}}\partial_z \bar{u}),\end{split}
    \end{align*}it holds
       $$\nabla_{\tau_1}^2\bar{u}-\nabla_{\eta}\nabla_{\xi}\bar{u}=\partial_y( Q_7)+I_7,$$ where $Q_7= -\frac{\int_0^z \partial_x\bar{u}dz'}{\bar{u}}\partial_z \bar{u}+\frac{\int_0^z \partial_xu_{n-1}dz'}{u_{n-1}}\partial_z \bar{u}$ with
    \begin{align*}
      |Q_7|+ | I_7|\leq&Cu_{n-1}\phi_{1,0}
   \quad \text{in} \quad \Omega\cap\{0\leq y\leq Y^*\}.
    \end{align*}

Moreover, by the expression of $\tilde{c}$ in \eqref{utan2c}, the induction assumptions and in particular, by  noting $|\partial_y u_{n-1}|\leq Cu_{n-1},$ we have
\begin{align*}
    |\tilde{c}|+|\partial_y\tilde{c}|\leq \frac{C}{u_{n-1}^3}e^{-\frac{(z+{\epsilon_0})^2}{x+1}\mu}\quad \text{in} \quad \Omega\cap\{0\leq y\leq Y^*\}.
\end{align*} Then we have
\begin{align*}\begin{split}
\tilde{c}( \nabla_{\tau_1}^2\bar{u}-\nabla_{\eta}\nabla_{\xi}\bar{u})=&\partial_y( Q_7\tilde{c})+I_7\tilde{c}-\partial_y\tilde{c}Q_7\\
=&\partial_y( Q_8)+I_8
   ,\end{split}\end{align*} with
  \begin{align*}
      |Q_8|+ | I_8|\leq&\frac{C}{u_{n-1}^2}\phi_{1,0}
   \quad \text{in} \quad \Omega\cap\{0\leq y\leq Y^*\}.
    \end{align*}
      In summary, all the three claims hold.

\noindent\textit{Step 2  Reduction to the case similar to \eqref{26-03-30-4}}

Next, by \eqref{26-03-29-1} and \eqref{26-03-30-5}-\eqref{26-03-30-6},  $\nabla_{\eta}\nabla_{\xi} u_n-\nabla_{\eta}\nabla_{\xi}\bar{u}$
satisfies the following equation
\begin{align}\begin{split}
    0
    =&P_1( \nabla_{\eta}\nabla_{\xi} u_n-\nabla_{\eta}\nabla_{\xi}\bar{u})-\frac{\partial_zu_{n} }{u_{n-1}^2}\partial_z( \nabla_{\eta}\nabla_{\xi} u_n-\nabla_{\eta}\nabla_{\xi}\bar{u})+\tilde{c}( \nabla_{\eta}\nabla_{\xi} u_n-\nabla_{\eta}\nabla_{\xi}\bar{u})\\&+\frac{1}{u_n^2}\partial_{y}E +f_{12}^{new},\end{split}
\end{align} with
\begin{align*}
    E=[(1+\tilde{q})K\partial_zu_n+Q]u_n^2,
\end{align*}
and
\begin{align}\label{26-03-15-7}
  |E|\leq 2\varepsilon^2 e^{-\frac{A}{x+1}}C_1  \quad \text{in} \quad \Omega\cap\{0\leq y\leq Y^*\},
\end{align}
 where
\begin{align*}f_{12}^{new}:=&I-\partial_{y}\tilde{q}\,K\partial_zu_n
-2\frac{\partial_y u_n}{u_n}[(1+\tilde{q})K\partial_zu_n+Q]
+f_{12}^u.
 \end{align*}
 Then by $|\partial_y u_n|\leq Cu_n$, \eqref{26-03-30-6} and \eqref{cf12u}, $f_{12}^{new}$ satisfies
 \begin{align}
|f_{12}^{new}|\leq  \frac{C}{u_{n-1}^{2}}\phi_{1,0} +\frac{\varepsilon^2C\phi_{1,0}}{u_{n-1}^{3-\alpha}},
\end{align} which can be controlled by the barrier function. Then we derive \eqref{1106} by using  the same method as for the term \eqref{26-03-30-1} through the estimation on  $F=u_n^2 f+remainder$ defined in \eqref{26-03-29-2} in Theorem \ref{1105}.
\end{proof}

 \subsection{ Estimates on $\partial_z\nabla_{\eta,\xi}u_n$}
 First, we estimate in the domain $\{ z\in[0,\delta_3]\}\cap\Omega\cap\{0\leq y\leq Y^*\}$ for some small positive constant $\delta_3$ and then estimate in the domain $\{ z\in[\delta_3,+
 \infty)\}\cap\Omega\cap\{0\leq y\leq Y^*\}$.
 \begin{theorem}For a small positive constant $\delta_3$,  it holds
 \begin{align}\label{zh1}
    |\partial_z\nabla_{\eta,\xi}u_n-\partial_z\nabla_{\tau_1}\bar{u} |\leq \varepsilon^4  \phi_{1,\alpha} \quad \text{in} \quad \Omega\cap\{0\leq y\leq Y^*\}\cap\{0\leq z\leq \delta_3\}.
 \end{align}
 \end{theorem}
\begin{proof} To estimate $|\partial_z\nabla_{\xi}u_n-\partial_z\nabla_{\tau_1}\bar{u} |,$
	firstly by straightforward  calculation, we have
 \begin{align*}
    \frac{1}{2}(\partial_z\nabla_{\xi}u_{n}^2 - \partial_z \nabla_{\tau_1} \bar{u}^2)=&\partial_zu_{n}\nabla_{\xi}u_{n}+u_{n}\partial_z\nabla_{\xi}u_{n}
    -\partial_z\bar{u}\nabla_{\tau_1}\bar{u}-\bar{u}\partial_z\nabla_{\tau_1}\bar{u}
 \\
 =&\partial_zu_{n}G+u_{n}\partial_z G
  +\nabla_{\tau_1}\bar{u}(\partial_zu_{n}-\partial_z\bar{u}) +(u_{n}-\bar{u})\partial_z\nabla_{\tau_1}\bar{u},
\end{align*}
 where $G=\nabla_{\xi}u_{n}-\nabla_{\tau_1}\bar{u}$.

Next, we \textit{claim} that
  \begin{align*}
    |\partial_z\nabla_{\xi}u_{n}^2 - \partial_z \nabla_{\tau_1} \bar{u}^2|\leq \varepsilon^5C\phi_{1,2\alpha}u_n+ e^{-\frac{A}{x+1}} C\bar{u}^2 \quad \text{in}\quad \{z\leq \varepsilon^3\}\cap\Omega\cap\{y\leq Y^*\}.
 \end{align*}
If the above claim holds, then we have
\begin{align}\label{ungback}
\pm(\partial_zu_{n}G+u_{n}\partial_z G)=\pm\partial_z(u_{n} G) \leq \varepsilon^5C\phi_{1,\alpha}u_n+ e^{-\frac{A}{x+1}} C\bar{u}^2.
\end{align}
Then
for $z$ near $0,$ by  $|G|\leq  e^{-\frac{A}{x+1}} C{\epsilon_0}^2$ on $z=0$, we have\begin{align*}
   | G| \leq \varepsilon^5C\phi_{1,\alpha}u_n+ e^{-\frac{A}{x+1}} C\bar{u}^2.
\end{align*}
Substituting this to \eqref{ungback},
   for sufficiently small  positive constant $\delta_3$ such that
   \begin{align}\label{delta3}
   \delta_3=\varepsilon^{\frac{5}{1-\alpha}},
   \end{align}
we have
\begin{align}\label{zeta-2025-05-23}
 |\partial_z\nabla_{\xi}u_{n} - \partial_z \nabla_{\tau_1} \bar{u}|\leq
 \varepsilon^5 e^{-\frac{A}{x+1}} C\bar{u}^{\alpha}\quad\text{in}\quad\{ z\in[0,\delta_3]\}\cap\Omega\cap\{0\leq y\leq Y^*\}.
\end{align}Similarly, we can prove
\begin{align}
 |\partial_z\nabla_{\eta}u_{n} - \partial_z \nabla_{\tau_1} \bar{u}|\leq
 \varepsilon^5 e^{-\frac{A}{x+1}} C\bar{u}^{\alpha}\quad\text{in}\quad\{ z\in[0,\delta_3]\}\cap\Omega\cap\{0\leq y\leq Y^*\}.
\end{align}

\textit{Proof of the claim}

  By \eqref{0823}, we have
\begin{align}\label{ntau1}
\nabla_{n}^2\bar{u}^2=4\nabla_{\tau_1} \bar{u},\quad \nabla_{n}^2\nabla_{\tau_1}\bar{u}^2=
\nabla_{\tau_1}\nabla_{n}^2\bar{u}^2=4\nabla_{\tau_1} ^2\bar{u}.
\end{align}
By \eqref{ppun2},
straightforward calculation yields
\begin{align}\label{xz5}\begin{split}
 I:=u_{n}\nabla_{\psi}^2\nabla_{\xi}u_{n}^2 - \bar{u}\nabla_{n}^2 \nabla_{\tau_1} \bar{u}^2= &\nabla_{\xi} ^2u_{n}^2 +\nabla_{\xi}\tilde{q}\nabla_{\eta}u_{n}^2+(1+\tilde{q})\nabla_{\xi}\nabla_{\eta}u_{n}^2  - 2\nabla_{\tau_1} ^2   \bar{u}^2\\
    &
    -\nabla_{\xi}u_{n}\,\nabla_{\psi}^2u_{n}^2+\nabla_{\tau_1}\bar{u}\nabla_{n}^2  \bar{u}^2 .\end{split}\end{align}

   Step 1 We will prove \eqref{26-03-31-1}.

    Note by \eqref{ntau1},
    \begin{align*}
       | \nabla_{\xi}u_{n}\,\nabla_{\psi}^2u_{n}^2-\nabla_{\tau_1}\bar{u}\nabla_{n}^2  \bar{u}^2|\leq& | (\nabla_{\xi}u_{n}-\nabla_{\tau_1}\bar{u})\nabla_{n}^2  \bar{u}^2|+|\nabla_{\xi}u_{n}(
       \nabla_{\psi}^2u_{n}^2-\nabla_{n}^2  \bar{u}^2)|
       \\ \leq &C\bar{u} u_{n-1}\phi_{1,0}.
    \end{align*}
    In addition,
    \begin{align*}
        |\nabla_{\xi}\tilde{q}\nabla_{\eta}u_{n}^2
        +\tilde{q}\nabla_{\xi}\nabla_{\eta}u_{n}^2 |=&
        |2u_{n}\nabla_{\xi}\tilde{q}\nabla_{\eta}u_{n}
        +2\tilde{q}(u_{n}\nabla_{\xi}\nabla_{\eta}u_{n} +
        \nabla_{\xi}u_{n}\nabla_{\eta}u_{n} )|\leq \varepsilon^2Cu_{n-1}^{1+\frac{\alpha}{2}}\phi_{1,0},
    \end{align*}
     and \begin{align*}
        |\nabla_{\xi} ^2u_{n}^2 +\nabla_{\xi}\nabla_{\eta}u_{n}^2  -2\nabla_{\tau_1} ^2   \bar{u}^2|
        =& |2(\nabla_{\xi} u_{n})^2+2 u_{n}\nabla_{\xi}^2 u_{n} +2\nabla_{\xi} u_{n}\nabla_{\eta} u_{n}+2 u_{n}\nabla_{\xi}\nabla_{\eta} u_{n}
         \\&-4(\nabla_{\tau_1} \bar{u})^2 -4\bar{u}\nabla_{\tau_1}^2\bar{u}|
        \\ \leq & C\varepsilon^2 d_0^2\bar{u}^{\frac{\alpha}{2}} u_{n-1} \phi_{1,0}+ C\bar{u}u_{n-1}\phi_{1,0}.
      \end{align*} Therefore, by \eqref{xz5},
we have
   \begin{align*}
    |I|=|u_{n}\nabla_{\psi}^2\nabla_{\xi}u_{n}^2 - \bar{u}\nabla_{n}^2 \nabla_{\tau_1} \bar{u}^2| \leq   C\varepsilon^2  u_{n-1}^{1+\frac{\alpha}{2}} \phi_{1,0}+ C\bar{u}u_{n-1}\phi_{1,0}.
   \end{align*}
On the other hand,
\begin{align}\label{26-03-31-2}\begin{split}
I=& u_{n}\nabla_{\psi}^2\nabla_{\xi}u_{n}^2 - \bar{u}\nabla_{n}^2 \nabla_{\tau_1} \bar{u}^2\\= & u_{n}\nabla_{\psi}(\nabla_{\psi}\nabla_{\xi}u_{n}^2 - \nabla_{n} \nabla_{\tau_1} \bar{u}^2)+ (u_{n}\nabla_{\psi}- \bar{u} \nabla_{n}) \nabla_{n} \nabla_{\tau_1} \bar{u}^2\\= &
  u_{n}\nabla_{\psi}(\nabla_{\psi}\nabla_{\xi}u_{n}^2 - \nabla_{n} \nabla_{\tau_1} \bar{u}^2)+\bar{u}(\frac{u_{n}}{u_{n-1}}-1)\nabla_{n}^2 \nabla_{\tau_1} \bar{u}^2
  \\= &
 \frac{ u_{n}}{u_{n-1}}\partial_z(\frac{1}{u_{n-1}}\partial_z\nabla_{\xi}u_{n}^2 - \frac{1}{\bar{u}}\partial_z \nabla_{\tau_1} \bar{u}^2)+\bar{u}(\frac{u_{n}}{u_{n-1}}-1)\nabla_{n}^2 \nabla_{\tau_1} \bar{u}^2,\end{split}
 \end{align}
 where we have used that
 \begin{align*}
    (u_{n}\nabla_{\psi}- \bar{u} \nabla_{n}) \nabla_{n} \nabla_{\tau_1} \bar{u}^2=&  (\frac{u_{n}}{u_{n-1}}-1)\partial_z\nabla_{n} \nabla_{\tau_1} \bar{u}^2\\
    =&\bar{u}(\frac{u_{n}}{u_{n-1}}-1)\nabla_{n}^2 \nabla_{\tau_1} \bar{u}^2.
 \end{align*}
Then by  \eqref{ntau1} and \eqref{26-03-31-2}, we have
 \begin{align}\label{26-03-31-1}
    |\partial_z(\frac{1}{u_{n-1}}\partial_z\nabla_{\xi}u_{n}^2 - \frac{1}{\bar{u}}\partial_z \nabla_{\tau_1} \bar{u}^2) | \leq  C\varepsilon^2  u_{n-1}^{1+\frac{\alpha}{2}} \phi_{1,0}+ C\bar{u}u_{n-1}\phi_{1,0}.
 \end{align}

    Step 2 We will prove
 $ |\frac{1}{u_{n-1}}\partial_z\nabla_{\xi}u_{n}^2 - \frac{1}{\bar{u}}\partial_z \nabla_{\tau_1} \bar{u}^2|\leq \varepsilon^2  e^{-\frac{A}{x+1}}C{\epsilon_0}$ at $z=0$, through the compatible  boundary condition.

   By the definition of $\nabla_{\psi}$ and \eqref{kh1}, we have
 \begin{align}\label{26-03-30-7}\begin{split}
     \frac{1}{u_{n-1}}\partial_z\nabla_{\xi}u_{n}^2 - \frac{1}{\bar{u}}\partial_z \nabla_{\tau_1} \bar{u}^2=&\nabla_{\psi}\nabla_{\xi}u_{n}^2 - \nabla_{n} \nabla_{\tau_1} \bar{u}^2\\= & \nabla_{\xi}\nabla_{\psi}u_{n}^2 -   \nabla_{\tau_1} \nabla_{n} \bar{u}^2\\
     = & \nabla_{\xi}(2\frac{u_n}{u_{n-1}}\partial_zu_{n}) -    \nabla_{\tau_1} (2\partial_z\bar{u}).\end{split}
 \end{align}
And
 by  the definition of $ \nabla_{\xi}$, we have when $z=0,$
\begin{align}\label{xz6}\begin{split}
\nabla_{\xi}(\frac{u_n}{u_{n-1}}\partial_zu_{n})=&\partial_x(\frac{u_n}{u_{n-1}}\partial_zu_{n})
-\frac{\int_0^z \partial_{x}u_{n-1}dz'}{u_{n-1}}\partial_z(\frac{u_n}{u_{n-1}}\partial_zu_{n})
\\=&\partial_x(\frac{u_n}{u_{n-1}}\partial_zu_{n})\end{split}\end{align}
because  $\frac{\int_0^z \partial_{x}u_{n-1}dz'}{u_{n-1}}=0$ when $z=0.$
So does $\nabla_{\tau_1}\partial_z \bar{u}=\partial_{x}\partial_z \bar{u}$ at $z=0.$ Substituting back to \eqref{26-03-30-7}, we complete step 2 by the compatible boundary condition.

 Then combining step 1 and step 2, by multiplying $\bar{u}$ to the integral of \eqref{26-03-31-1},
 we have
 \begin{align*}
    |\frac{\bar{u}}{u_{n-1}}\partial_z\nabla_{\xi}u_{n}^2 - \partial_z \nabla_{\tau_1} \bar{u}^2|\leq  C\bar{u}^2(\varepsilon^2  u_{n-1}^{1+\frac{\alpha}{2}} \phi_{1,0}+ C\bar{u}u_{n-1}\phi_{1,0})+  e^{-\frac{A}{x+1}} C\bar{u}^2.
 \end{align*} Since
 $|\frac{\bar{u}}{u_{n-1}}-1|\leq \varepsilon^5\phi_{1,2\alpha}$
  and
  \begin{align*}
  |\partial_z\nabla_{\xi}u_{n}^2|=&|2\partial_z(u_{n}\nabla_{\xi}u_{n})|
  =|2\partial_zu_{n}\nabla_{\xi}u_{n}+2u_{n}\partial_z\nabla_{\xi}u_{n}|
  \leq  Cu_ne^{-\frac{(z+{\epsilon_0})^2}{x+1}\mu},
  \end{align*}
we  have
  \begin{align*}
    |\partial_z\nabla_{\xi}u_{n}^2 - \partial_z \nabla_{\tau_1} \bar{u}^2|\leq \varepsilon^5C\phi_{1,2\alpha}u_n+  e^{-\frac{A}{x+1}}C\bar{u}^2\quad \text{in}\quad \{z\leq \varepsilon^3\}\cap\Omega\cap\{y\leq Y^*\}
 \end{align*}
which gives the claim. And then we complete the proof of theorem.
 \end{proof}

\begin{remark}\label{yw3}
	By \eqref{zh1} and Lemma \ref{xz1104} in Appendix, we have for $\varepsilon$ being  sufficiently small compared to $\alpha,$ it holds
	\begin{align}\label{zh2}
		|\partial_z\nabla_{\eta,\xi}u_n-\partial_z\nabla_{\eta,\xi}\bar{u} |\leq  \frac{ \varepsilon^2}{2+\frac{\alpha}{3}} \phi_{1,\alpha} \quad \text{in} \quad \Omega\cap\{0\leq y\leq Y^*\}\cap\{0\leq z\leq \delta_\varepsilon\},
	\end{align}where $\delta_\varepsilon=\min\{\varepsilon^6 ,\delta_3\}.$
\end{remark}

Now we estimate the derivatives in the domain away from $z=0.$
\begin{theorem} It holds
 \begin{align}\label{1116-1}
    |\partial_z\nabla_{\eta,\xi}u_n-\partial_z\nabla_{\tau_1}\bar{u} |\leq \varepsilon^4  \phi_{1,\alpha} \quad \text{in} \quad \Omega\cap\{0\leq y\leq Y^*\},
 \end{align} which implies
  \begin{align}\label{1116-2}
    |\nabla_{\eta,\xi}u_n-\nabla_{\tau_1}\bar{u} |\leq \varepsilon^3 \phi_{1,1} \quad \text{in} \quad \Omega\cap\{0\leq y\leq Y^*\}.
 \end{align}
 \end{theorem}

\begin{proof}
Set $$H=\nabla_{\eta,\xi}u_n-\nabla_{\tau_1}\bar{u}.$$ We estimate  $\nabla_{\psi }H$ in $(0,X]\times(0,Y^*]\times[\delta_3,+\infty).$
By \eqref{P1111},
\begin{align}\label{h0}
    P_1H-\nabla_{\psi}H\nabla_{\psi}u_{n}-H\nabla_{\psi}^2u_{n}=R_2.
\end{align}
Observe that by \eqref{dpsi3un} and the bootstrap assumption, it holds
\begin{align}\label{h1}\begin{split}&|\nabla_{\psi}^2u_{n}|+|\nabla_{\psi}^3u_{n}|
\leq C_{\delta_3}e^{-\frac{(z+{\epsilon_0})^2}{x+1}\mu},
\\&|\nabla_{\psi }\nabla_{\xi }H|+|\nabla_{\psi }\nabla_{\eta }H|+|\nabla_{\psi }H|+|\nabla_{\xi }H|+|\nabla_{\eta }H|+|H| \leq C_{\delta_3}\phi_{1,0},
\\&
\text{in} \quad
[0,X]\times[0,Y^*]\times[\delta_3,+\infty).\end{split}
\end{align} Then  by \eqref{h0}, we have
\begin{align}\label{h2}
  |\nabla_{\psi }^2H|\leq& C_{\delta_3}\phi_{1,0}\quad
\text{in} \quad
[0,X]\times[0,Y^*]\times[\delta_3,+\infty).
\end{align}
 By \eqref{26-03-17-1}, we have
 \begin{align}\label{h21}
  |\nabla_{\psi }^2\nabla_{\eta}\tilde{q}|\leq& C_{\delta_3}\phi_{1,0}\quad
\text{in} \quad
[0,X]\times[0,Y^*]\times[\delta_3,+\infty).
\end{align}
Combining \eqref{h1}-\eqref{h21} and \eqref{dzKsmallinfty}, we have
 \begin{align}\label{h3}
  |\nabla_{\psi }R_2|\leq& C_{\delta_3}\phi_{1,0}\quad
\text{in} \quad
[0,X]\times[0,Y^*]\times[\delta_3,+\infty).
\end{align}
Therefore,  by \eqref{h0},
\begin{align}\label{h4}
    P_2\nabla_{\psi}H=R_3,
\end{align}with
\begin{align}\label{h5}
  |R_3|\leq& C_{\delta_3}\phi_{1,0}\quad
\text{in} \quad
[0,X]\times[0,Y^*]\times[\delta_3,+\infty),
\end{align}
where we have used
\begin{align*}
  \nabla_{\psi}\nabla_{\eta}=\nabla_{\eta} \nabla_{\psi}- \frac{ \nabla_{\eta}\tilde{q}}{1+\tilde{q}}\nabla_{\psi},
\end{align*}
 which is derived by
$(1+\tilde{q})\nabla_{\tilde{\psi}}=\nabla_{\psi}$ in \eqref{kh1}.
Moreover, by \eqref{qn0},
$P_2 u_n=( \nabla_{\psi}u_n)^2.$ Then
\begin{align}\label{565}
    P_2(u_n\nabla_{\psi}H)=R_4,
\end{align}with
\begin{align}\label{h10}\begin{split}
  R_4:=&u_nR_3+( \nabla_{\psi}u_n)^2\nabla_{\psi}H-2u_n\nabla_{\psi}u_n\nabla_{\psi}^2H,\\
  |R_4|\leq& C_{\delta_3}\phi_{1,0}\quad
\text{in} \quad
[0,X]\times[0,Y^*]\times[\delta_3,+\infty),\end{split}
\end{align}where we have used \eqref{h1}-\eqref{h2} and \eqref{h5}.
Set $$g=\pm u_n\nabla_{\psi}H-\varepsilon^{\frac{9}{2}}\phi_{1,0}.$$
Then by \eqref{zeta-2025-05-23}, we have $g<0$ in $z\in[0,\delta_3].$ Take $A\gg C_{\delta_{3}}.$ Then
$P_2 g<0$ in $(0,X]\times(0,Y^*]\times[\delta_3,+\infty) $. Then we have \eqref{1116-1} by the boundary condition and the maximum principle.
\end{proof}

Similarly, we have the following lemma.
\begin{lemma}\label{65} It holds
 \begin{align*}
    |\partial_z\nabla_{\eta,\xi}u_n-\partial_z\nabla_{\eta,\xi}\bar{u} |\leq \varepsilon^4  \phi_{1,\alpha} \quad \text{in} \quad \Omega\cap\{0\leq y\leq Y^*\}\cap\{z\geq\delta_\varepsilon\},
 \end{align*}
where $\delta_\varepsilon$ is defined in Remark \ref{yw3}.
 \end{lemma}
 \begin{proof}Set $$\hat{H}=(\nabla_{\eta,\xi}u_n-\nabla_{\eta,\xi}\bar{u})\varphi_{\delta_\varepsilon},$$
where $ \varphi_{\delta_\varepsilon}(z)$ is a smooth cutoff function such that
$0\leq \varphi_{\delta_\varepsilon}\leq 1$, $\varphi_{\delta_\varepsilon}=1$ for $z\geq \delta_\varepsilon$ and $\varphi_{\delta_\varepsilon}=0$ for $0\leq z\leq  \frac{\delta_\varepsilon}{2},$ $|\partial^j_z\varphi_{\delta_\varepsilon}|\leq C_{\delta_\varepsilon},$ $j=1,2,3$.  Here,  $C_{\delta_\varepsilon}$ is a positive constant depending on $\delta_\varepsilon.$
 Then  by \eqref{565}, \begin{align}
    P_2(u_n\nabla_{\psi}\hat{H})=&R_5,\end{align}
    where
\begin{align*}
    R_5=&R_4\varphi_{\delta_\varepsilon}+P_2(\varphi_{\delta_\varepsilon})
    u_n\nabla_{\psi}H-2u_n\nabla_{\psi}\varphi_{\delta_\varepsilon}\nabla_{\psi}(u_n\nabla_{\psi}H)
    \\& + P_2\big(\frac{u_n}{u_{n-1}}\partial_z(\nabla_{\tau_1}\bar{u}-\nabla_{\eta,\xi}\bar{u})\varphi_{\delta_\varepsilon}\big)
     \\&+P_2(u_n(\nabla_{\eta,\xi}u_n-\nabla_{\eta,\xi}\bar{u})\nabla_\psi\varphi_{\delta_\varepsilon}).\end{align*}
Then
\begin{align*}
    |R_5|\leq& C_{\delta_\varepsilon}\phi_{1,0}\quad
\text{in} \quad
[0,X]\times[0,Y^*]\times[\delta_\varepsilon,+\infty).
\end{align*}
 Set $$g=\pm u_n\nabla_{\psi}\hat{H}-\varepsilon^{\frac{9}{2}}\phi_{1,0}.$$
 By taking $A\gg C_{\delta_{\varepsilon}},$ we have
$P_2 g<0$ in $(0,X]\times(0,Y^*]\times[\delta_\varepsilon,+\infty) $. Then by the boundary data and  the  maximum principle, we complete the proof of the theorem.
 \end{proof}

\subsection{Estimate on $\partial_z^2 (u_n-\bar{u})$}
Finally,  we note that the second to last inequality in \eqref{assumpn-2} is the direct consequence of the other inequalities in \eqref{assumpn-2}. In fact, we have the following theorem.
\begin{theorem}
It holds that
\begin{align*}
|\partial_{z}^2(u_n-\bar{u})|\leq \frac{\varepsilon e^{-\frac{A}{x+1}}}{1+\frac{\alpha}{6}}\min\{1,(z+ \frac{\bar{u}(x,y,0)}{\partial_z\bar{u}(x,y,0)})^\alpha\}
\quad \text{in} \quad[0,X]\times[0,Y^*]\times[0,+\infty).
\end{align*}
\end{theorem} \begin{proof}By \eqref{appeqeuclidean}, we have
\begin{align}\label{68}\begin{split}
\partial_{z}u_n^2= &[\int_0^z2u_{n-1}(\nabla_{\xi} u_n +(1+\tilde{q})\nabla_{\eta} u_n )dz']u_{n-1}+(\frac{1}{u_{n-1}}\partial_{z}u_n^2)|_{z=0}u_{n-1} \quad \text{in} \quad \Omega.
    \end{split}
\end{align} Then
\begin{align}\begin{split}
\partial_{z}u_n= &[\int_0^zu_{n-1}(\nabla_{\xi} u_n +(1+\tilde{q})\nabla_{\eta} u_n )dz']\frac{u_{n-1}}{u_n}+(\frac{u_n}{u_{n-1}}\partial_{z}u_n)|_{z=0}\frac{u_{n-1}}{u_n} \quad \text{in} \quad \Omega.\end{split}
\end{align}
This implies
\begin{align}\label{xzx1}\begin{split}
\partial_{z}^2u_n= &\partial_{z}\{[\int_0^zu_{n-1}(\nabla_{\xi} u_n +(1+\tilde{q})\nabla_{\eta} u_n )dz']\frac{u_{n-1}}{u_n}\}\\&+(\frac{u_n}{u_{n-1}}\partial_{z}u_n)|_{z=0}\partial_{z}(\frac{u_{n-1}}{u_n} ),\\
\partial_{z}^2\bar{u}= &2\bar{u}\nabla_{\tau_1} \bar{u}.
    \end{split}
\end{align}
Therefore,
\begin{align}\label{1116-3-exp}\begin{split}
\partial_{z}^2(u_n-\bar{u})= &\partial_{z}\{[\int_0^zu_{n-1}(\nabla_{\xi} u_n +(1+\tilde{q})\nabla_{\eta} u_n )dz']\frac{u_{n-1}}{u_n}\}-2\bar{u}\nabla_{\tau_1} \bar{u}
\\&+(\frac{u_n}{u_{n-1}}\partial_{z}u_n)|_{z=0}\partial_{z}(\frac{u_{n-1}}{u_n} ).
    \end{split}
\end{align}
We claim that the first term for $\partial_{z}^2(u_n-\bar{u})$ in \eqref{1116-3-exp} satisfies
\begin{align}\label{1029}
    |\partial_{z}\{[\int_0^zu_{n-1}(\nabla_{\xi} u_n +(1+\tilde{q})\nabla_{\eta} u_n )dz']\frac{u_{n-1}}{u_n}\}-2\bar{u}\nabla_{\tau_1} \bar{u}|\leq \varepsilon^2 C\phi_{1,1}.
\end{align}

We prove the claim \eqref{1029} as follows.
By straightforward calculation, we have
\begin{align*}
    &\partial_{z}[\int_0^zu_{n-1}\nabla_{\xi} u_ndz'\frac{u_{n-1}}{u_n}]-\bar{u}\nabla_{\tau_1} \bar{u}
   \\ =&u_{n-1}\nabla_{\xi} u_n\frac{u_{n-1}}{u_n}-\bar{u}\nabla_{\tau_1} \bar{u}+\int_0^zu_{n-1}\nabla_{\xi} u_ndz'
    (\frac{\partial_z(u_{n-1}-u_n)}{u_n}-\frac{\partial_zu_n(u_{n-1}-u_n)}{u_n^2})
 \\ =&u_{n-1}\nabla_{\xi} u_n\frac{u_{n-1}-u_n}{u_n}+u_{n-1}\nabla_{\xi} u_n-\bar{u}\nabla_{\tau_1} \bar{u}+\int_0^zu_{n-1}\nabla_{\xi} u_ndz'
    (\frac{\partial_z(u_{n-1}-u_n)}{u_n}-\frac{\partial_zu_n(u_{n-1}-u_n)}{u_n^2})
    \\ =&u_{n-1}\nabla_{\xi} u_n\frac{u_{n-1}-u_n}{u_n}+(u_{n-1}-\bar{u})\nabla_{\xi} u_n+\bar{u}(\nabla_{\xi} u_n-\nabla_{\tau_1} \bar{u})
    \\&+\int_0^zu_{n-1}\nabla_{\xi} u_ndz'
    (\frac{\partial_z(u_{n-1}-u_n)}{u_n}-\frac{\partial_zu_n(u_{n-1}-u_n)}{u_n^2}).
\end{align*} Since $|\int_0^zu_{n-1}\nabla_{\xi} u_ndz'|\leq C\bar{u}^2,$ by \eqref{0821jl}, \eqref{1116-2}
and induction assumption, we have
\begin{align}\label{1117-1}
|\partial_{z}[\int_0^zu_{n-1}\nabla_{\xi} u_ndz'\frac{u_{n-1}}{u_n}]-\bar{u}\nabla_{\tau_1} \bar{u}|\leq \varepsilon^2 C\phi_{1,1}\quad\text{in}\quad \Omega\cap\{0\leq y\leq Y^*\} .
\end{align}
We can estimate the other terms in \eqref{1029} similarly by using  $|\tilde{q}|\leq\varepsilon^2 C\phi_{1,\alpha}$ and the induction assumption. Then the claim \eqref{1029} holds.

Next, we estimate the other terms  for  $\partial_{z}^2(u_n-\bar{u})$ in \eqref{1116-3-exp}.
By Lemma \ref{dzun-1un} in Appendix,  we have, for a small positive constant $\delta_5$ independent of $\varepsilon,$
\begin{align}\label{1117-2}
\big|(\frac{u_n}{u_{n-1}}\partial_{z}u_n)|_{z=0}\partial_{z}(\frac{u_{n-1}}{u_n} )\big|\leq \frac{\varepsilon e^{-\frac{A}{x+1}}(z+\frac{u_n(x,y,0)}{\partial_zu_n(x,y,0)})^{\alpha}}
{1+\frac{\alpha}{5}}\quad \text{in}\quad \Omega\cap\{0\leq y\leq Y^*\}\cap\{z\leq \delta_5\}.
\end{align}
Combining \eqref{1117-1} and \eqref{1117-2}, we have,
\begin{align*}
|\partial_{z}^2(u_n-\bar{u})|\leq \frac{\varepsilon e^{-\frac{A}{x+1}}}{1+\frac{\alpha}{6}}\min\{1,(z+ \frac{\bar{u}(x,y,0)}{\partial_z\bar{u}(x,y,0)})^\alpha\}
\quad \text{in} \quad[0,X]\times[0,Y^*]\times[0,\delta_5].
\end{align*}
On the other hand, by $\partial_z \bar{u}>0$,  \eqref{n-1assumption} and \eqref{assumpn-2}, we have $u_{n-1}, u_n\geq c_0\delta_5$ in $[0,X]\times[0,Y^*]\times[\delta_5,+\infty).$ Then by \eqref{0821jl} and \eqref{1012},
 \begin{align}
    |\partial_z (\frac{u_{n-1}}{u_n})|=|\partial_z (\frac{u_n-u_{n-1}}{u_n})|\leq \varepsilon^3 \phi_{1,\alpha} \quad \text{in} \quad [0,X]\times[0,Y^*]\times[\delta_5,+\infty).
 \end{align}
Then
$|\partial_{z}^2(u_n-\bar{u})|\leq \frac{\varepsilon e^{-\frac{A}{x+1}}}{1+\frac{\alpha}{6}}(z+ \frac{\bar{u}(x,y,0)}{\partial_z\bar{u}(x,y,0)})^\alpha$ in $[0,X]\times[0,Y^*]\times[\delta_5,+\infty)$ because $$|(\frac{u_n}{u_{n-1}}\partial_{z}u_n)|_{z=0}|\leq C\quad \text{in} \quad[0,X]\times[0,Y^*].$$
In summary, $|\partial_{z}^2(u_n-\bar{u})|\leq \frac{\varepsilon e^{-\frac{A}{x+1}}}{1+\frac{\alpha}{6}}\min\{1,(z+ \frac{\bar{u}(x,y,0)}{\partial_z\bar{u}(x,y,0)})^\alpha\}$  in $[0,X]\times[0,Y^*]\times[0,+\infty).$
And this completes the proof of the theorem.
\end{proof}
By $\frac{\bar{u}}{\partial_z \bar{u}}|_{z=0}=\epsilon_0(1+o(\epsilon_0)), $ we have $|\partial_{z}^2u_{n}-\partial_{z}^2\bar{u}|\leq \frac{\varepsilon e^{-\frac{A}{x+1}}}{1+\frac{2\alpha}{13}}\min\{1,(z+\min_{[0,X]\times[0,Y]}\frac{\bar{u}(x,y,0)}{\partial_z\bar{u}(x,y,0)})^\alpha\}
.$
Then we \textit{complete  the proof of  the bootstrap argument}. Furthermore, we can derive the $z$-infinity decay estimate of $\partial_z^2(u_n-\bar{u})$ through the equation.

\begin{lemma} \label{58}It holds that
 \begin{align*}
   | \partial_{z}^2(u_n-\bar{u})|\leq C\varepsilon^{2.3} \phi_{1,\alpha}\quad\text{in}\quad \Omega\cap \{z\geq (\varepsilon d_0)^{\frac{1}{1-\alpha}}\},
 \end{align*}
where $d_0$ is defined in \eqref{assumpn-2}.
\end{lemma}
\begin{proof}
By Theorem \ref{thm1} and  Theorem \ref{dzunphi0}, we can improve the estimates for $u_{n-1}-\bar{u}, v_{n-1}-\bar{u}, \partial_z(u_{n-1}-\bar{u})$ as follows.
\begin{align}
    | u_{n-1}-\bar{u} |+| v_{n-1}-\bar{u} |\leq Cd_0\varepsilon^6 \phi_{1,0},\quad
    |\partial_z u_{n-1}-\partial_z\bar{u}|\leq \frac{d_0\varepsilon^4}{4} \phi_{1,\alpha}\quad
 \text{in} \quad \Omega.
 \end{align} Then for a small positive constant $\alpha,$
 \begin{align}\label{ubdep}\begin{split}
 |\frac{u_{n-1}}{u_{n}} -1|=&\frac{|u_{n-1}-u_{n}|}{u_{n}}\leq C\varepsilon^{4.5} \phi_{1,0} ,\quad\frac{|u_{n-1}-u_{n}|}{u_{n-1}^2}\leq C\varepsilon^3 \phi_{1,0},\\ \frac{ |\partial_z u_n-\partial_zu_{n-1}|}{u_{n-1}}\leq & C\varepsilon^{2.5} \phi_{1,\alpha}, \quad |\tilde{q}|=|\frac{q_{n-1}}{u_{n-1}}|\leq  C\varepsilon^4 \phi_{1,0}\quad\text{in}\quad \Omega\cap \{z\geq (\varepsilon d_0)^{\frac{1}{1-\alpha}}\},\end{split}
\end{align}
because   $d_0$ is independent of $\varepsilon$ and
\begin{align}\label{lbdep}
u_{n-1}, u_{n}, \bar{u}\geq c_0(\varepsilon d_0)^{\frac{1}{1-\alpha}}\quad\text{in}\quad \Omega\cap \{z\geq (\varepsilon d_0)^{\frac{1}{1-\alpha}}\}.
\end{align}
Recall $c_0$  in \eqref{positivelbu}.
Then
\begin{align}\label{1104}\begin{split}
|\partial_z (\frac{ u_n}{u_{n-1}})|=&|\partial_z (\frac{ u_n-u_{n-1}}{u_{n-1}})|\\ \leq &\frac{ |\partial_z u_n-\partial_zu_{n-1}|}{u_{n-1}}+\frac{\partial_zu_{n-1}|u_{n-1}-u_{n}|}{u_{n-1}^2}
\\ \leq& C\varepsilon^{2.5} \phi_{1,0}\quad\text{in}\quad \Omega\cap \{z\geq (\varepsilon d_0)^{\frac{1}{1-\alpha}}\}.\end{split}
\end{align}
Next, by \eqref{appeqeuclidean}, we have
\begin{align}\begin{split}
\partial_{z}^2u_n= &[u_{n-1}(\nabla_{\xi} u_n +(1+\tilde{q})\nabla_{\eta} u_n )-\partial_{z}u_n \partial_{z}(\frac{u_n}{u_{n-1}} )]\frac{u_{n-1}}{u_{n}},\\
\partial_{z}^2\bar{u}= &2\bar{u}\nabla_{\tau_1} \bar{u} .
    \end{split}
\end{align}
Then by Theorem \ref{r2imp}, i.e.
$|\nabla_{\eta,\xi}u_n-\nabla_{\tau_1}\bar{u} |\leq C\varepsilon^5 \phi_{1,0} $ in $ \Omega,$
 \eqref{ubdep} and \eqref{1104}, we have
 \begin{align*}
   | \partial_{z}^2(u_n-\bar{u})|\leq C\varepsilon^{2.5} \phi_{1,0}\quad\text{in}\quad \Omega\cap \{z\geq (\varepsilon d_0)^{\frac{1}{1-\alpha}}\}.
 \end{align*}
 In particular,  for $\varepsilon$ small enough, we have
 \begin{align*}
  C\varepsilon^{0.2} \phi_{1,0} \leq\phi_{1,\alpha},\quad (\varepsilon d_0)^{\frac{1}{1-\alpha}}\leq z\leq \delta,
 \end{align*}
  because  for 
  $\frac{\alpha}{1-\alpha}\leq 0.1,$ it holds
 $$C\varepsilon^{0.2}\leq \frac{1}{2}(\frac{1}{\sqrt{X+1}})^\alpha(\varepsilon d_0)^{\frac{\alpha}{1-\alpha}} \leq (\frac{z+\epsilon_0}{\sqrt{x+1}})^\alpha,\quad (\varepsilon d_0)^{\frac{1}{1-\alpha}}\leq z\leq \delta,$$
  where  $\delta$ is defined in \eqref{phi1}.
 And this completes the proof of the lemma.
\end{proof}

\section{Appendix}
In the Appendix, we will give some details of the calculations, derivation and the change of variables together with the transformation between derivatives, the approximation of the boundary data, growth rates of the background flow, commutator $K$, the barrier functions and some basic estimates used in the proof.
\subsection{Self-similar change of  variables}\label{ssv}
By straightforward  calculation,  for any positive constants $a,b ,k$ satisfying $b^2=ak,$ we have, \begin{itemize} \item { For 2D, if $(u,w)$ is the solution to \eqref{eq:Prandtl}, then
\begin{align*}
    (\tilde{u}(x,z),\tilde{w}(x,z))=(ku(ax,bz),bw(ax,bz))
\end{align*} is also a solution to \eqref{eq:Prandtl} with $\tilde{U}=kU$ replacing $U.$ }\item {  For 3D, if $(u,v,w)$ is the solution to \eqref{eq:3DPrandtl}, then
\begin{align*}
    (\tilde{u}(x,y,z),\tilde{v}(x,y,z),\tilde{w}(x,y,z))=(ku(ax,ay,bz),kv(ax,ay,bz),bw(ax,ay,bz))
\end{align*} is also a solution to \eqref{eq:3DPrandtl} with $\tilde{U}=kU$ replacing $U.$}
\end{itemize}

In particular, letting $(u,v,w) $ and $(u_B,v_B,w_B) $  be the profiles in Theorem \ref{1121-250628} and taking $$a=\frac{1}{\theta}(X+1),\quad b=\sqrt{\frac{3\mu}{2m_0}}$$ for any small positive constants $\theta$ and $\mu$ independent of $\varepsilon$, then $(\tilde{u},\tilde{v},\tilde{w})$ and $(\tilde{u}_B,\tilde{v}_B,\tilde{w}_B)$ defined by  the self-similar change \begin{align}\label{26-02-08-sscg}
    (\tilde{u}(x,y,z),\tilde{v}(x,y,z),\tilde{w}(x,y,z))=(ku(ax,ay,bz),kv(ax,ay,bz),bw(ax,ay,bz))
\end{align}  are solutions to \eqref{eq:3DPrandtl} defined on
\begin{align*}
[0,\frac{X}{X+1}\theta]\times[0,\frac{Y}{X+1}\theta]\times[0,\infty)
\subset[0,\theta]\times[0,Y]\times[0,\infty)
\end{align*}with
  \begin{align*}
	\partial_z \tilde{u}_B&\gtrsim e^{-\frac{3}{2}\mu z^2}\quad\text{for}\quad z>0,
\\
	|\partial_x^{n_1} \partial_y^{n_2} \partial_z^{n_3}\tilde{u}_B| \lesssim e^{-\mu z^2}&\lesssim e^{-\mu \frac{z^2}{1+x}}\quad\text{as}\quad z\rightarrow+\infty.
\end{align*}
\subsection{Approximation of  boundary data}\label{appbddata}

Set
\begin{align}\label{1118-2}
\bar{u}(x,y,z)=u_B(x,y,z+{\epsilon_0}),\quad
u_{bd}^{\epsilon_0}(x,y,z)=u_{bd}(x,y,z+{\epsilon_0}),
\end{align}
and define
\begin{align*}
   h:=u_{bd}^{\epsilon_0} -\bar{u}\quad\text{on}\quad (\{x=0\}\cup\{y=0\})\cap([0,X]\times[0,Y]\times[0,\epsilon_0]).
\end{align*}
By extension theorem such as  Whitney extension theorem,
we extend  $h$
to a smooth function $\tilde{h}$ defined on  $[0,X]\times[0,Y]\times[0,\epsilon_0]$ preserving the growth rates with respect to $\epsilon_0$. Take
\begin{align}\label{appbdry}\begin{split}
    u_n(x,y,0):=\bar{u}(x,y,0)+\tilde{h}(x,y,0)&\quad(x,y) \in[0,X]\times[0,Y],\\
    u_n(x,0,z):=u_{bd}^{\epsilon_0}(x,0,z)&\quad (x,z) \in[0,X]\times[0,+\infty),\\ u_n(0,y,z):=u_{bd}^{\epsilon_0}(0,y,z)&\quad (y,z) \in[0,Y]\times[0,+\infty).\end{split}
\end{align}
For the boundary value on $z=0,$ by
 $
 |u_{bd}-u_B|\leq \varepsilon^8 \Phi_2
 $ in \eqref{bddata} and \eqref{appbdry}, it holds
\begin{align}\label{tria0106}
 |u_{n}-\bar{u}|(x,y,0)\leq \varepsilon^8 e^{-\frac{A}{x+1}}C\epsilon_0^2,
\end{align} which  leads to the non-degenerate boundary value on $z=0,$ i.e.
\begin{align}\label{26-06-10-ncg}
  u_n(x,y,0)\sim  \epsilon_0 \quad\text{on}\quad z=0,
\end{align}since $\bar{u}(x,y,0)=u_B(x,y,\epsilon_0)\sim \epsilon_0$ by noting $\partial_z u_B|_{z=0}>0$ and $u_B|_{z=0}=0.$
By \eqref{tria0106} and \eqref{26-06-10-ncg}, we have
  \begin{align}\label{flw0106}
 |u_{n}^2-\bar{u}^2|(x,y,0)\leq \varepsilon^8 e^{-\frac{A}{x+1}} C\epsilon_0^3.
\end{align}
Moreover, by \eqref{bddata},
 \begin{align}\label{dzcompat}
 |\partial_zu_{bd}-\partial_zu_B|\leq \varepsilon^7 \Phi_2.
 \end{align}
 Then by \eqref{dzcompat},
\begin{align*}
  |u_{bd}-u_B|=&|\int_0^z\partial_z u_{bd}-\partial_z u_B dz'|\\
  \leq &\varepsilon^7 e^{-\frac{A}{x+1}} C\min\{z^3,1\}.
\end{align*}Hence, it holds
\begin{align}\label{z0ep3}
  |u_{n}-\bar{u}|(x,y,0)  \leq \varepsilon^7  e^{-\frac{A}{x+1}}C\epsilon_0^3.
\end{align}

On the other hand, noting that the boundary data such as $\partial_xu_n|_{x=0}$ and $\partial_yu_n|_{y=0}$ are obtained by compatibility as shown in \eqref{code}, we  can obtain the growth rates of derivatives of $u_n-\bar{u}$
 on the boundary by \eqref{bddata}.

Furthermore, we can derive the growth rates of the vector field derivatives
 of  boundary data based on subsection \ref{1027} as follows.
\begin{align*}\begin{split}
| \nabla _{\eta,\xi} (u_{n}-\bar{u})|\leq
\varepsilon^6 C\bar{u}\phi_{1,1}\quad \text{on}\quad &(\{x=0\}\cup\{y=0\}\cup\{z=0\})\cap\overline{\Omega},\\
| \partial_z\nabla _{\eta,\xi} (u_{n}-\bar{u})|\leq
\varepsilon^6 C\phi_{1,1}\quad \text{on}\quad &(\{x=0\}\cup\{y=0\}\cup\{z=0\})\cap\overline{\Omega},\\
| \nabla _{\eta,\xi}  \nabla _{\eta,\xi}(u_{n}-\bar{u})|\leq
\varepsilon^6 C\phi_{1,1}\quad \text{on}\quad &(\{x=0\}\cup\{y=0\}\cup\{z=0\})\cap\overline{\Omega},\\
| \nabla _{\eta,\xi} \bar{u}-\nabla_{\tau_1}\bar{u}|\leq
\varepsilon^6 C\bar{u}\phi_{1,1}\quad \text{on}\quad &(\{x=0\}\cup\{y=0\}\cup\{z=0\})\cap\overline{\Omega},\\
| \partial_z(\nabla _{\eta,\xi} \bar{u}-\nabla_{\tau_1}\bar{u})|\leq
\varepsilon^6 C\phi_{1,1}\quad \text{on}\quad &(\{x=0\}\cup\{y=0\}\cup\{z=0\})\cap\overline{\Omega},\\
| \nabla _{\eta,\xi} \nabla _{\eta,\xi}\bar{u}-\nabla_{\tau_1}^2\bar{u}|\leq
\varepsilon^6 C\phi_{1,1}\quad \text{on}\quad &(\{x=0\}\cup\{y=0\}\cup\{z=0\})\cap\overline{\Omega},
\end{split}\end{align*}
In fact, by \eqref{26-03-08-def-der}, we have
\begin{align*}
  \partial_z\nabla _{\eta} (u_{n}-\bar{u})=  & \partial_z\partial_y(u_{n}-\bar{u})-\partial_z( \frac{\int_0^z \partial_{y}v_{n-1}dz'}{v_{n-1}})\partial_z(u_{n}-\bar{u})
  \\&- \frac{\int_0^z \partial_{y}v_{n-1}dz'}{v_{n-1}}\partial_z^2(u_{n}-\bar{u}).
\end{align*}
Then by the compatible boundary conditions, it holds \begin{align*}
| \partial_z\nabla _{\eta} (u_{n}-\bar{u})|\leq
\varepsilon^6 C\phi_{1,1}\quad \text{on}\quad (\{x=0\}\cup\{y=0\}\cup\{z=0\})\cap\overline{\Omega},
\end{align*} which implies
\begin{align*}
| \nabla _{\eta} (u_{n}-\bar{u})|\leq
\varepsilon^6 C\bar{u}\phi_{1,1}\quad \text{on}\quad (\{x=0\}\cup\{y=0\}\cup\{z=0\})\cap\overline{\Omega}.
\end{align*}

By \eqref{26-03-08-def-der}, we have
\begin{align*}\begin{split}
\nabla _{\eta} \nabla _{\xi}(u_{n}-\bar{u})=&\partial_{yx}^2(u_{n}-\bar{u})- \frac{\int_0^z \partial_{xy}^2u_{n-1}dz'}{u_{n-1}}\partial_z (u_{n}-\bar{u})
\\&+ \frac{\partial_yu_{n-1}\int_0^z \partial_{x}u_{n-1}dz'}{u_{n-1}^2}\partial_z (u_{n}-\bar{u})
\\&- \frac{\int_0^z \partial_{x}u_{n-1}dz'}{u_{n-1}}\partial_{zy}^2 (u_{n}-\bar{u})
- \frac{\int_0^z \partial_{y}v_{n-1}dz'}{v_{n-1}}\partial_z\nabla _{\xi}(u_{n}-\bar{u}).
\end{split}\end{align*}
Then by the compatible boundary conditions, it holds
\begin{align*}
| \nabla _{\eta}\nabla _{\xi} (u_{n}-\bar{u})|\leq
\varepsilon^6 C\phi_{1,1}\quad \text{on}\quad (\{x=0\}\cup\{y=0\}\cup\{z=0\})\cap\overline{\Omega}.
\end{align*}
Next, since \begin{align*}
     \partial_z\nabla_{\xi}\bar{u}=&
       \partial_z\partial_x\bar{u}+\partial_{z}( -\frac{\int_0^z \partial_xu_{n-1}dz'}{u_{n-1}}\partial_z \bar{u}) ,
\end{align*} and $$|\partial_x(u_{n-1}-\bar{u})|\leq \varepsilon^6 C\bar{u}\phi_{1,1}$$ which is derived from
$|\partial_z\partial_x(u_{bd}-u_B)|\leq \varepsilon^6 C\Phi_{1}$ in \eqref{bddata},  by the compatible boundary conditions, it holds
\begin{align*}
       |\partial_z(\nabla_{\tau_1}\bar{u}-\nabla_{\xi}\bar{u})|\leq
\varepsilon^6 C\phi_{1,1}\quad \text{on}\quad (\{x=0\}\cup\{y=0\}\cup\{z=0\})\cap\overline{\Omega}.
       \end{align*}
which implies\begin{align*}
       |\nabla_{\tau_1}\bar{u}-\nabla_{\xi}\bar{u}|\leq
\varepsilon^6 C\bar{u}\phi_{1,1}\quad \text{on}\quad (\{x=0\}\cup\{y=0\}\cup\{z=0\})\cap\overline{\Omega}.
       \end{align*}
Based on this, by
\begin{align*}\begin{split}
     \nabla_{\eta}\nabla_{\xi}\bar{u}=&
       \partial_{y}\partial_x\bar{u}+\partial_{y}( -\frac{\int_0^z \partial_xu_{n-1}dz'}{u_{n-1}}\partial_z \bar{u})
       \\&-\frac{\int_0^z \partial_yv_{n-1}dz'}{v_{n-1}}\partial_z\partial_x\bar{u}-\frac{\int_0^z \partial_yv_{n-1}dz'}{v_{n-1}}\partial_z( -\frac{\int_0^z \partial_xu_{n-1}dz'}{u_{n-1}}\partial_z \bar{u}),\end{split}
    \end{align*}and the compatible boundary conditions, it holds
       \begin{align*}
       |\nabla_{\tau_1}^2\bar{u}-\nabla_{\eta}\nabla_{\xi}\bar{u}|=|\nabla_{\tau_2}\nabla_{\tau_1}\bar{u}-\nabla_{\eta}\nabla_{\xi}\bar{u}|
\leq
\varepsilon^6 C\phi_{1,1}\quad \text{on}\quad (\{x=0\}\cup\{y=0\}\cup\{z=0\})\cap\overline{\Omega}.
       \end{align*}
Similarly, we can derive the other inequalities.
\subsection{Growth rates of $\bar{u}$ and its derivatives}\label{1027}
Here are the growth rates of the smooth function
 $\bar{u}$ and its derivatives which are  used in the proofs.
By the properties of the background profile $u_B$,  we have for any $k,m \in \Z_+,$ $l\in\N,$
\begin{align*}
|\partial_x^k\bar{u}|\leq C_k\bar{u}e^{-\frac{(z+{\epsilon_0})^2}{x+1}\mu}, \quad |\partial_x^l\partial_z^m\bar{u}|\leq C_{l,m}e^{-\frac{(z+{\epsilon_0})^2}{x+1}\mu}.
\end{align*}
This  implies
\begin{align}\label{nnbaru1}\begin{split}
&|\nabla_{\tau_1}^k \bar{u}|\leq C\bar{u}e^{-\frac{(z+{\epsilon_0})^2}{x+1}\mu},\quad   |\nabla_{n}^2 \bar{u}|\leq \frac{C}{\bar{u}^3}e^{-\frac{(z+{\epsilon_0})^2}{x+1}\mu},\quad  |\nabla_{\tau_1}\nabla_{n}^2 \bar{u}|\leq \frac{C}{\bar{u}^3}e^{-\frac{(z+{\epsilon_0})^2}{x+1}\mu},
\\&
|\partial_z\nabla_{\tau_1}\nabla_{n}^2 \bar{u}|\leq \frac{C}{\bar{u}^4}e^{-\frac{(z+{\epsilon_0})^2}{x+1}\mu},\end{split}
\end{align}
where we have used $\nabla_{n}(gh)=\frac{1}{\bar{u}}\partial_z(gh)=h\nabla_{n}g+g\nabla_{n}h$,
and for some positive constant   $\bar{\delta},$
$
  0<2c_0\leq \partial_z \bar{u}\leq 2C_0$ for $ z\in[0,\bar{\delta}].
$ Moreover,  since \begin{align}\label{taunu}
 0= &\nabla_{\tau_1}    \bar{u}+   \nabla_{\tau_1} \bar{u}-\nabla_{n}( \bar{u}\nabla_{n}  \bar{u}),
 \end{align}
 and
 \begin{align*}
 \nabla_{n} \bar{u}=&\frac{\partial_z\bar{u}}{\bar{u}}, \quad
 \nabla_{n}( \bar{u}\nabla_{n}  \bar{u})=\frac{1}{\bar{u}}\partial_z^2\bar{u},
 \quad\partial_z^3\bar{u}|_{z=0}=0,
\end{align*} we have
\begin{align}\label{nnbaru}\begin{split}
  &|\partial_z^3\bar{u}|\leq C\bar{u}e^{-\frac{(z+{\epsilon_0})^2}{x+1}\mu},\quad  |\partial_z^2\bar{u}|\leq C\bar{u}^2e^{-\frac{(z+{\epsilon_0})^2}{x+1}\mu},\quad |\nabla_{n}^2 \bar{u}^2|\leq C\bar{u}e^{-\frac{(z+{\epsilon_0})^2}{x+1}\mu},\\& |\nabla_{n}^2 \bar{u}|\leq \frac{C}{\bar{u}^3}e^{-\frac{(z+{\epsilon_0})^2}{x+1}\mu},\quad |\partial_z\nabla_{n}^2 \bar{u}^2|\leq Ce^{-\frac{(z+{\epsilon_0})^2}{x+1}\mu},\end{split}
\end{align} and $4\bar{u}\nabla_{\tau_1}    \bar{u}=\partial_z ( \nabla_{n} \bar{u}^2)$ so that
\begin{align}\label{znbaru}
  |\partial_z  \nabla_{n} \bar{u}^2|\leq C\bar{u}^2e^{-\frac{(z+{\epsilon_0})^2}{x+1}\mu}.
\end{align}

\subsection{Test equation of commutator $K$}\label{AppK}
Suppose that  $f$ is a smooth function. By  \eqref{kh1}, \eqref{kh2} and \eqref{K}, we have the following calculation.

\subsubsection{Tangential vector field derivatives}
\begin{align*}
  \nabla_{\xi}  ( \nabla_{\xi} f+(1+\tilde{q})\nabla_{\eta}f)
  =&  \nabla_{\xi}^2f+  \nabla_{\xi}\tilde{q}\,\nabla_{\eta}f
  +(1+\tilde{q})\nabla_{\xi}\nabla_{\eta}f,\\
   \nabla_{\eta}  ( \nabla_{\xi} f+(1+\tilde{q})\nabla_{\eta}f)
  =&  \nabla_{\eta}\nabla_{\xi}f+  \nabla_{\eta}\tilde{q}\,\nabla_{\eta}f
  +(1+\tilde{q})\nabla_{\eta}^2f.
\end{align*}
By \eqref{K}, we have
\begin{align}\label{communicator}
\nabla_{\xi}\nabla_{\eta}  =\nabla_{\eta}  \nabla_{\xi}+K\partial_z.
\end{align}
 Then
\begin{align*}
    &\nabla_{\eta}  \nabla_{\xi}^2f  +(1+\tilde{q})\nabla_{\eta}\nabla_{\xi}\nabla_{\eta}f
    \\=&(\nabla_{\xi}\nabla_{\eta} -K\partial_z)\nabla_{\xi}f
     +(1+\tilde{q})\nabla_{\eta}(\nabla_{\eta}\nabla_{\xi}+K\partial_z)f\\
     =&\nabla_{\xi}\nabla_{\eta}\nabla_{\xi}f-K\partial_z\nabla_{\xi}f
     +(1+\tilde{q})\nabla_{\eta}(\nabla_{\eta}\nabla_{\xi}f)
     \\&+(1+\tilde{q})\nabla_{\eta}(K\partial_zf),
\end{align*}and
\begin{align*}
    &\nabla_{\xi} \nabla_{\eta}\nabla_{\xi}f  +(1+\tilde{q})\nabla_{\xi}\nabla_{\eta}^2f\\
    =&\nabla_{\xi} (\nabla_{\xi}\nabla_{\eta}-K\partial_z)f  +(1+\tilde{q})(\nabla_{\eta}\nabla_{\xi}+K\partial_z)\nabla_{\eta}f
    \\=&\nabla_{\xi} (\nabla_{\xi}\nabla_{\eta}f)-\nabla_{\xi}(K\partial_zf)
    +(1+\tilde{q})\nabla_{\eta}(\nabla_{\xi}\nabla_{\eta}f)
    +(1+\tilde{q})K\partial_z\nabla_{\eta}f.
\end{align*}
This implies
\begin{align*}
  &\nabla_{\eta}  \nabla_{\xi}  ( \nabla_{\xi} f+(1+\tilde{q})\nabla_{\eta}f)
\\ =& \nabla_{\eta}  \nabla_{\xi}^2f  +(1+\tilde{q})\nabla_{\eta}\nabla_{\xi}\nabla_{\eta}f
 +  \nabla_{\eta} \nabla_{\xi}\tilde{q}\,\nabla_{\eta}f+ \nabla_{\xi}\tilde{q}\,\nabla_{\eta}^2f +\nabla_{\eta}\tilde{q}\nabla_{\xi}\nabla_{\eta}f
 \\ =&\nabla_{\xi}\nabla_{\eta}\nabla_{\xi}f-K\partial_z\nabla_{\xi}f
     +(1+\tilde{q})\nabla_{\eta}(\nabla_{\eta}\nabla_{\xi}f)
     +(1+\tilde{q})\nabla_{\eta}(K\partial_zf)
     \\&+  \nabla_{\eta} \nabla_{\xi}\tilde{q}\,\nabla_{\eta}f+ \nabla_{\xi}\tilde{q}\,\nabla_{\eta}^2f +\nabla_{\eta}\tilde{q}\nabla_{\xi}\nabla_{\eta}f,
     \end{align*}and
     \begin{align*}
     &  \nabla_{\xi} \nabla_{\eta} ( \nabla_{\xi} f+(1+\tilde{q})\nabla_{\eta}f)
\\ =& \nabla_{\xi} \nabla_{\eta}\nabla_{\xi}f  +(1+\tilde{q})\nabla_{\xi}\nabla_{\eta}^2f+  \nabla_{\xi} \nabla_{\eta}\tilde{q}\,\nabla_{\eta}f+\nabla_{\eta}\tilde{q}\, \nabla_{\xi} \nabla_{\eta}f+\nabla_{\xi} \tilde{q}\, \nabla_{\eta}^2f
\\=&\nabla_{\xi} (\nabla_{\xi}\nabla_{\eta}f)-\nabla_{\xi}(K\partial_zf)
    +(1+\tilde{q})\nabla_{\eta}(\nabla_{\xi}\nabla_{\eta}f)
    +(1+\tilde{q})K\partial_z\nabla_{\eta}f
\\&+  \nabla_{\xi} \nabla_{\eta}\tilde{q}\,\nabla_{\eta}f+\nabla_{\eta}\tilde{q}\, \nabla_{\xi} \nabla_{\eta}f+\nabla_{\xi} \tilde{q}\, \nabla_{\eta}^2f.
\end{align*}
Hence,
 \begin{align*}
     &\nabla_{\xi} \nabla_{\eta} ( \nabla_{\xi} f+(1+\tilde{q})\nabla_{\eta}f)- \nabla_{\eta}  \nabla_{\xi}  ( \nabla_{\xi} f+(1+\tilde{q})\nabla_{\eta}f)
\\=&\nabla_{\xi} (\nabla_{\xi}\nabla_{\eta}f-\nabla_{\eta}\nabla_{\xi}f)
    +(1+\tilde{q})\nabla_{\eta}(\nabla_{\xi}\nabla_{\eta}f-\nabla_{\eta}\nabla_{\xi}f)
   -(1+\tilde{q})\nabla_{\eta}(K\partial_zf)-\nabla_{\xi}(K\partial_zf)
  \\& +(1+\tilde{q})K\partial_z\nabla_{\eta}f+K\partial_z\nabla_{\xi}f
+  (\nabla_{\xi} \nabla_{\eta}\tilde{q}-\nabla_{\eta} \nabla_{\xi}\tilde{q})\,\nabla_{\eta}f.
\end{align*}By \eqref{communicator}, the first line on the right hand side is equal to $0$ so that
 \begin{align*}
     &\nabla_{\xi} \nabla_{\eta} ( \nabla_{\xi} f+(1+\tilde{q})\nabla_{\eta}f)- \nabla_{\eta}  \nabla_{\xi}  ( \nabla_{\xi} f+(1+\tilde{q})\nabla_{\eta}f)
\\=& (1+\tilde{q})K\partial_z\nabla_{\eta}f+K\partial_z\nabla_{\xi}f
+  K\partial_z\tilde{q}\,\,\nabla_{\eta}f.
\end{align*}
By direct calculation, we have
\begin{align}\label{tanpsipsi}\begin{split}
\nabla_{\xi}(u_{n-1}\nabla_{\psi}^2f)=&
\nabla_{\xi}u_{n-1}\,\nabla_{\psi}^2f+u_{n-1}\nabla_{\psi}^2\nabla_{\xi}f,\\
\nabla_{\eta}(u_{n-1}\nabla_{\psi}^2f)=&
\nabla_{\eta}u_{n-1}\,\nabla_{\psi}^2f
+u_{n-1}\nabla_{\eta}\big((1+\tilde{q})\nabla_{\tilde{\psi}}\nabla_{\psi}f\big)\\
=&
\nabla_{\eta}u_{n-1}\,\nabla_{\psi}^2f+u_{n-1}\frac{\nabla_{\eta}\tilde{q}}{1+\tilde{q}}\,\,\nabla_{\psi}^2f
+u_{n-1}\nabla_{\psi}\nabla_{\eta}\big((1+\tilde{q})\nabla_{\tilde{\psi}}f\big)
\\=&\nabla_{\eta}u_{n-1}\,\nabla_{\psi}^2f+u_{n-1}\nabla_{\psi}^2\nabla_{\eta}f
+u_{n-1}\frac{\nabla_{\eta}\tilde{q}}{1+\tilde{q}}\,\,\nabla_{\psi}^2f
\\&+u_{n-1}\nabla_{\psi}\big(\frac{\nabla_{\eta}\tilde{q}}{1+\tilde{q}}\nabla_{\psi}f\big)
,\end{split}
\end{align} and
\begin{align*}
  & \nabla_{\xi}u_{n-1}
\nabla_{\psi}\nabla_{\eta}\nabla_{\psi}f
+u_{n-1}\nabla_{\psi}\nabla_{\eta}\nabla_{\psi}\nabla_{\xi}f\\
=&\nabla_{\xi}u_{n-1}
\nabla_{\psi}\nabla_{\eta}\big((1+\tilde{q})\nabla_{\tilde{\psi}}f\big)
+u_{n-1}\nabla_{\psi}\nabla_{\eta}\big((1+\tilde{q})\nabla_{\tilde{\psi}}\nabla_{\xi}f
\big)
\\=&\nabla_{\xi}u_{n-1}
\nabla_{\psi}^2\nabla_{\eta}f
+u_{n-1}\nabla_{\psi}^2\nabla_{\eta}\nabla_{\xi}f
\\&+\nabla_{\xi}u_{n-1}
\nabla_{\psi}(\frac{\nabla_{\eta}\tilde{q}}{1+\tilde{q}}\nabla_{\psi}f)
+u_{n-1}\nabla_{\psi}(\frac{\nabla_{\eta}\tilde{q}}{1+\tilde{q}}\nabla_{\psi}\nabla_{\xi}f
\big).\end{align*}
Then we have
\begin{align*}
&\nabla_{\eta}\nabla_{\xi}(u_{n-1}\nabla_{\psi}^2f)
\\=&
\nabla_{\eta}\nabla_{\xi}u_{n-1}\,\nabla_{\psi}^2f+\nabla_{\xi}u_{n-1}
\nabla_{\eta}\big((1+\tilde{q})\nabla_{\tilde{\psi}}\nabla_{\psi}f\big)
\\&+\nabla_{\eta}u_{n-1}\nabla_{\psi}^2\nabla_{\xi}f
+u_{n-1}\nabla_{\eta}\big((1+\tilde{q})\nabla_{\tilde{\psi}}\nabla_{\psi}\nabla_{\xi}f\big)
\\=&
\nabla_{\eta}\nabla_{\xi}u_{n-1}\,\nabla_{\psi}^2f
+\nabla_{\eta}u_{n-1}\nabla_{\psi}^2\nabla_{\xi}f
+ \nabla_{\xi}u_{n-1}
\nabla_{\psi}\nabla_{\eta}\nabla_{\psi}f
+u_{n-1}\nabla_{\psi}\nabla_{\eta}\nabla_{\psi}\nabla_{\xi}f
\\&+\nabla_{\xi}u_{n-1}
\frac{\nabla_{\eta}\tilde{q}}{1+\tilde{q}}\nabla_{\psi}^2f
+u_{n-1}\frac{\nabla_{\eta}\tilde{q}}{1+\tilde{q}}\,\,\nabla_{\psi}^2\nabla_{\xi}f
\\=&\nabla_{\xi}u_{n-1}
\nabla_{\psi}^2\nabla_{\eta}f
+u_{n-1}\nabla_{\psi}^2\nabla_{\eta}\nabla_{\xi}f
+\nabla_{\eta}\nabla_{\xi}u_{n-1}\,\nabla_{\psi}^2f
+\nabla_{\eta}u_{n-1}\nabla_{\psi}^2\nabla_{\xi}f
\\&+\nabla_{\xi}u_{n-1}
\nabla_{\psi}(\frac{\nabla_{\eta}\tilde{q}}{1+\tilde{q}}\nabla_{\psi}f)
+u_{n-1}\nabla_{\psi}(\frac{\nabla_{\eta}\tilde{q}}{1+\tilde{q}}\nabla_{\psi}\nabla_{\xi}f
\big)\\&+\nabla_{\xi}u_{n-1}
\frac{\nabla_{\eta}\tilde{q}}{1+\tilde{q}}\nabla_{\psi}^2f
+u_{n-1}\frac{\nabla_{\eta}\tilde{q}}{1+\tilde{q}}\,\,\nabla_{\psi}^2\nabla_{\xi}f,
\end{align*}
and\begin{align*}
 \nabla_{\xi} \nabla_{\eta}(u_{n-1}\nabla_{\psi}^2f)=& \nabla_{\xi} \nabla_{\eta}u_{n-1}\,\nabla_{\psi}^2f+ \nabla_{\eta}u_{n-1}\,\nabla_{\psi}^2\nabla_{\xi} f
 +\nabla_{\xi} u_{n-1}\,\,\nabla_{\psi}^2\nabla_{\eta}f+ u_{n-1}\,\,\nabla_{\psi}^2\nabla_{\xi} \nabla_{\eta}f
\\&+\nabla_{\xi}u_{n-1}\frac{\nabla_{\eta}\tilde{q}}{1+\tilde{q}}\,\,\nabla_{\psi}^2f
+u_{n-1}\nabla_{\xi}(\frac{\nabla_{\eta}\tilde{q}}{1+\tilde{q}})\,\,\nabla_{\psi}^2f
+u_{n-1}\frac{\nabla_{\eta}\tilde{q}}{1+\tilde{q}}\,\,\nabla_{\psi}^2\nabla_{\xi}f
\\&+\nabla_{\xi}u_{n-1}\nabla_{\psi}\big(\frac{\nabla_{\eta}\tilde{q}}{1+\tilde{q}}\nabla_{\psi}f\big)
+u_{n-1}\nabla_{\psi}\big(\frac{\nabla_{\eta}\tilde{q}}{1+\tilde{q}}\nabla_{\psi}\nabla_{\xi}f\big)
+u_{n-1}\nabla_{\psi}\big(\nabla_{\xi}(\frac{\nabla_{\eta}\tilde{q}}{1+\tilde{q}})\nabla_{\psi}f\big)
.
\end{align*}Hence,
\begin{align*}
    \nabla_{\xi} \nabla_{\eta}(u_{n-1}\nabla_{\psi}^2f)-\nabla_{\eta}\nabla_{\xi} (u_{n-1}\nabla_{\psi}^2f)=& (\nabla_{\xi} \nabla_{\eta}u_{n-1}-\nabla_{\eta}\nabla_{\xi} u_{n-1})\nabla_{\psi}^2f
    \\&+ u_{n-1}\nabla_{\psi}^2(\nabla_{\xi} \nabla_{\eta}f- \nabla_{\eta}\nabla_{\xi}f)\\
    &+u_{n-1}\nabla_{\xi}(\frac{\nabla_{\eta}\tilde{q}}{1+\tilde{q}})\,\,\nabla_{\psi}^2f
+u_{n-1}\nabla_{\psi}\big(\nabla_{\xi}(\frac{\nabla_{\eta}\tilde{q}}{1+\tilde{q}})\nabla_{\psi}f\big).
\end{align*}
In summary, \begin{align}\begin{split}
&K\partial_z(\nabla_{\xi} f+(1+\tilde{q})\nabla_{\eta}f-u_{n-1}\nabla_{\psi}^2f)\\ =  &( \nabla_{\xi} \nabla_{\eta}-\nabla_{\eta}\nabla_{\xi} ) (\nabla_{\xi} f+(1+\tilde{q})\nabla_{\eta}f-u_{n-1}\nabla_{\psi}^2f)\\
   =&- u_{n-1}\nabla_{\psi}^2(\nabla_{\xi} \nabla_{\eta}f- \nabla_{\eta}\nabla_{\xi}f)\\&- (\nabla_{\xi} \nabla_{\eta}u_{n-1}-\nabla_{\eta}\nabla_{\xi} u_{n-1})\nabla_{\psi}^2f
    \\&+ (1+\tilde{q})K\partial_z\nabla_{\eta}f+K\partial_z\nabla_{\xi}f
+  K\partial_z\tilde{q}\,\,\nabla_{\eta}f\\
    &-2u_{n-1}\nabla_{\xi}(\frac{\nabla_{\eta}\tilde{q}}{1+\tilde{q}})\,\,\nabla_{\psi}^2f
-u_{n-1}\big(\nabla_{\psi}\nabla_{\xi}(\frac{\nabla_{\eta}\tilde{q}}{1+\tilde{q}})\big)\nabla_{\psi}f
,\end{split}
\end{align} and\begin{align}\label{symcore}\begin{split}
K\partial_z(-u_{n-1}\nabla_{\psi}^2f)=  &- u_{n-1}\nabla_{\psi}^2(K\partial_zf)\\&- (\nabla_{\xi} \nabla_{\eta}u_{n-1}-\nabla_{\eta}\nabla_{\xi} u_{n-1})\nabla_{\psi}^2f
    \\
    &-2u_{n-1}\nabla_{\xi}(\frac{\nabla_{\eta}\tilde{q}}{1+\tilde{q}})\,\,\nabla_{\psi}^2f
-u_{n-1}\big(\nabla_{\psi}\nabla_{\xi}(\frac{\nabla_{\eta}\tilde{q}}{1+\tilde{q}})\big)\nabla_{\psi}f
.\end{split}
\end{align}
\subsection{Transformation between derivatives}\label{66}
In this subsection, we will prove the following theorem.
\begin{theorem}\label{lm6}If \eqref{assumpn-2} holds in $\Omega$, then
\begin{align}\label{10271}\begin{split}
&|\partial_{x,y}u_{n}-\partial_{x}\bar{u}|,\,|\partial_{x,y}v_{n}-\partial_{x}\bar{u}|\leq \varepsilon^2 \phi_{1,1},
\\&
|\partial_z\partial_{x,y}u_{n}-\partial_z\partial_{x}\bar{u}|,\,|\partial_z\partial_{x,y}v_{n}-\partial_z\partial_{x}\bar{u}|\leq \varepsilon^2 \phi_{1,\alpha},\quad
\\&|\partial_{x,y}\partial_{x,y}u_{n}-\partial_{x}^2\bar{u}|,\,
|\partial_{x,y}\partial_{x,y}v_{n}-\partial_{x}^2\bar{u}|\leq\varepsilon^2 \phi_{1,\frac{\alpha}{2}}\quad \text{in }\quad\Omega
.\end{split}
\end{align}
\end{theorem} For brevity, we will prove the following three lemmas and  Theorem \ref{lm6} follows.
\begin{lemma}\label{11dj-1}If \eqref{assumpn-2} holds in $\Omega$, then\begin{align}\label{618}\begin{split}
|\partial_z\partial_{x}u_{n}-\partial_z\partial_{x}\bar{u}|\leq \varepsilon^2 \phi_{1,\alpha}\quad \text{in }\quad\Omega
.\end{split}
\end{align}
\end{lemma}
\begin{lemma}\label{11dj-2}If \eqref{assumpn-2} holds in $\Omega$, then\begin{align}\begin{split}
|\partial_{x}u_{n}-\partial_{x}\bar{u}|\leq \varepsilon^2 \phi_{1,1}\quad \text{in }\quad\Omega
.\end{split}
\end{align}
\end{lemma}
\begin{lemma}\label{11dj-3}If \eqref{assumpn-2} holds in $\Omega$, then\begin{align}\begin{split}
|\partial_{x}^2u_{n}-\partial_{x}^2\bar{u}|\leq\varepsilon^2 \phi_{1,\frac{\alpha}{2}}\quad \text{in }\quad\Omega.\end{split}
\end{align}
\end{lemma}
\begin{proof}[Proof of Lemma \ref{11dj-1}]
Based on the transformation formula in \eqref{rlt}, we have
\begin{align}\label{1105-1}
    \partial_{z}\partial_{x}=&  \partial_{z}\nabla _{\xi} +\partial_{z}(\frac{\int_0^z \partial_{x}u_{n-1}dz'}{u_{n-1}})\partial_z
    +\frac{\int_0^z \partial_{x}u_{n-1}dz'}{u_{n-1}}\partial_z^2.
\end{align}By induction assumption, we have
\begin{align}\label{dz1}
  |\partial_z(\frac{\int_0^z \partial_{x}u_{n-1}dz'}{u_{n-1}})|\leq C, \quad |\frac{\int_0^z \partial_{x}u_{n-1}dz'}{u_{n-1}}|\leq C\min\{z,1\}.
\end{align} Then by  \eqref{assumpn-2} and \eqref{1119-1}, we have
\begin{align*}
    |\frac{\int_0^z \partial_{x}u_{n-1}dz'}{u_{n-1}}\partial_z^2(u_n-\bar{u})|\leq&
    Cz\varepsilon e^{-\frac{A}{x+1}}(z+\epsilon_0)^{\alpha}\\ \leq& C\varepsilon^4 e^{-\frac{A}{x+1}}(z+\epsilon_0)^{\alpha} \quad \text{in} \quad \Omega\cap\{0\leq z\leq \varepsilon^3\}.
\end{align*}

We now consider the estimate in two regions.

\textit{Step 1} Estimate in $\Omega\cap\{0\leq z\leq \delta_\varepsilon\}.$  Here,  $\delta_{\varepsilon}$ is defined in Remark \ref{yw3} with $\delta_{\varepsilon}\leq \varepsilon^3.$ Firstly, by the above estimates,
\begin{align*}
    |\partial_{z}(\frac{\int_0^z \partial_{x}u_{n-1}dz'}{u_{n-1}})\partial_z(u_n-\bar{u})
    +\frac{\int_0^z \partial_{x}u_{n-1}dz'}{u_{n-1}}\partial_z^2(u_n-\bar{u})|\leq \varepsilon^{3}\phi_{1,\alpha} \quad \text{in} \quad \Omega\cap\{0\leq z\leq \delta_\varepsilon\}.
\end{align*}
Next, by
 \eqref{zh2}, we have
 \begin{align}
    |\partial_z\nabla_{\eta,\xi}u_n-\partial_z\nabla_{\eta,\xi}\bar{u} |\leq  \frac{ \varepsilon^2}{2+\frac{\alpha}{3}} \phi_{1,\alpha} \quad \text{in} \quad \Omega\cap\{0\leq z\leq \delta_\varepsilon\}.
 \end{align}
 By \eqref{1105-1}, we have
\begin{align}\label{b4}
    |\partial_z\partial_x(u_n-\bar{u}) |\leq  \frac{ \varepsilon^2}{2} \phi_{1,\alpha} \quad \text{in} \quad \Omega\cap\{0\leq z\leq \delta_\varepsilon\}.
 \end{align}

 \textit{Step 2} Estimate in $\Omega\cap\{z\geq \delta_\varepsilon\}.$
In this region, by \eqref{dz1},  we have
\begin{align}
  |\frac{\int_0^z \partial_{x}u_{n-1}dz'}{u_{n-1}}|\leq Cd_0\varepsilon\quad
  \text{in }\quad \Omega\cap \{0\leq z\leq\varepsilon d_0\}.
\end{align}Then by \eqref{assumpn-2}, we have for small $d_0$ independent of $\varepsilon,$
\begin{align}\label{b1-17}
    |\frac{\int_0^z \partial_{x}u_{n-1}dz'}{u_{n-1}}\partial_z^2(u_n-\bar{u}) |
    \leq Cd_0\varepsilon^2\phi_{1,\alpha}\leq \frac{\varepsilon^2}{3}\phi_{1,\alpha}\quad
  \text{in }\quad \Omega\cap \{0\leq z\leq\varepsilon d_0\}.
\end{align}
By Lemma \ref{58}, we have
 \begin{align}\label{26-04-02-01}
   | \partial_{z}^2(u_n-\bar{u})|\leq C\varepsilon^{2.3} \phi_{1,\alpha}\quad\text{in}\quad \Omega\cap \{z\geq \varepsilon d_0\},
 \end{align}
because $d_0\varepsilon<1$ so that  $\varepsilon d_0>(\varepsilon d_0)^{\frac{1}{1-\alpha}}.$ Then by \eqref{dz1}, \eqref{b1-17} and  \eqref{26-04-02-01}, we obtain
 \begin{align}\label{b1}
    |\frac{\int_0^z \partial_{x}u_{n-1}dz'}{u_{n-1}}\partial_z^2(u_n-\bar{u}) |
    \leq Cd_0\varepsilon^2\phi_{1,\alpha}\leq \frac{\varepsilon^2}{3}\phi_{1,\alpha}\quad
  \text{in }\quad \Omega\cap \{z\geq \delta_\varepsilon\}.
\end{align}
Moreover, by Lemma \ref{65} and  \eqref{assumpn-2}, we have
 \begin{align}\label{b2}
    |\partial_z\nabla_{\eta,\xi}u_n-\partial_z\nabla_{\eta,\xi}\bar{u} |\leq \varepsilon^4  \phi_{1,\alpha} \quad \text{in} \quad \Omega\cap\{z\geq\delta_\varepsilon\},
 \end{align} and
 \begin{align}\label{b3}
    |\partial_{z}(\frac{\int_0^z \partial_{x}u_{n-1}dz'}{u_{n-1}})\partial_z(u_n-\bar{u})|\leq \varepsilon^3\phi_{1,\alpha}\quad \text{in} \quad \Omega.
 \end{align}
Finally,  by \eqref{1105-1}, we have
\begin{align}\label{b5}
    |\partial_z\partial_x(u_n-\bar{u}) |\leq  \frac{ \varepsilon^2}{2} \phi_{1,\alpha} \quad \text{in} \quad \Omega\cap\{ z\geq \delta_\varepsilon\}.
 \end{align} Combining \eqref{b4} and \eqref{b5} completes the proof of the lemma.
\end{proof}
\begin{remark}\label{26-04-02-222}By \eqref{b4} and \eqref{b5} in the proof of Lemma \ref{11dj-1}, we actually have
\begin{align*}
    |\partial_z\partial_x(u_n-\bar{u}) |\leq  \frac{ \varepsilon^2}{2} \phi_{1,\alpha} \quad \text{in} \quad \Omega.
 \end{align*}\end{remark}
Before the proof of Lemma \ref{11dj-2}, we prove the following estimate.
\begin{lemma}\label{lm65} It holds \begin{align*}
    |\nabla_{\xi}u_n-\nabla_{\xi}\bar{u} |\leq \varepsilon^4 \phi_{1,0}\quad \text{in} \quad \Omega\cap\{z+\epsilon_0\geq \delta\}.
 \end{align*}
\end{lemma}
\begin{proof}By \eqref{426}, we have
 \begin{align}\begin{split}
 ( P_2+ P_3)(\nabla_{\xi}u_n-\nabla_{\xi}\bar{u} )  =R
  ,
  \end{split}
\end{align} with
\begin{align*}
   R=&R_2+( P_2+ P_3)( \nabla_{\tau_1}\bar{u}-\nabla_{\xi}\bar{u} ),\\
   |R|\leq &C_{\delta}\phi_{1,0} \quad \text{in} \quad \Omega\cap \{z\geq \delta\},
\end{align*} by \eqref{xitau} where $C_\delta$ is a positive constant depending on $\delta$. Since  $u_n, u_{n-1}, \bar{u}\geq c_0 \delta$ in $\Omega\cap \{z\geq \delta\}$ and the properties of the background profile $\bar{u}$ are known,
by taking $A$ large  depending on $C_{\delta}$ and applying the maximum principle as  in the proof of Theorem \ref{r2imp}, we have the estimate stated in the lemma.
\end{proof}

We are now ready to prove Lemmas \ref{11dj-2}-\ref{11dj-3} as follows.

\begin{proof}[Proof of Lemma \ref{11dj-2}]
By Remark \ref{26-04-02-222}, 
we have
\begin{align*}
|\partial_x(u_n-\bar{u}) |\leq \frac{\varepsilon^2}{2(1+\alpha)}\sqrt{x+1} \phi_{1,\alpha+1}+\epsilon_0^2<\varepsilon^2\phi_{1,1},\quad z+\epsilon_0\leq \delta,
\end{align*}
because  $x\leq X\leq \theta$ with $\theta$ small.

For $\varepsilon$ small, we have
\begin{align}\label{yw5}
\varepsilon\phi_{1,0}\leq \phi_{1,1},\quad z+\epsilon_0\geq \delta.
\end{align} Then by \begin{align}\label{1105-2}
    \partial_{x}=&  \nabla _{\xi} +\frac{\int_0^z \partial_{x}u_{n-1}dz'}{u_{n-1}}\partial_z
\end{align}from  \eqref{rlt}, Lemma \ref{lm65}, \eqref{yw5}, \eqref{dz1} and \eqref{assumpn-2}, we have
\begin{align*}
|\partial_x(u_n-\bar{u}) |\leq \varepsilon^2 \phi_{1,1},\quad z+\epsilon_0\geq \delta.
\end{align*}

\end{proof}
\begin{proof}[Proof of Lemma \ref{11dj-3}]
By  \eqref{rlt}, we have
 \begin{align}\label{1105-3} \begin{split}
   \partial_{x}^2=& (\nabla _{\xi} +\frac{\int_0^z \partial_{x}u_{n-1}dz'}{u_{n-1}}\partial_z)\partial_{x}\\
   =& \nabla _{\xi} \partial_{x}+\frac{\int_0^z \partial_{x}u_{n-1}dz'}{u_{n-1}}\partial_z\partial_{x}\\
   =& \nabla _{\xi}  (\nabla _{\xi} +\frac{\int_0^z \partial_{x}u_{n-1}dz'}{u_{n-1}}\partial_z)+\frac{\int_0^z \partial_{x}u_{n-1}dz'}{u_{n-1}}\partial_z\partial_{x}\\
    =& \nabla _{\xi}^2 +\nabla _{\xi}(\frac{\int_0^z \partial_{x}u_{n-1}dz'}{u_{n-1}})\partial_z+\frac{\int_0^z \partial_{x}u_{n-1}dz'}{u_{n-1}}\nabla _{\xi}\partial_z+\frac{\int_0^z \partial_{x}u_{n-1}dz'}{u_{n-1}}\partial_z\partial_{x}\\
     =& \nabla _{\xi}^2 +\nabla _{\xi}(\frac{\int_0^z \partial_{x}u_{n-1}dz'}{u_{n-1}})\partial_z+\frac{\int_0^z \partial_{x}u_{n-1}dz'}{u_{n-1}}\partial_z\partial_{x}
    \\& +\frac{\int_0^z \partial_{x}u_{n-1}dz'}{u_{n-1}}[(\frac{\nabla _{\xi}u_{n-1}}{u_{n-1}})\partial_z+\partial_z\nabla _{\xi}],
 \end{split}\end{align}
 where we have used
 \begin{align*}
    \nabla _{\xi}\partial_z=\nabla _{\xi}(u_{n-1}\nabla _{\psi})
    =\nabla _{\xi}u_{n-1}\nabla _{\psi}+u_{n-1}\nabla _{\psi}\nabla _{\xi}
  =(\frac{\nabla _{\xi}u_{n-1}}{u_{n-1}})\partial_z+\partial_z\nabla _{\xi}.
 \end{align*}
 By the induction assumption, it holds
 $$|\nabla _{\xi}(\frac{\int_0^z \partial_{x}u_{n-1}dz'}{u_{n-1}})|\leq C\min\{(z+\epsilon_0)^{\frac{\alpha}{2}},1\},\quad |\frac{\nabla _{\xi}u_{n-1}}{u_{n-1}}|\leq C,$$
 and
 \begin{align*}
   |\frac{\int_0^z \partial_{x}u_{n-1}dz'\nabla _{\xi}u_{n-1}}{u_{n-1}^2}|\leq C\min\{z,1\}.
 \end{align*}
Then by Corollary \ref{d2tan} with  $d_0$ in
\eqref{15} being small
and \eqref{assumpn-2}, we have
\begin{align}\label{117}\begin{split}
    &|\nabla _{\xi}^2(u_n-\bar{u}) +\nabla _{\xi}(\frac{\int_0^z \partial_{x}u_{n-1}dz'}{u_{n-1}})\partial_z(u_n-\bar{u})
     +\frac{\int_0^z \partial_{x}u_{n-1}dz'}{u_{n-1}}(\frac{\nabla _{\xi}u_{n-1}}{u_{n-1}})\partial_z(u_n-\bar{u})|
 \\    \leq&\varepsilon^2d_0\phi_{1,\frac{\alpha}{2}}\quad \text{in }\quad \Omega. \end{split} \end{align} By \eqref{dz1}, \eqref{zh2}, Lemma \ref{65} and \eqref{618}, we have
\begin{align*}
    |\frac{\int_0^z \partial_{x}u_{n-1}dz'}{u_{n-1}}\big(\partial_z\partial_{x}(u_n-\bar{u})
    +\partial_z\nabla _{\xi}(u_n-\bar{u})\big)|\leq d_0C\varepsilon^2 \phi_{1,\alpha}\quad \text{in }\quad \Omega\cap\{z\leq  d_0\}.
\end{align*}
 On the other hand, by  \eqref{1105-1}, \eqref{dz1}, \eqref{assumpn-2} and Lemma \ref{58},\begin{align}\begin{split}
    &|\partial_{z}\partial_{x}(u_n-\bar{u})-  \partial_{z}\nabla _{\xi} (u_n-\bar{u})|
  \\  =& |\partial_{z}\big(\frac{\int_0^z \partial_{x}u_{n-1}dz'}{u_{n-1}}\partial_{z}(u_n-\bar{u})\big)|\\
    =& |\frac{\int_0^z \partial_{x}u_{n-1}dz'}{u_{n-1}}\partial_{z}^2(u_n-\bar{u})|+
    |\partial_{z}(\frac{\int_0^z \partial_{x}u_{n-1}dz'}{u_{n-1}})\partial_{z}(u_n-\bar{u})|\\
    \leq&\varepsilon^{2.1}\phi_{1,\alpha} \quad\Omega\cap\{z\geq d_0\}.
\end{split}\end{align}
Then by Lemma \ref{65},
\begin{align*}
    |\frac{\int_0^z \partial_{x}u_{n-1}dz'}{u_{n-1}}\big(\partial_z\partial_{x}(u_n-\bar{u})
    +\partial_z\nabla _{\xi}(u_n-\bar{u})\big)|\leq C\varepsilon^{2.1} \phi_{1,\alpha}\quad \text{in }\quad \Omega\cap\{z\geq  d_0\}.
\end{align*}
In summary, we have
\begin{align*}
    |\frac{\int_0^z \partial_{x}u_{n-1}dz'}{u_{n-1}}\big(\partial_z\partial_{x}(u_n-\bar{u})
    +\partial_z\nabla _{\xi}(u_n-\bar{u})\big)|\leq d_0C\varepsilon^2 \phi_{1,\alpha}\quad \text{in }\quad \Omega.
\end{align*}
Combining this with \eqref{117} and \eqref{1105-3},
we have \begin{align}\begin{split}
|\partial_{x}^2u_{n}-\partial_{x}^2\bar{u}|\leq C d_0\varepsilon^2 \phi_{1,\frac{\alpha}{2}}<\varepsilon^2 \phi_{1,\frac{\alpha}{2}}\quad \text{in }\quad\Omega.\end{split}
\end{align}
 \end{proof}

\subsection{Barrier functions for $P_1$ and $P_2$}\label{bf}
Recall $P_1$ and $P_2$ defined in \eqref{p10106}-\eqref{p20106}.
In this subsection, we will consider functions  in $ [0,X]\times[0,Y^*]\times[0,+\infty)$. In particular, by the definition of $Y^*$ and the induction assumption,
\begin{align}\label{bfassp}\begin{split}
   &|\partial_zu_{n-1}-\partial_zu_{n}|\leq \varepsilon e^{-\frac{A}{x+1}}, \quad
 |-u_{n-1}  +u_{n}|\leq \varepsilon \phi_{1,1},\\&
 \frac{1}{2}c_0e^{-\frac{3}{2}\mu(z+{\epsilon_0})^2}\leq \partial_zu_{n},\partial_zu_{n-1}\leq 2C_0e^{-\frac{(z+{\epsilon_0})^2}{x+1}\mu},\\
 &\text{ and \eqref{n-1assumption} holds in}\quad [0,X]\times[0,Y^*]\times[0,+\infty).\end{split}
\end{align}

\subsubsection{ Barrier functions near $z=0$ }

 \begin{lemma}\label{Lq} For $\alpha\in(0,1),$ there exist some small positive constants $\delta$ and $\lambda$ which are independent of $\varepsilon$ such that
\begin{align}\begin{split}
   P_1(1+z^{\alpha}) =P_1z^{\alpha}\geq& \lambda \frac{z^{\alpha-2} }{u_{n-1}}  \quad\text{in}\quad(0,X]\times(0,Y^*]\times(0,\delta).
  \end{split} \end{align}
\end{lemma}\begin{proof}
 By \eqref{bfassp}, for small $\varepsilon,$ $\alpha\in(0,1)$ and some small positive constants $\lambda$ and $\delta$, we have
\begin{align}\begin{split}
   P_1 z^{\alpha}=& -\tilde{b}\alpha z^{\alpha-1}
   +\alpha(-\frac{\partial_zu_{n}}{u_{n-1}^2}
   +\frac{u_{n}\partial_zu_{n-1}}{u_{n-1}^3})z^{\alpha-1}
   -\frac{u_n}{u_{n-1}^2}\alpha(\alpha-1)z^{\alpha-2}\\
   \geq& -C\alpha z^\alpha +\frac{1}{2}\frac{u_n}{u_{n-1}^2}\alpha(-\alpha+1)z^{\alpha-2}
   \\ \geq &\frac{\lambda}{u_{n-1}}\alpha(-\alpha+1)z^{\alpha-2}, \quad z\leq \delta,
\end{split}\end{align}
where we have used
\begin{align*}
    |\frac{u_n}{u_{n-1}^2}-\frac{1}{u_{n-1}}|= |\frac{u_n-u_{n-1}}{u_{n-1}^2}|\leq \frac{\varepsilon  e^{-\frac{A}{x+1}} C}{u_{n-1}},
\end{align*}
and
\begin{align*}
  |  -\frac{\partial_zu_{n}}{u_{n-1}^2}
   +\frac{u_{n}\partial_zu_{n-1}}{u_{n-1}^3}|=&|
   \frac{-\partial_zu_{n}u_{n-1}+u_{n}\partial_zu_{n-1}}{u_{n-1}^3}|\\
   =&| \frac{(\partial_zu_{n-1}-\partial_zu_{n})u_{n-1}
 +(-u_{n-1}  +u_{n})\partial_zu_{n-1}}{u_{n-1}^3}|\\
 \leq & \frac{\varepsilon  e^{-\frac{A}{x+1}}C}{u_{n-1}^2},
\end{align*}by \eqref{bfassp}.
 \end{proof}

\begin{lemma}\label{L2}  For some small positive constants $\delta$ and $\lambda$  which are independent of $\varepsilon$, we have
\begin{align}\begin{split}
    P_2\psi_{n-1}^{\alpha}\geq& \lambda\psi_{n-1}^{\alpha-\frac{3}{2}}   \quad\text{in}\quad(0,X]\times(0,Y^*]\times(0,\delta),\\
     P_2(\psi_{n-1}-\psi_{n-1}^{1+\alpha})\geq & \lambda\psi_{n-1}^{\alpha-\frac{1}{2}} \quad\text{in}\quad(0,X]\times(0,Y^*]\times(0,\delta),
  \end{split} \end{align}  where $\alpha\in(0,1)$ and $\psi_{n-1}$ is defined in \eqref{psi0125}.
\end{lemma}\begin{proof}
Since   $u_{n-1}\psi_{n-1}^{\alpha-2}\geq \lambda_0 \psi_{n-1}^{\alpha-\frac{3}{2}}$, $z\leq \delta_0$ for some small positive constants $\lambda_0 $ and $\delta_0$, by  \eqref{bfassp},  \eqref{4.2} and \eqref{psin-1x1}, for $\alpha\in(0,1),$ it holds
\begin{align}\begin{split}
  P_2\psi_{n-1}^{\alpha}=  &\nabla_\xi \psi_{n-1}^\alpha +(1+\tilde{q})\nabla_{\eta}  \psi_{n-1}^\alpha-u_{n}\nabla_{\psi}^2\psi_{n-1}^\alpha
 \\ =&\alpha(\tilde{q}\int_0^z \partial_yu_{n-1}dz'- \int_0^z \partial_yq_{n-1}dz')\psi_{n-1}^{\alpha-1}+\alpha(-\alpha+1)u_{n}\psi_{n-1}^{\alpha-2}
\\ \geq &\lambda\psi_{n-1}^{\alpha-\frac{3}{2}} ,\\
 P_2(\psi_{n-1}-\psi_{n-1}^{1+\alpha})=&-\nabla_\xi \psi_{n-1}^{1+\alpha} +(1+\tilde{q})\nabla_{\eta}  (\psi_{n-1}-\psi_{n-1}^{1+\alpha})+u_{n}\nabla_{\psi}^2\psi_{n-1}^{1+\alpha}
 \\ =&(\tilde{q}\int_0^z \partial_yu_{n-1}dz'- \int_0^z \partial_yq_{n-1}dz')(1-(1+\alpha)\psi_{n-1}^{\alpha})\\&+\alpha(\alpha+1)u_{n}\psi_{n-1}^{\alpha-1}
\\ \geq &\lambda\psi_{n-1}^{\alpha-\frac{1}{2}} ,
\end{split}\end{align}  for $z\leq \delta$  such that $\psi_{n-1}\ll1.$ \end{proof}

\subsubsection{ Barrier functions for large $z$}
\begin{lemma}\label{infgrowthz}For a small positive constant $\mu$ and a big positive constant $N$  which are independent of $\varepsilon,$
\begin{align}\begin{split}
    P_1e^{-\frac{z^2}{x+1}\mu}\geq\frac{1}{2}e^{-\frac{z^2}{x+1}\mu}
    \frac{z^2}{(x+1)^2}\mu>0 \quad\text{in}\quad(0,X]\times(0,Y^*]\times(N,\infty),\\
    P_2e^{-\frac{z^2}{x+1}\mu}\geq\frac{1}{4}e^{-\frac{z^2}{x+1}\mu}
    \frac{z^2}{(x+1)^2}\mu>0 \quad\text{in}\quad(0,X]\times(0,Y^*]\times(N,\infty).\end{split}
   \end{align}
\end{lemma}
\begin{proof} By \eqref{linemethodb} and \eqref{bfassp}, for $\mu $ small enough and
	for $z$ large enough, it holds
\begin{align*}
P_1e^{-\frac{z^2}{x+1}\mu}=&\partial_x e^{-\frac{z^2}{x+1}\mu} -\tilde{b}\partial_{z} e^{-\frac{z^2}{x+1}\mu} -\frac{u_n}{u_{n-1}^2}\partial_{z}^2e^{-\frac{z^2}{x+1}\mu}
+(-\frac{\partial_zu_{n}}{u_{n-1}^2}
   +\frac{u_{n}\partial_zu_{n-1}}{u_{n-1}^3}) (-\frac{2z}{x+1}\mu)e^{-\frac{z^2}{x+1}\mu}
     \\  \geq  &e^{-\frac{z^2}{x+1}\mu}
    [\frac{z^2}{(x+1)^2}\mu -C\frac{2z}{x+1}\mu -\frac{u_n}{u_{n-1}^2}(\frac{4z^2}{(x+1)^2}\mu^2-\frac{2}{x+1}\mu)
  ]\\ \geq& \frac{1}{2}e^{-\frac{z^2}{x+1}\mu}
    \frac{z^2}{(x+1)^2}\mu \geq    0.
\end{align*} 
Next, by \eqref{p10106}-\eqref{linemethodbsign},
$$P_2=P_1+\nabla_{\psi}u_n\nabla_{\psi}=P_1+\frac{\partial_zu_n}{u_{n-1}^2}\partial_z.$$ Then for $z$ large enough, it holds\begin{align*}
P_2e^{-\frac{z^2}{x+1}\mu}=&P_1e^{-\frac{z^2}{x+1}\mu}
+\frac{\partial_zu_n}{u_{n-1}^2}\partial_ze^{-\frac{z^2}{x+1}\mu}\\ \geq& e^{-\frac{z^2}{x+1}\mu}
   (\frac{1}{2} \frac{z^2}{(x+1)^2}\mu-C\frac{z}{x+1}\mu)
    \\ \geq &\frac{1}{4}e^{-\frac{z^2}{x+1}\mu} \frac{z^2}{(x+1)^2}\mu >  0.
\end{align*}
\end{proof}
\begin{lemma}\label{infgrowthpsi} For a small positive constant $\mu$ and a big positive constant $N$ which are independent of $\varepsilon,$
\begin{align}
    P_2e^{-\frac{\psi_{n-1}^2(x,y,z)}{x+1}\mu}\geq0 \quad\text{in}\quad(0,X]\times(0,Y^*]\times(N,\infty),
   \end{align} where  $\psi_{n-1}$ is defined in \eqref{psi0125}.
\end{lemma}
\begin{proof}By  \eqref{bfassp} and \eqref{psin-1}, for $\mu $ independent of $\varepsilon$  and small enough such that  $\|\mu u_n\|_{L^\infty}\leq \frac{1}{8}$,  when $z$ is large, we have
\begin{align*}
    P_2 e^{-\frac{\psi_{n-1}^2(x,y,z)}{x+1}\mu}=&\nabla_\xi e^{-\frac{\psi_{n-1}^2(x,y,z)}{x+1}\mu} +(1+\tilde{q})\nabla_{\eta} e^{-\frac{\psi_{n-1}^2(x,y,z)}{x+1}\mu} -u_{n}\nabla_{\psi}^2e^{-\frac{\psi_{n-1}^2(x,y,z)}{x+1}\mu} \\
    =&e^{-\frac{\psi_{n-1}^2(x,y,z)}{x+1}\mu}
    [
    -(\tilde{q}\int_0^z \partial_yu_{n-1}dz'-\int_0^z \partial_yq_{n-1}dz')\frac{2\psi_{n-1}(x,y,z)}{x+1}\mu
     ]\\
   &+e^{-\frac{\psi_{n-1}^2(x,y,z)}{x+1}\mu}
    [\frac{\psi_{n-1}^2(x,y,z)}{(x+1)^2}\mu -u_{n}(\frac{4\psi_{n-1}^2(x,y,z)}{(x+1)^2}\mu^2-\frac{2}{x+1}\mu)
  ]\\ \geq& 0.
\end{align*}
\end{proof}
\begin{lemma}\label{mon-3-15}For a small positive constant $\mu$ and a big positive constant $N$ which are independent of $\varepsilon,$
\begin{align*}
    P_1e^{-\frac{3}{2}\mu z^2}\leq -e^{-\frac{3}{2}\mu z^2}<0
 \quad\text{in}\quad(0,X]\times(0,Y^*]\times[N,\infty).
   \end{align*}
\end{lemma}
\begin{proof}By \eqref{linemethodb} and \eqref{bfassp}, for $\mu $ small enough and $z$ big enough such that $z\mu\gg1$, it holds
\begin{align*}
P_1e^{-\frac{3}{2}\mu z^2}=& -\tilde{b}\partial_{z} e^{-\frac{3}{2}\mu z^2} -\frac{u_n}{u_{n-1}^2}\partial_{z}^2e^{-\frac{3}{2}\mu z^2}
\\&+(-\frac{\partial_zu_{n}}{u_{n-1}^2}
   +\frac{u_{n}\partial_zu_{n-1}}{u_{n-1}^3}) (-3\mu z)e^{-\frac{3}{2}\mu z^2}
     \\  \leq  &e^{-\frac{3}{2}\mu z^2}
    [ Cz\mu -\frac{u_n}{u_{n-1}^2}(9\mu^2 z^2-3\mu)
  ]\\ \leq& -e^{-\frac{3}{2}\mu z^2}<    0.
\end{align*}
\end{proof}
\subsubsection{Barrier functions on $\Omega$}
For $\alpha, \beta\in(0,1),$ take $\phi_{1,\beta},\phi_{2,\beta},\phi_{2,1}$ to be the  barrier functions with ridges which are defined in \eqref{phi1}, \eqref{phi2beta} and  \eqref{phi4} respectively.

Note
\begin{align*}
    \partial_x e^{-\frac{A}{x+1}}=\frac{A}{(x+1)^2}e^{-\frac{A}{x+1}}.
\end{align*}

By taking $A$ large enough, we have
\begin{align*}
   \partial_x( e^{-\frac{A}{x+1}}(x+1)^{-\frac{\alpha}{2}})= & e^{-\frac{A}{x+1}}(x+1)^{-2-\frac{\alpha}{2}}(A-\frac{\alpha}{2}(x+1))\\
   \geq &e^{-\frac{A}{x+1}}(X+1)^{-2-\frac{\alpha}{2}}(\frac{99}{100}A).
\end{align*} Then similar to the proofs of Lemma \ref{Lq}-\ref{infgrowthpsi}, we can show that for a positive constant $c_2$ independent of $\alpha, \beta\in(0,1)$, two small positive constants $\delta$ and $\delta_0$, a big positive constant $N$  and a large positive constant $A$ depending on $\delta_0$,
\begin{align}\label{pphi}\begin{split}
   & P_1 \phi_{1,\beta}\geq c_2(\frac{\beta(1-\beta)}{u_{n-1} (z+\epsilon_0)^2}+A)\phi_{1,\beta},
     \,\, P_2 \phi_{2,\beta}\geq c_2(\beta(1-\beta)\psi_{n-1}^{-\frac{3}{2 } }+A)\phi_{2,\beta},\\
     & P_2 \phi_{2,1}\geq c_2(\alpha\psi_{n-1}^{\alpha-\frac{3}{2 } }+A)\phi_{2,1},  \quad \text{in} \quad(0,X]\times(0,Y^*]\times[0,\delta_0],\\
    &  P_1 \phi_{1,\beta}\geq c_2A\phi_{1,0},  P_2 \phi_{2,\beta}, P_2 \phi_{2,1}\geq c_2A\phi_{2,0},  \\& \text{in} \quad(0,X]\times(0,Y^*]\times[0,+\infty)\setminus\{\text{ridges of the barrier functions}\},
\end{split}\end{align} where  $\phi_{1,0}$ is defined in \eqref{phi3} and by \eqref{26-02-08-xtheta},
\begin{align}\label{0323-p2phi}
P_2\phi_{1,0}\geq \frac{1}{2}A\phi_{1,0} \quad\text{in}\quad(0,X]\times(0,Y^*]\times(N,\infty) . \end{align}

\subsection{Some basic lemmas}\label{tl}
To measure the difference between vector fields,
we list the following equalities that have been used frequently. By  \eqref{kh1}, \eqref{kh2} and \eqref{K},  we have
\begin{align}\label{psi-n}
    \nabla_{\psi}-\nabla_{n}=(\frac{1}{u_{n-1}}-\frac{1}{\bar{u}})\partial_z
=(\frac{\bar{u}-u_{n-1}}{u_{n-1}\bar{u}})\partial_z,
\end{align}
and
 \begin{align}\label{xitau}\begin{split}
\nabla_{\xi}-\nabla_{\tau_1}=&(\frac{\int_0^z\partial_x\bar{u}dz'}{\bar{u}}-\frac{\int_0^z\partial_x u_{n-1}dz'}{u_{n-1}})\partial_z
\\
=&\frac{(u_{n-1}-\bar{u})\int_0^z\partial_x\bar{u}dz'+\bar{u}\int_0^z\partial_x\bar{u}-\partial_xu_{n-1}dz'}{u_{n-1}\bar{u}}\partial_z
\\:=&g_1\partial_z,\\
\nabla_{\eta}-\nabla_{\tau_2}=&(\frac{\int_0^z\partial_y\bar{u}dz'}{\bar{u}}-\frac{\int_0^z\partial_y v_{n-1}dz'}{v_{n-1}})\partial_z
\\
=&\frac{(v_{n-1}-\bar{u})\int_0^z\partial_y\bar{u}dz'+\bar{u}\int_0^z\partial_y\bar{u}-\partial_yv_{n-1}dz'}{v_{n-1}\bar{u}}\partial_z
\\:=&g_2\partial_z.\end{split}\end{align}
 Under the induction assumption, when $z$ is small, it holds that
 \begin{align}\label{gg}\begin{split}
|g_1|+ |g_2|\leq  |\frac{(z+{\epsilon_0})^3\varepsilon^2 e^{-\frac{A}{x+1}} C}{u_{n-1}\bar{u}}|
 \leq (z+{\epsilon_0})\varepsilon^2 e^{-\frac{A}{x+1}}C.\end{split}\end{align}

In addition, we have
\begin{align}\label{cy}\begin{split}(u_n\nabla_{\psi}^2- \bar{u}\nabla_{n}^2)=&
    \frac{u_n}{u_{n-1}}\partial_z( \frac{1}{u_{n-1}}\partial_z \cdot)-\partial_z(\frac{1}{\bar{u}}\partial_z \cdot)
    \\=&
    (\frac{u_n}{u_{n-1}}-1)\partial_z( \frac{1}{\bar{u}} \partial_z\cdot)+\frac{u_n}{u_{n-1}}\partial_z((- \frac{1}{\bar{u}} + \frac{1}{u_{n-1}})\partial_z\cdot)
    \\=&
    (\frac{u_n}{u_{n-1}}-1)\partial_z \nabla_{n}+\frac{u_n}{u_{n-1}}\partial_z( \frac{\bar{u}-u_{n-1}}{u_{n-1}})\nabla_{n}+\frac{u_n}{u_{n-1}}( \frac{\bar{u}-u_{n-1}}{u_{n-1}})\partial_z\nabla_{n}.
\end{split}\end{align}
And the auxiliary   function $\psi_{n-1}$ has the following  growth rates in $z$ near $z=0$,
 \begin{align}\label{psin-1un-1}\begin{split}
        \psi_{n-1}(x,y,z)=&\int_0^{z} u_{n-1}(x,y,z')dz'
        \leq \int_0^{z}  C{\epsilon_0}+2C_0z' dz'
        \\
        \leq & C{\epsilon_0} z+C_0z^2
         \leq  Cu_{n-1}z.\end{split}
      \end{align}

      A direct consequence of the above calculation is the following  lemma, which is used to estimate the remainder $R_0$ in the equation \eqref{u-baru} for $u_n^2-\bar{u}^2.$
      \begin{lemma}\label{6100205} It holds that
\begin{align}\label{cy-1}\begin{split}|(u_n\nabla_{\psi}^2- \bar{u}\nabla_{n}^2)\bar{u}^2|\leq C\phi_{1,0} \quad \text{in} \quad \Omega\cap\{0\leq y\leq Y^*\},
\end{split}\end{align}
and \begin{align}\label{xz2}
u_{n}|\nabla_\psi (u_{n}^2-\bar{u}^2)|\psi_{n-1}^{-\alpha-1}\leq \varepsilon e^{-\frac{A}{x+1}}C \psi_{n-1}^{-\alpha}u_{n-1}^2\psi_{n-1}^{-1} \quad \text{in} \quad \Omega\cap\{0\leq y\leq Y^*\}.
\end{align}
\end{lemma}
\begin{proof}
	First, we consider the estimates  near $z=0.$ By
$|u_{n}-\bar{u}|\leq \varepsilon  e^{-\frac{A}{x+1}}C{\epsilon_0}^2,\,\,|\partial_z (u_{n}-\bar{u})|\leq \varepsilon  e^{-\frac{A}{x+1}}C{\epsilon_0}$ at $z=0$ and $|\partial_z^2 u_{n}-\partial_z^2 \bar{u}|\leq \varepsilon  e^{-\frac{A}{x+1}}$  in $\Omega\cap\{0\leq y\leq Y^*\}$, we have
\begin{align}\label{26-06-13-1}
    |\partial_z (u_{n}-\bar{u})|\leq \varepsilon  e^{-\frac{A}{x+1}}C(z+{\epsilon_0}), \quad   | u_{n}-\bar{u}|\leq \varepsilon  e^{-\frac{A}{x+1}} C(z+{\epsilon_0})^2\quad \text{in}\quad \Omega\cap\{0\leq y\leq Y^*\}.
\end{align} Then we have, in $\Omega\cap\{0\leq y\leq Y^*\},$
\begin{align}\begin{split}
|\nabla_{n}\bar{u}^2|=|2\partial_z\bar{u}|\leq C, \quad |\partial_z\nabla_{n}\bar{u}^2|=|2\partial_z^2\bar{u}|\leq &C,\\
 |\frac{ \partial_z(\bar{u}-u_{n-1})}{u_{n-1}}- \frac{ \partial_zu_{n-1}(\bar{u}-u_{n-1})}{u_{n-1}^2}|\leq& \varepsilon  e^{-\frac{A}{x+1}}C ,
\end{split}\end{align}
which gives \eqref{cy-1} by \eqref{cy}, the decay rates of derivatives of  $\bar{u}$ for large $z$ and the definition of $Y^*$.

By the  definition of $\nabla_\psi,$ we have
\begin{align}\label{sm2}\begin{split}
\nabla_\psi (u_{n}^2-\bar{u}^2)=\frac{1}{u_{n-1}}\partial_z (u_{n}^2-\bar{u}^2)
=\frac{2}{u_{n-1}}(u_{n}\partial_z u_{n}-\bar{u}\partial_z \bar{u}).
\end{split}\end{align} By \eqref{26-06-13-1}, we have
\begin{align}\label{xinzeng1}
|2(\frac{u_{n}}{u_{n-1}}\partial_z u_{n}-\frac{\bar{u}}{u_{n-1}}\partial_z \bar{u})|\leq \varepsilon  e^{-\frac{A}{x+1}}C{\epsilon_0},\quad \text{at} \quad z=0.
\end{align}
Now we estimate
\begin{align}\label{xinzeng}\begin{split}
   \partial_z\big( \frac{1}{u_{n-1}}(u_{n}\partial_z u_{n}-\bar{u}\partial_z \bar{u})\big)=&-\frac{\partial_zu_{n-1}}{u_{n-1}^2}(u_{n}\partial_z u_{n}-\bar{u}\partial_z \bar{u})\\&+\frac{1}{u_{n-1}}((\partial_z u_{n})^2-(\partial_z \bar{u})^2+u_{n}\partial_z^2 u_{n}-\bar{u}\partial_z^2\bar{u}).\end{split}
\end{align} Again, since
$|u_{n}-\bar{u}|\leq \varepsilon  e^{-\frac{A}{x+1}}C{\epsilon_0}^2, |\partial_z (u_{n}-\bar{u})|\leq \varepsilon  e^{-\frac{A}{x+1}}C{\epsilon_0}$ at $z=0$ and $|\partial_z^2 u_{n}-\partial_z^2 \bar{u}|\leq \varepsilon e^{-\frac{A}{x+1}}$  in $\Omega\cap\{0\leq y\leq Y^*\}$, we have
\begin{align*}
   | u_{n}\partial_z u_{n}-\bar{u}\partial_z \bar{u}|=&
    |u_{n}(\partial_z u_{n}-\partial_z \bar{u})+(u_{n}-\bar{u})\partial_z \bar{u}|
    \leq \varepsilon  e^{-\frac{A}{x+1}}C(z+{\epsilon_0})^2 ,
\end{align*}
and
\begin{align*}
    |(\partial_z u_{n})^2-(\partial_z \bar{u})^2+u_{n}\partial_z^2 u_{n}-\bar{u}\partial_z^2\bar{u}|\leq \varepsilon  e^{-\frac{A}{x+1}}C(z+{\epsilon_0}).
\end{align*} Hence, by \eqref{xinzeng}, we have
$|\partial_z \big(\frac{1}{u_{n-1}}(u_{n}\partial_z u_{n}-\bar{u}\partial_z \bar{u})\big)|\leq \varepsilon  e^{-\frac{A}{x+1}}C.$ Combining this with \eqref{xinzeng1}, we have
\begin{align}
|\frac{u_{n}}{u_{n-1}}\partial_z u_{n}-\frac{\bar{u}}{u_{n-1}}\partial_z \bar{u}|\leq \varepsilon  e^{-\frac{A}{x+1}}C(z+{\epsilon_0}),\quad \text{in}\quad \Omega\cap\{0\leq y\leq Y^*\}.
\end{align}
Substituting this  back to \eqref{sm2}, we obtain \eqref{xz2}.
\end{proof}

The following  lemmas are about the growth rates which have been used to obtain the estimate on $\partial_z^2(u_n-\bar{u}).$

\begin{lemma}\label{tl2}
Assume $f\in C^\infty(\R)$ and $|\partial_z f|\leq2 \varepsilon (z+c)^\alpha$ where $c$ is a non-negative constant. Then
\begin{align*}
    |f-\frac{\int_0^zf dz'}{z+c}|\leq \frac{\varepsilon(z+c)^{1+\alpha}}{1+\frac{\alpha}{2}}+|f(0)|\quad \text{for}\quad z>0.
\end{align*}
\end{lemma}
\begin{proof}
Set
$G=(z+c)f-\int_0^zf dz'.$ Then $G(0)=cf(0)$, $\partial_z G=(z+c)\partial_zf$ and $$|\partial_z G|\leq 2\varepsilon (z+c)^{1+\alpha}, \quad z>0, $$ which implies $$|G|\leq\frac{\varepsilon (z+c)^{2+\alpha}}{1+\frac{\alpha}{2}}+c|f(0)| \quad \text{ for}\quad z>0.$$ Hence, $$|\frac{G}{z+c}|\leq\frac{\varepsilon (z+c)^{1+\alpha}}{1+\frac{\alpha}{2}}+\frac{c}{z+c}|f(0)|\quad \text{ for}\quad z>0.$$
\end{proof}
\begin{remark}\label{1117-6}
In general, if $|\partial_z f|\leq b(z+c)^\alpha$, then \begin{align*}
    |f-\frac{\int_0^zf dz'}{z+c}|\leq \frac{b(z+c)^{1+\alpha}}{2+\alpha}+|f(0)|\quad \text{for}\quad z>0.
\end{align*}
\end{remark}
To apply Lemma \ref{tl2}, we firstly give the following rough bound estimate.
Since $\partial_z \bar{u}>0$, by   the compatible approximate boundary data, we have, for small $\varepsilon\ll \alpha,$
\begin{align*}
&\frac{\varepsilon}{1+\frac{\alpha}{7}}
\min\{1,(z+ \frac{\bar{u}(x,y,0)}{\partial_z\bar{u}(x,y,0)})^\alpha\}
\\ \leq &
\varepsilon
\min\{1,(z+\frac{u_n(x,y,0)}{\partial_zu_n(x,y,0)})^\alpha\}\quad \text{in} \quad[0,X]\times[0,Y^*]\times[0,1].
\end{align*}  Then by  \eqref{n-1assumption} and \eqref{assumpn-2}, it holds,
\begin{align}\label{zzf}\begin{split}
   &|\partial_{z}^2(u_{n-1}-\bar{u})|, |\partial_{z}^2(u_n-\bar{u})|\leq \varepsilon e^{-\frac{A}{x+1}}(z+\frac{u_n(x,y,0)}{\partial_zu_n(x,y,0)})^\alpha
   \\& \text{in} \quad[0,X]\times[0,Y^*]\times[0,1].\end{split}
\end{align}
Now we have the following corollary.
\begin{Corollary}\label{cr}It holds that
\begin{align*}\begin{split}
&\frac{\partial_zu_n(x,y,0)}{z\partial_{z}u_n(x,y,0)+u_n(x,y,0)}[\partial_{z}(u_{n-1}-u_n) -\frac{\partial_zu_n(x,y,0)\big((u_{n-1}-u_n)-(u_{n-1}-u_n)|_{z=0}\big)}{z\partial_zu_n(x,y,0)+u_n(x,y,0)}]
\\ \leq & \frac{\varepsilon e^{-\frac{A}{x+1}}(z+\frac{u_n(x,y,0)}{\partial_zu_n(x,y,0)})^{\alpha}}{1+\frac{\alpha}{3}}
\quad \text{in}\quad \Omega\cap\{0\leq y\leq Y^*\}.\end{split}
\end{align*}
\end{Corollary}
\begin{proof}
For any fixed $(x,y)\in[0,X]\times[0,Y^*]$, take
\begin{align*}
    f=\partial_{z}(u_{n-1}-u_n)(x,y,z), \quad c=\frac{u_n(x,y,0)}{\partial_zu_n(x,y,0)}
\end{align*} in Lemma \ref{tl2}. By \eqref{zzf},  $|\partial_z f|\leq |\partial_z^2(u_n-\bar{u})|+|\partial_z^2(u_{n-1}-\bar{u})|\leq2 \varepsilon  e^{-\frac{A}{x+1}} (z+c)^\alpha$. Hence, by Lemma \ref{tl2}, we have
\begin{align}\begin{split}
&\frac{1}{z+\frac{u_n(x,y,0)}{\partial_zu_n(x,y,0)}}[\partial_{z}(u_{n-1}-u_n) -\frac{(u_{n-1}-u_n)-(u_{n-1}-u_n)|_{z=0}}{z+\frac{u_n(x,y,0)}{\partial_zu_n(x,y,0)}}]\\
\leq & \frac{\varepsilon e^{-\frac{A}{x+1}}(z+c)^{\alpha}}{1+\frac{\alpha}{2}}+\frac{1}{z+c}|f(0)|
\\
\leq & \frac{\varepsilon e^{-\frac{A}{x+1}}(z+c)^{\alpha}}{1+\frac{\alpha}{3}}
\quad \text{in}\quad \Omega\cap\{0\leq y\leq Y^*\},\end{split}
\end{align}where we have used $|f(0)|\leq \varepsilon^7 e^{-\frac{A}{x+1}}C\epsilon_0^2$ by the compatible approximate boundary data.
\end{proof}

With  this corollary, we have the following lemma.

\begin{lemma}\label{dzun-1un} For some positive constant $\delta_5$ independent of $\varepsilon,$
it holds that
\begin{align}
\big|(\frac{u_n}{u_{n-1}}\partial_{z}u_n)|_{z=0}\partial_{z}(\frac{u_{n-1}}{u_n} )\big|\leq \frac{\varepsilon e^{-\frac{A}{x+1}}(z+\frac{u_n(x,y,0)}{\partial_zu_n(x,y,0)})^{\alpha}}
{1+\frac{\alpha}{5}}\quad \text{in}\quad \Omega\cap\{0\leq y\leq Y^*\}\cap\{z\leq \delta_5\}.
\end{align}
\end{lemma}
\begin{proof} Take $$\tilde{c}=\frac{u_n(x,y,0)}{\partial_zu_n(x,y,0)}.$$
First, we investigate $(\frac{u_n}{u_{n-1}}\partial_{z}u_n)|_{z=0}.$
 Since $\partial_z\bar{u}\geq c_0$ at $z=0$, $|\partial_z^2(u_{n}-\bar{u})|\leq 1$ in $[0,Y^*]$ by the  bootstrap assumption and $|\partial_z^2(u_{n-1}-\bar{u})|\leq 1$ by the induction assumption, we have, by the boundary data at $z=0$, that for some positive constants $k_4\ll\alpha$ and $\delta_5$  independent of $\varepsilon,$
\begin{align}\label{dzunun}\begin{split}
   \partial_{z}u_n(1-k_4)\leq&(\frac{u_n}{u_{n-1}}\partial_{z}u_n)|_{z=0}\leq   \partial_{z}u_n(1+k_4),\\ (1-k_4)\frac{1}{z+\tilde{c}}\leq&\frac{\partial_{z}u_n}{u_n}= \frac{\partial_{z}u_n}{\int_0^z \partial_{z}u_ndz'+u_n|_{z=0}}\leq (1+k_4)\frac{1}{z+\tilde{c}},\end{split}
\end{align} in $\Omega\cap\{y\leq Y^*\}\cap \{z\leq \delta_5\}$ which imply
\begin{align}\label{1117-3}
   (1-k_4)^2\frac{1}{z+\tilde{c}}\leq  (\frac{u_n}{u_{n-1}}\partial_{z}u_n)|_{z=0}\frac{1}{u_n}\leq  (1+k_4)^2\frac{1}{z+\tilde{c}},
\end{align}in $\Omega\cap\{y\leq Y^*\}\cap \{z\leq \delta_5\}.$
In fact,  the second line in \eqref{dzunun} comes from  the following  calculation. For any fixed $(x,y)\in[0,X]\times[0,Y^*],$ set $b=u_n(x,y,0), a=\partial_zu_n(x,y,0)\geq c_0$ by \eqref{positivelbu}.  Note $\tilde{c}=\frac{b}{a}$.  For small $z$, we have
\begin{align}\label{1117-5}\begin{split}
    \frac{\partial_{z}u_n}{u_n}=& \frac{\partial_{z}u_n}{\int_0^z \partial_{z}u_ndz'+u_n|_{z=0}}
    = \frac{a(1+o(1))}{\int_0^z a(1+o(1))dz'+b}\\
    =&(1+o(1))\frac{1}{\int_0^z (1+o(1))dz'+\tilde{c}}
     =(1+o(1))\frac{1}{z (1+o(1))+\tilde{c}}\\
     =&(1+o(1))^2\frac{1}{z +\tilde{c}}
     =(1+o(1))\frac{1}{z +\tilde{c}}.\end{split}
\end{align}
Second, we note that
 \begin{align*}
\partial_{z}(\frac{u_{n-1}}{u_n} )=&\partial_{z}(\frac{u_{n-1}-u_n}{u_n} )
=\frac{\partial_{z}(u_{n-1}-u_n)}{u_n} -\frac{\partial_{z}u_n(u_{n-1}-u_n)}{u_n^2} \\
=&\frac{1}{u_n}\big(\partial_{z}(u_{n-1}-u_n) -\frac{\partial_{z}u_n(u_{n-1}-u_n)}{u_n}\big).\end{align*}
Then\begin{align}\label{666-1-0205}\begin{split}
|\partial_{z}(\frac{u_{n-1}}{u_n} )|\leq&\frac{1}{u_n}|\partial_{z}(u_{n-1}-u_n) -\frac{\partial_{z}u_n\big((u_{n-1}-u_n)-(u_{n-1}-u_n)|_{z=0}\big)}{u_n}|
\\&+\varepsilon^6  e^{-\frac{A}{x+1}}C\epsilon_0,\end{split}\end{align} by the compatible approximate boundary condition \eqref{z0ep3} at $z=0.$
Set $f=\partial_{z}(u_{n-1}-u_n).$ Then
\begin{align}\label{1119-2}\begin{split}
    &\partial_{z}(u_{n-1}-u_n) -\frac{\partial_{z}u_n\big((u_{n-1}-u_n)-(u_{n-1}-u_n)|_{z=0}\big)}{u_n}
   \\ =&f- (1+o(1))\frac{1}{z+\tilde{c}}\int_0^z fdz'\\
   =&(1+o(1))(f- \frac{1}{z+\tilde{c}}\int_0^z fdz')+o(1)f,\end{split}
\end{align}where we have used \eqref{dzunun}. Then by Corollary \ref{cr},
we have, for some positive constant $\delta_5$ independent of $\varepsilon,$
\begin{align}\label{649}\begin{split}
   &\frac{\partial_z u_n}{u_n}| \partial_{z}(u_{n-1}-u_n)-\frac{\partial_{z}u_n}{u_n}\big((u_{n-1}-u_n)-(u_{n-1}-u_n)|_{z=0}\big)|
  \\ \leq& \frac{\varepsilon e^{-\frac{A}{x+1}}(z+\tilde{c})^{\alpha}}
{1+\frac{\alpha}{4}} \quad \text{in}\quad \Omega\cap\{0\leq y\leq Y^*\}\cap\{z\leq \delta_5\}.\end{split}
\end{align}
For the estimate on $o(1)f,$ by  the boundary condition \eqref{z0ep3}, \eqref{zzf} and integrating with respect to $z$,  we have\begin{align}
  |f|=|\partial_{z}(u_n-u_{n-1})|\leq \varepsilon^7 e^{-\frac{A}{x+1}}C\epsilon_0^2+\frac{2}{1+\alpha}\varepsilon  e^{-\frac{A}{x+1}} (z+\tilde{c})^{1+\alpha}.\end{align}
  Combining \eqref{dzunun}, \eqref{666-1-0205} and  \eqref{649}, we
  complete the proof of the lemma.
  \end{proof}

 Similarly, we can prove the following  lemmas which have been  used in the estimates on $\partial_{z}\nabla_{\eta,\xi}(u_{n}-\bar{u}).$

 \begin{lemma}\label{6969}For some   small positive constant  $\delta_6$ independent of $\varepsilon,$
 \begin{align*}
|\partial_z \bar{u}\partial_z(\frac{\int_0^z\partial_x\bar{u}-\partial_xu_{n-1}dz'}{u_{n-1}})|
\leq  \frac{ \varepsilon^2}{2+\frac{\alpha}{2}} (\frac{1}{x+1})^{\frac{\alpha}{2}}e^{-\frac{A}{x+1}} (z+\epsilon_0)^{\alpha},\quad \Omega\cap \{z\leq \delta_6\}.\end{align*}
 \end{lemma}
\begin{proof} Set $f=\partial_x\bar{u}-\partial_xu_{n-1}$ and $\tilde{c}=\frac{u_{n-1}(x,y,0)}{\partial_zu_{n-1}(x,y,0)}.$
By the compatible boundary data \eqref{z0ep3} and $\partial_zu_B>0$ at $z=0,$ we have
\begin{align}\label{1119-3}\begin{split}
   \tilde{c} =&\frac{ u_{n-1}}{\partial_zu_{n-1}}|_{z=0}= (1+R_{\epsilon_0}) \frac{\bar{u}}{\partial_zu_{n-1}}|_{z=0}\\
    =& (1+R_{\epsilon_0})^2 \frac{\bar{u}}{\partial_z\bar{u}}|_{z=0}
     = (1+R_{\epsilon_0})^2 \frac{u_B}{\partial_zu_B}|_{z=\epsilon_0}
     \\
     =& (1+R_{\epsilon_0})^3 \epsilon_0,\end{split}
\end{align} where $R_{\epsilon_0}$ stands for the terms satisfying $|R_{\epsilon_0}|\leq C\epsilon_0$ which may vary from line to line.
Then by the induction assumption and \eqref{1119-3}, we have, for $f=\partial_{x}u_{n-1}-\partial_{x}\bar{u},$ it holds
\begin{align}\label{0126huahua}\begin{split}
 |\partial_z f|=&|\partial_z\partial_{x}u_{n-1}-\partial_z\partial_{x}\bar{u}|
 \leq \varepsilon^2e^{-\frac{A}{x+1}} (\frac{1}{x+1})^{\frac{\alpha}{2}}(z+{\epsilon_0})^{\alpha}
\\  \leq& \varepsilon^2e^{-\frac{A}{x+1}} (\frac{1}{x+1})^{\frac{\alpha}{2}}(z+\tilde{c}(1+C\epsilon_0))^{\alpha}
 ,\quad \Omega\cap \{\frac{z+{\epsilon_0}}{\sqrt{x+1}}\leq \delta\}.\end{split}
\end{align}
Then similar to the estimate  \eqref{1119-2}, by  Remark \ref{1117-6} and \eqref{1119-3}, for some $\delta_6\ll\delta$  independent of $\varepsilon,$ we have, in $\Omega\cap \{z\leq \delta_6\},$
\begin{align*}
&\frac{\partial_z \bar{u}}{u_{n-1}}| ( \partial_x\bar{u}-\partial_xu_{n-1})-\frac{\partial_zu_{n-1}}{u_{n-1}}
\int_0^z\partial_x\bar{u}-\partial_xu_{n-1}dz'
|\\ =&|(1+o(1))\frac{1}{z+\tilde{c}}(f- (1+o(1))\frac{1}{z+\tilde{c}}\int_0^z fdz')|\\
=&|(1+o(1))^2\frac{1}{z+\tilde{c}}(f- \frac{1}{z+\tilde{c}}\int_0^z fdz')+o(1)\frac{f}{z+\tilde{c}}|
\\
\leq &\frac{ \varepsilon^2}{2+\frac{\alpha}{2}} (\frac{1}{x+1})^{\frac{\alpha}{2}}e^{-\frac{A}{x+1}} (z+\epsilon_0)^{\alpha},\end{align*}
where we have used $\frac{\partial_z u_{n-1}}{u_{n-1}}=(1+o(1))\frac{1}{z+\tilde{c}}$ by \eqref{1117-5},
\begin{align*}
    \frac{\partial_z \bar{u}}{u_{n-1}}=&\frac{(1+o(1))\partial_z u_{n-1}}{u_{n-1}}
    =(1+o(1))\frac{1}{z+\tilde{c}},
\end{align*}and \eqref{0126huahua} to estimate $|o(1)\frac{f}{z+\tilde{c}}|.$
Hence,  in $\Omega\cap \{z\leq \delta_6\},$
\begin{align*}
|\partial_z \bar{u}\partial_z(\frac{\int_0^z\partial_x\bar{u}-\partial_xu_{n-1}dz'}{u_{n-1}})|
=&\frac{\partial_z \bar{u}}{u_{n-1}}| ( \partial_x\bar{u}-\partial_xu_{n-1})-\frac{\partial_zu_{n-1}}{u_{n-1}}
\int_0^z\partial_x\bar{u}-\partial_xu_{n-1}dz'
|\\
\leq & \frac{ \varepsilon^2}{2+\frac{\alpha}{2}} (\frac{1}{x+1})^{\frac{\alpha}{2}}e^{-\frac{A}{x+1}} (z+\epsilon_0)^{\alpha}.\end{align*}
\end{proof}

\begin{lemma}\label{xz1104} It holds that
\begin{align}\begin{split}
 | \partial_z\nabla_{\tau_1}\bar{u}-\partial_z\nabla_{\xi,\eta}\bar{u}|
 \leq \frac{ \varepsilon^2}{2+\frac{5\alpha}{12}} (\frac{1}{x+1})^{\frac{\alpha}{2}}e^{-\frac{A}{x+1}} (z+\epsilon_0)^{\alpha},\quad \Omega\cap \{z+\epsilon_0\leq \varepsilon^6 \}.\end{split}\end{align}
\end{lemma}\begin{proof}
By \eqref{xitau}, \begin{align}\begin{split}
  -\partial_z\nabla_{\tau_1}\bar{u}+\partial_z\nabla_{\xi}\bar{u}=&(\frac{\int_0^z\partial_x\bar{u}dz'}{\bar{u}}-\frac{\int_0^z\partial_x u_{n-1}dz'}{u_{n-1}})\partial_z^2\bar{u}
\\
&+\partial_z(\frac{(u_{n-1}-\bar{u})\int_0^z\partial_x\bar{u}dz'}{u_{n-1}\bar{u}})\partial_z
\bar{u}+\partial_z(\frac{\int_0^z\partial_x\bar{u}-\partial_xu_{n-1}dz'}{u_{n-1}})\partial_z
\bar{u}.\end{split}\end{align}
By \eqref{nnbaru}, we have
\begin{align*}
& |\partial_z^2\bar{u}|\leq C(z+\epsilon_0),\quad |\frac{\int_0^z\partial_x\bar{u}dz'}{u_{n-1}\bar{u}}|\leq C,\quad
 |\partial_z\frac{\int_0^z\partial_x\bar{u}dz'}{u_{n-1}\bar{u}}|\leq \frac{C}{\bar{u}},
 \\&|u_{n-1}-\bar{u}|\leq\varepsilon^6 \phi_{1,1+2\alpha},\quad
 |\partial_zu_{n-1}-\partial_z\bar{u}|\leq \varepsilon^5 \phi_{1,\alpha}.
\end{align*}
Then by  Lemma  \ref{6969}, it holds that
\begin{align*}\begin{split}
 | \partial_z\nabla_{\tau_1}\bar{u}-\partial_z\nabla_{\xi}\bar{u}|
 \leq \varepsilon^3 \phi_{1,\alpha}C+\frac{ \varepsilon^2}{2+\frac{\alpha}{2}} (\frac{1}{x+1})^{\frac{\alpha}{2}}e^{-\frac{A}{x+1}} (z+\epsilon_0)^{\alpha},\quad \Omega\cap \{\sqrt{z+\epsilon_0}\leq \varepsilon^3 \},\end{split}\end{align*}
 where we have used for the small positive constant $\alpha,$
 \begin{align*}
    |\partial_z^2\bar{u}|\leq C(z+\epsilon_0)^{1-\alpha}(z+\epsilon_0)^\alpha
    \leq C\varepsilon^3(z+\epsilon_0)^\alpha\quad
    \text{in}\quad \Omega\cap\{z+\epsilon_0\leq \varepsilon^6 \}.
 \end{align*}
 Note that  for $\varepsilon$ small, $\sqrt{z+\epsilon_0}\leq \varepsilon^3 $ implies $z\leq \delta_6$ where  $\delta_6$ is defined in Lemma  \ref{6969} which is independent of $\varepsilon$.
Similarly, we can derive
\begin{align*}\begin{split}
 | \partial_z\nabla_{\tau_1}\bar{u}-\partial_z\nabla_{\eta}\bar{u}|
 \leq \varepsilon^3 \phi_{1,\alpha}C+\frac{ \varepsilon^2}{2+\frac{\alpha}{2}} (\frac{1}{x+1})^{\frac{\alpha}{2}}e^{-\frac{A}{x+1}} (z+\epsilon_0)^{\alpha},\quad \Omega\cap \{\sqrt{z+\epsilon_0}\leq \varepsilon^3 \}.\end{split}\end{align*}
\end{proof}

\bigbreak

\noindent{\bf Acknowledgements}.
The research of W. Shen is supported by NSFC(Grant 12371208). The research of  Y. Wang is supported by NSFC(Grant 12371236 and Grant 12001383) and the National Key Research $\&$ Development Program(Grant 2024YFA1014900). The research of T. Yang is supported by the General Research Fund of Hong Kong (Project No. 11303521). He would also like to thank the Kuok Group foundation for its generous support. The authors would also like to thank the support by the Research Center for Nonlinear Analysis in The Hong Kong Polytechnic University.

 \end{document}